\documentclass[10pt,a4paper]{article}
\usepackage[a4paper,left=2.5cm,right=2.5cm,top=3cm,bottom=3cm]{geometry}

\usepackage[utf8]{inputenc}
\usepackage[T1]{fontenc}
\usepackage[english]{babel}

\usepackage{amsmath,amssymb,amsfonts,amsthm}
\usepackage{latexsym}
\usepackage{eucal}
\usepackage{bm}
\usepackage{cases}
\usepackage{esint,comment}
\usepackage{dsfont}
\usepackage{algorithm}
\usepackage{algpseudocode}
\numberwithin{equation}{section}

\usepackage{graphicx}
\usepackage{wrapfig}
\usepackage{float}

\usepackage{subfig}

\usepackage{tikz}
\usetikzlibrary{decorations.markings}

\usepackage{pgfplots}
\pgfplotsset{compat=1.17}

\usepackage{xcolor}
\usepackage{verbatim}
\usepackage{comment}
\usepackage{cancel}
\usepackage{enumerate}

\usepackage[colorlinks=true, urlcolor=blue, citecolor=red, linkcolor=blue,    
linktocpage,  pdfpagelabels,  bookmarksnumbered,  bookmarksopen
]{hyperref}

\newtheorem{theorem}{Theorem}[section]

\newtheorem{proposition}[theorem]{Proposition}
\newtheorem{lemma}[theorem]{Lemma}
\newtheorem{corollary}[theorem]{Corollary}
\newtheorem{remark}[theorem]{Remark}
\newtheorem{remarks}[theorem]{Remarks}

\newcommand{\dys}{\displaystyle}

\newcommand{\beq}{\begin{equation}}
\newcommand{\eeq}{\end{equation}}

\newcommand{\R}{\mathbb R}
\newcommand{\N}{\mathbb N}

\newcommand{\eps}{\varepsilon}
\newcommand{\e}{\varepsilon}

\newcommand{\into}{\int_{\Omega}}

\newcommand{\Mcal}{\mathcal M}
\newcommand{\Ecal}{\mathcal E}

\newcommand{\vfi}{\varphi}

\definecolor{purple}{rgb}{0.5,0.0,0.5}

\newcommand{\sideremark}[1]{%
\ifvmode\leavevmode\fi
\vadjust{%
\vbox to0pt{%
\vss
\hbox to 0pt{%
\hskip\hsize\hskip1em
\vbox{%
\hsize2.1cm
\tiny
\raggedright
\pretolerance10000
\noindent #1\hfill
}%
\hss
}%
\vbox to15pt{\vfil}
\vss
}%
}%
}

\title{Optimization of the total tumor population under Gompertz growth}

\author{
Iulia Martina Bulai\thanks{Department of Mathematics "Giuseppe Peano", Via Carlo Alberto 10, 10123 Torino, Italy. Email: \texttt{iuliamartina.bulai@unito.it}}
\and
Francesca Gladiali\thanks{Department of Chemical, Physical, Mathematical and Natural Sciences, Via Vienna 2, 07100 Sassari, Italy. Email: \texttt{fgladiali@uniss.it}}
\and
Benedetta Pellacci\thanks{Dipartimento di Matematica e Fisica, Universit\`a della Campania ``Luigi Vanvitelli'', via A. Lincoln 5, 81100 Caserta, Italy. Email: \texttt{benedetta.pellacci@unicampania.it}}
}

\date{}

\begin{document}

\maketitle

\begin{abstract}
We study optimal control problems for a stationary reaction--diffusion model 
describing the spatial distribution of a tumor cell population with Gompertz 
growth. The control $m(x)$ represents a  treatment term 
acting as a  density-dependent removal rate and it is subject to  $L^{1}-L^{\infty}$ constraints. 
When the intrinsic growth rate is constant, the uniform distribution of the treatment is shown to be the unique minimizer. 
For the maximization problem, 
we prove that every optimal control is of bang-bang type. 
In addition, we show that in the one dimensional case and for sufficiently large diffusion rates, the positivity set of optimal controls is an interval sticking to one  of the extrema of the domain.

Finally, numerical simulations complement the theoretical analysis and explore regimes that are not fully covered by the results proved in the paper. The computations confirm the bang-bang structure of maximizers, and illustrate how
the shape of optimal controls and the associated states are affected by
 spatial heterogeneity in the growth rate, localized admissible treatment regions, and the diffusion coefficient. Moreover,
they reveal a monotone dependence of the optimized total population on the diffusion coefficient: this
is a new phenomenon with respect to the logistic setting.
\end{abstract}



\section{Introduction}
Tumor growth and invasion can be described as spatial processes resulting 
from the interaction between local cell proliferation and dispersal through 
surrounding tissue. Reaction--diffusion equations provide a natural mathematical 
framework to capture this interaction, with diffusion terms modeling cell motility 
and invasion, and reaction terms accounting for population growth and saturation 
effects \cite{Chaplain2008,Murray2003,Swanson2002,Swanson2003}.
Restricting our attention to isotropic diffusion models,  different reaction terms
are  taken into consideration depending on the particular issues of interest: logistic growth has been extensively  employed to represent density-dependent 
proliferation constrained by a carrying capacity (\cite{Murray2003}), while, among 
others,   the one introduced by Gompertz (see 
\cite{gompertz, laird})   has been used, in many cases more efficiently (see \cite{Vaghi2020}),   to reproduce experimentally observed tumor growth dynamics 
(\cite{gerlee, TalkingtonDurrett2015}).
Related mathematical and numerical approaches have also been developed for tumor growth models where the evolution of cell density is described through differential equations and integral formulations. In this direction, several recent works have addressed numerical resolution, growth prediction, treatment effects, and the computation of biologically relevant observables \cite{BulaiDeBonisLauritaSagaria2022, BulaiDeBonisLauritaSagaria2023, BulaiDeBonisLaurita2025AMC, BulaiDeBonisLaurita2025MCS}.
Within this research framework, the model we study  is governed by the following reaction-diffusion problem
\begin{equation}\begin{cases}\label{PP}
u_{t}-d\Delta u=s(x)u \ln \frac K u-m(x) u 
& \text{ in } \Omega\times [0,T)\\
\qquad u>0 & \text{ in } \Omega\times [0,T)\\
\quad\partial _\nu u=0 & \text{ on } \partial \Omega\times [0,T)
\\
u(x,0)=u_{0}  & \text{ in } \Omega\
\end{cases}
\end{equation}
where $u_{0}\geq 0$ is an initial datum, the non-negative initial density of the population,  $d>0$ is the diffusion (or motility) 
coefficient and Neumann boundary conditions are imposed on the boundary of 
$\Omega$ which is supposed to be a bounded  domain  of class $C^{2}$ in $\R^n$, 
($n\geq 1$).
The equation describes the dispersal of a tumor cell population of density 
\(u=u(x,t)\),  evolving in a heterogeneous tissue.  The  
function \(s=s(x)\) denotes an  intrinsic growth rate satisfying 
\begin{equation}\label{ipo:s}
0<\underline{s}\leq s(x)\leq \overline{s},\quad \text{a.e. $x\in \Omega$ with $\underline{s}$, $\overline{s}$ $\in \R^+$},
\end{equation}
and  the constant carrying capacity is given by \(K>0\). 
The Gompertz law is encoded in the non-linear reaction term.

From a control-theoretic viewpoint, therapeutic interventions can be incorporated into
tumor-growth models through spatially dependent control functions that represent
treatment, removal, or therapeutic mechanisms acting on the population
\cite{AllegrettiBulaiLenhartOrru2026, FisterPanetta2000, LenhartWorkman2007,Troltzsch2010}.
In Problem \eqref{PP} this role is played by the control function $m(x)$, which, due to the minus sign, hinders  the growth of $u$.

The controls are typically subjected to both pointwise constraints, reflecting 
local limitations on treatment  intensity,
and integral constraints, accounting for bounded global resource or toxicity bounds.
These requirements naturally lead one to consider a control $m$ varying in the class
\begin{equation}\label{M-set}
\mathcal M:=\{  0\leq m(x)\leq \overline{m}(x), \text{ a.e.  in } \Omega \text{ and }\int_\Omega m(x)= M\}
\end{equation}
where $\overline{m}(x)\in L^\infty(\Omega)$  and $0<M<\int_\Omega \overline{m}(x)\, dx$. Note that the integral constraint implies that $m^{+}\not \equiv 0$. 
\\
The integral bound given by $M$ allows to avoid the possibility to  saturate the entire tissue with the maximum treatment, owing to, the above mentioned, limited global resources or toxicity bounds. This
kind of $L^{1}-L^{\infty}$ constraints are classically adopted in optimization problems (see  \cite{HP}).

As a consequence of the above assumptions, there exists a unique  solution $u(x,t)$ to \eqref{PP} associated with every non-negative initial datum $u_{0}\not \equiv 0$, and $u(x,t)$ converges, as t goes to infinity,
to the unique positive solution to the elliptic problem, (see \cite[Chapter 10]{SV} and Remark \ref{rem-par}),
\begin{equation}\label{P}
\begin{cases}
-d\Delta u=s(x)u \ln \frac K u-m(x) u 
& \text{ in } \Omega\\
\qquad u>0 & \text{ in } \Omega\\
\quad\partial _\nu u=0 & \text{ on } \partial \Omega.
\end{cases}\end{equation}
As a first consequence of the features of the nonlinearity,  this result holds for every $d>0$ independently of the sign of $M$  (see Remark  \ref{rmk:ipom}). 
This marks a first difference with respect to the problem governed by a logistic growth.
Indeed, in this latter model,  if $M<0$ and $m^{+}\not\equiv 0$, 
 a survival threshold arises as a principal eigenvalue  $\lambda(m)$ (see e.g. \cite[Chapter 10]{SV}). Then, a widely studied optimization problem consists 
in the minimization of the cost functional $\lambda(m)$ with respect to the control 
$m$; A variety of results have been obtained in the study of the qualitative properties of optimizers, depending on the different boundary conditions, including in asymptotic regimes.
   (see e.g.  \cite{cc, louyan, hkl, llnp, mapeve1, mnpchal, mapeve2,  pepisc, feve1, feve2, fmpv}).

In our context, a natural optimization problem regards the 
study of the total population, namely the cost functional $J_{d}: \Mcal \to \R$ given by
\begin{equation}\label{eq:Jd}
J_{d}(m):=\into u(x)dx
\end{equation}
where $u$ is the unique  solution of \eqref{P} for a given $m\in \Mcal$. In particular, we will be 
focused on the  following two optimization problems, 
\begin{equation}\label{eq:defmax}
\min_{m\in \Mcal}J_{d}(m)=\min_{m\in \Mcal}\into u(x) , \qquad \max_{m\in \Mcal}J_{d}(m)=\max_{m\in \Mcal}\into u(x) .
\end{equation}
These can be seen as control problems, with $m$ the control, $u$ the associated state, i.e. the solution to \eqref{P}, and $J_{d}$ the cost functional.
We tackle this study  focusing on these different aspects:
our nonlinearity is not of polynomial type, and it has an infinite derivative at zero; 
$m$ and $s$ are decorrelated (see  \cite{deanizha, guoheni}); $\overline{m}$ is not constant and it  can be zero in a subset of $\Omega$ of positive measure. This last condition opens a new situation  even in the logistic problem.

By means of  the direct methods  of the calculus of variations one can prove that these extrema are attained respectively by some weights $m_{*},\, m^{*} \in \Mcal $. 
Then natural questions arise concerning qualitative properties of $m_{*}$ and $m^{*}$.

As a first observation, note that Problem \eqref{P} has a constant solution if and only if $m$ and $s$ are proportional with the proportionality constant given by $\alpha=M/\into s$; in this case the constant solution is given by
\[
U_{\infty}:=Ke^{-\frac{M}{\into s(x) dx}}.
\] 
The constant $U_{\infty}$ is naturally related to the minimum problem at least when $s(x)$
is constant as established by the following result.
\begin{theorem}\label{th:minscost_intro}
Assume $s(x)\equiv s$ is constant in $\Omega$ and 
\begin{equation}\label{ipo:M_intro}
\overline{m}\geq \frac{M}{|\Omega|}.
\end{equation}
Then, the only solution of the minimum problem in \eqref{eq:defmax} is  the constant weight $M/|\Omega|$.
\end{theorem}
Theorem \ref{th:minscost_intro} follows by  applying Theorem \ref{th:minscost} and Proposition \ref{u-costante}: We first exploit the problem solved by $w:=\ln u$ (see also Remark \ref{rem:ln}): this can be studied following the strategy introduced by \cite{Lou}; then, we can transfer back the information to the original problem making use of the convexity property 
of the transformation. In this last step the hypothesis on $s(x)$ is crucial. 
If $s$ is not constant and  
\[
\overline{m}\geq \frac{M}{\into s(x)dx}s(x),
\]
 the control $m:=\frac{M}{\into s(x)dx}s(x)$, although belonging to the class $
 \mathcal M$, is not a minimizer (see Proposition
\ref{pro:minsnocostante}). 
Actually, when $s$ is not constant,
the  minimization problem is also open  for the logistic model,  even 
assuming $\overline{m}$ to be  constant, indeed, for
$f(x,t)=m(x)t-s(x)t^{2}$, different minimizers can be obtained depending on the correlation
between $m$ and $s$ (see  \cite{deanizha, guoheni}).
We have addressed this issue also using numerical simulations, and the results,
included in Figures ~\ref{fig:gompertz_1D_min}, 
\ref{fig:gompertz_1D_min_piecewise_heatmaps_profiles}, 
suggest that the minimizer $m_{*}$ should keep a proportionality dependence on $s$ 
respecting at the same time the measure constraint (see Remark 
\ref{rem:minlogus}). 
Numerical analysis has also been exploited to provide insights into the case 
$|\Omega'|>0$, which remains open for logistic growth, where
\beq\label{eq:omegaprimo}
\Omega':=\left\{x\in \Omega : \overline{m}(x)=0\right\}.
\eeq
The results, included in Figure~\ref{fig:gompertz_1D_i_ii}, suggest that $m_{*}$ saturates 
the constraint $\overline{m}$ in a set of measure depending on $M$ and $\into \overline{m}$
(see Remark \ref{rem:min_s_const}).
\vskip3pt
As for the maximization problem in \eqref{eq:defmax},
in the framework of logistic growth and assuming $\overline{m}\equiv 1$, it has 
been recently proved that this problem 
 is solved by a bang-bang weight (see \cite{yanagida, manaprcpde, FMP}), namely
$m^{*}(x)=\chi_{E}$ where $E$ turns out to be a sub-level set of the so-called switching function $\vfi$ (see \cite{manaprjmpa}) 
which, in our case, is  the unique solution to the problem
 \begin{equation}\label{eq:switch}
\begin{cases}
-d\Delta \vfi+ 2d\frac{\nabla u\cdot \nabla \vfi}{u}=2d\frac{|\nabla u|^{2}}{u^{2}}\vfi+
\vfi\left(2s(x)\ln\frac{K}u-2m(x)-s(x)\right)+ u&\text{in $\Omega$}
\\
\partial_{\nu}\vfi=0 &\text{on $\partial\Omega$,}
\end{cases}
\end{equation}
with $u$ is the solution to \eqref{P} associated with the control $m$.
Our result is related to the bang-bang property of the maximizer and it  is contained in the following Theorem.
 \begin{theorem}\label{thm:mbang}
Let $\Omega\subset \R^n$  be a bounded open connected domain with a $C^2$ boundary.
Let $m^*$ be a solution to the maximum problem in \eqref{eq:defmax}, and $u^*$, $\vfi^*$ be the corresponding state and switching function. Then there exists a measurable subset  $\omega^*\subset \Omega\setminus\Omega'$,  where  $\Omega'$ is introduced in \eqref{eq:omegaprimo},  such that 
\[m^*=\overline{m}(x)\chi_{\omega^*} \]
and $\int_{\omega^*}\overline{m}=M$. Moreover, 
\[
\text{either }\; \omega^*=\{x\in \Omega\setminus \Omega': \varphi^*< c^*\}
\quad \text{or }\;  \omega^*=\{x\in \Omega\setminus \Omega': \varphi^*\leq c^*\}
\]
 where $\dys c^*=  \sup_{\omega^*} \varphi^*(x)$.
\end{theorem}
The proof follows  the strategy introduced in \cite{manaprcpde,FMP}: we will include the relevant details regarding our setting, mainly due to the assumption on $\overline{m}$.
Theorem \ref{thm:mbang} shows that the worst distribution of treatment is of 
``generalized'' bang-bang type, namely it concentrates  the maximum amount of treatment,
which is not constant in this context,  in a subset $\omega^*$ of $\Omega$.

Our analysis is deepened in  dimension one, as the following result shows.
\begin{theorem}\label{teo-concentration}
Let us assume that $n=1$, $\Omega=(0,1)$,  $s(x)=s\in (0,+\infty)$ and $\overline{m}(x)\equiv 
1$. There exists $\hat d>0$ such that for every $d>\hat d$ any optimal weight $m^*_d$ for 
$J_{d}(m)$ is equal, almost everywhere in $\Omega$, either to $ m_1=\chi_{(0,M)}$ or $ 
m_2=\chi_{(1-M,M)}$.
\end{theorem}
This result relies on a careful analysis of the corresponding optimization problem for the first 
order expansion of the cost functional with respect to $d$ tending to plus infinity. In the spirit
of \cite[Theorem 3]{manaprjmpa}, this ``first order optimization problem'' is recast as a 
``min-min'' problem--here the assumption on $s$ and $\overline{m}$ are 
crucial--whose 
solution can be completely characterized in the one dimensional case. These 
qualitative  properties are transferred back for $d$ sufficiently large thanks to good convergence.
\\
This Theorem can be exploited to obtain some new results in the case of  
$\overline{m}\not \equiv 1$ (see for more details Corollary \ref{cor:teoest} and Remark \ref{rem:teoest}).
This result is confined to the one dimensional case due to different reasons: first the 
minimizers of the ``first order optimization problem'' are not characterized in higher 
dimensions. With this respect, we believe that  qualitative asymptotical results could be 
obtained, as in the logistic growth models, in the spirit of \cite{mapeve1, mapeve2, femapeve}  or   by exploiting special symmetry properties of the domain. In particular, when \(\Omega\) is a square, symmetry may favor optimal sets concentrated near corners. This type of behavior is consistent with numerical evidence for related logistic total-population optimization problems in higher-dimensional domains; see, for instance, \cite{ding-finotti-lenhart, manaprjmpa,  mabalet, KaoMohammadi2022}. The two-dimensional computation reported in Figure~\ref{fig:gompertz_2D_max} points in the same direction, although a rigorous characterization of such configurations remains open.

Let us finally notice that the contribution in \cite{heokim,manaprjmpa,mabalet} reveal a strong influence of the 
diffusion coefficient on this study. This actually leads to new phenomena with respect to the 
more studied optimization problem of the survival threshold.
In particular, in these contributions it is shown that concentrating the distribution
of resources increases the total mass of the population if $d$ is large, while fragmenting it
is  a better strategy   for $d$ is small.
This latter property is obtained starting from this crucial information, already proved by \cite{Lou},
\[
\lim_{d\to 0^+} J_{d}(m)=\lim_{d\to \infty} J_{d}(m) .
\]
Here, for $s$ constant, a simple application of Jensen inequality yields
\[
\lim_{d\to 0^+} J_{d}(m)\geq \lim_{d\to \infty} J_{d}(m), \; \text{and}\;
\lim_{d\to 0^+} J_{d}(m)= \lim_{d\to \infty} J_{d}(m)\; \Leftrightarrow \; m \,\text{ is constant.}
\]
On the other hand, it would be sufficient to know that
$d\mapsto J_d$ is increasing for $d$ sufficiently small to produce a fragmentation effect in the one dimensional case. This information is not only not available at the moment,  but it seems to be not expected in our setting. 
The numerical simulations included in Figure \ref{fig:s_and_J_max_1D} show a monotone 
dependence of $J_{d}$ with respect to $d$ which could allow to exclude fragmentation for every
diffusion rate.
In this respect, we believe that a very interesting open problem could be the asymptotical analysis of the singular limit of $J_{d}$ as $d\to 0^{+}$.
 \vskip5pt
The paper is organized as follows. In Section~\ref{se:esistenza} we introduce the stationary Gompertz problem, prove existence and uniqueness of positive solutions, and establish the existence of optimal controls for the minimization and maximization problems. Section~\ref{se:lin_eq} is devoted to the linearized equation, the differentiability of the control-to-state map, and the adjoint formulation used to derive the optimality conditions. In Section~\ref{se:bang_bang_max} we prove the ``generalized'' bang-bang property of maximizers. Section~\ref{se:miniprop} discusses qualitative properties of minimizers, including the case of constant growth rate and the role of spatial heterogeneity. In Section~\ref{se:asym_d} we study the asymptotic behavior of the state, adjoint state, and switching function as the diffusion coefficient tends to zero or to infinity. These asymptotic expansions are then used in 
Section~\ref{se:thm13} to prove the one-dimensional large-diffusion characterization of maximizers. Finally, Section~\ref{se:num_results} presents numerical simulations illustrating the theoretical results and exploring regimes not fully covered by the analysis.
\vskip5pt
{\bf Notation.}
In the following we will use the following notations.
\begin{itemize}
\item
$\|\cdot \|_{p}$ denotes the $L^{p}(\Omega)$ norm.
\item
For any measurable function $v$, let $v(x)=v^+(x)-v^-(x)$ where $v^+(x):=\max\{0,v(x)\}$ and $v^-:=\max\{0,-v(x)\}$.
\item
For any measurable set $\omega$,  $\chi_\omega$ denotes the characteristic function associated with $\omega$; namely $\chi_\omega(x)=1$ iff $x\in \omega$ and $\chi_\omega(x)=0$ otherwise.
\item
\(u\) will denote the unique positive
solution of \eqref{P}. When the dependence on the diffusion coefficient is
emphasized, we will write \(u_d\). The corresponding adjoint state will be denoted by \(p\), or by \(p_d\) when the dependence on \(d\) is
made explicit.
\item
\(m^*\) will denote a maximizer of \(J_{d}\) over \(\mathcal M\)
(see \eqref{M-set}), and  \(m_*\) a
minimizer. The corresponding states and adjoint states will be denoted 
respectively by $ u^*,\ p^*,$ and $ u_*,\ p_* .$

\end{itemize}

\section{The elliptic problem}\label{se:esistenza}
Let us introduce the Carath\'eodory function
\begin{equation}\label{f}
f(x,u):=\begin{cases} 
s(x)u \ln \frac K u-m(x) u & \text{ for  } u>0\\
0 & \text{ for  } u\leq 0
\end{cases}
\end{equation}
so that every positive solution to the problem
\begin{equation}\label{Pf}
\begin{cases}
-d\Delta u=f(x,u)
& \text{ in } \Omega\\
\quad\partial _\nu u=0 & \text{ on } \partial \Omega
\end{cases}\end{equation}
is a  solution to \eqref{P}.

Note that \eqref{P} is variational, so that every solution  is a critical point of the functional $\Ecal: H^{1}(\Omega)\mapsto \R$ defined by
\begin{equation}\label{eq:defE}
\Ecal(u):=\dfrac{d}2\into |\nabla u|^{2}-\into F(x,u)
\end{equation}
where $F(x,t):=\int_{0}^{t}f(x,z)dz$ is  given by
\[
F(x,t)=\frac{t^{2}}2\Big(\frac{s(x)}2-m(x)\Big)+\frac{s(x)}2t^{2}\ln\frac{K}t   \ \ \hbox{ for }t>0.
\]
Let us first prove that  there always  exists a unique positive solution.
\begin{proposition}\label{existence}
Assume \eqref{ipo:s}. Then,
for every $d>0$ and for every $m\in \mathcal M$, problem \eqref{P} always admits a unique positive weak solution $u\in C^{1,\alpha}(\Omega)$, for every $\alpha\in(0,1)$. Moreover,  
$u$ is a global minimum of the action functional $\Ecal$ and the following uniform bounds hold.
\begin{equation}\label{eq:bound}
\bar \e\leq u(x)< K\  \text{ in $\Omega$,   with $\bar \e=Ke^{-\frac{m_1}{\underline{s}}}$ and $m_1=\sup_\Omega \overline{m}(x)$. }
\end{equation}
\end{proposition}

\begin{proof}
The existence of a positive  weak solution follows from a classical  iteration method when we prove the existence of a sub-solution and a super-solution to \eqref{P}.
First we prove that the carrying capacity $K$ is a  super-solution to  \eqref{P} for every $d$ and for every $m(x)\in \mathcal M$. Indeed  
\[-d \Delta K-f(x,K)=m(x)K\geq 0.\]
Next we show that $0<\bar \e<K$  is a sub-solution to \eqref{P}. 
Indeed 
\[-d \Delta \bar \e-f(x,\bar \e)=\bar \e \left(-s(x) \ln \frac {K}{\bar \e} +m(x)\right)\leq \bar \e\left(-\underline{s} \ln \frac {K}{\bar \e} +m_1\right)=0\]
by the definition of $\bar \e$. Here we show that any solution to \eqref{P} must satisfy $\bar \e\leq u(x)< K$ in $\Omega$.  
\\
The function $(u-K)$ solves
\begin{equation}\label{u-k}
-d\Delta (u-K)=u s(x)\ln \frac K u-m(x) u \ \ \mbox{ in $\Omega$}
\end{equation}
with Neumann boundary conditions. We multiply by $(u-K)^+$ and integrate over $\Omega$, getting
\[d\int_{\Omega}|\nabla (u-K)^+|^2 =\int_{\Omega} u(x)(u-K)^+\left( s(x)\ln \frac K u-m(x) \right)\leq 0\]
and this shows that either $u\leq K$ in $\Omega$ or $u$ is constant on $\Omega$.\\
The function $(\bar \e-u)$ solves 
$-d\Delta (\bar \e-u)=-u s(x)\ln \frac K u+m(x) u$ in $\Omega$ with Neumann boundary conditions. 
We multiply by $(\bar \e-u)^+$ and integrate over $\Omega$, getting
\[d\int_{\Omega}|\nabla (\bar \e-u)^+|^2 =\int_{\Omega} u(x)(\bar \e-u)^+\left( m(x)-s(x)\ln \frac K u \right)\leq 0\]
since when $u\leq \bar \e=Ke^{-\frac {m_1}{\underline{s}}}$ we have $m(x)-s(x)\ln \frac K u\leq 0$.
 This shows that either $u\geq \bar \e$ in $\Omega$ or $u$ is constant on $\Omega$.  Next, if $u$ is constant on $\Omega$, integrating the equation 
 in \eqref{P}, we immediately obtain that $u=Ke^{-\into m/\into s}$, so that
 $\bar \e\le u\le K$ as well.
Finally we show that $u<K$ in $\Omega$. The function $u-K$ satisfies \eqref{u-k}. 
Since the function $g(t)=t\ln \frac Kt$ is $C^1$ in $[\bar \e, K]$ there exists $c(x)$ such that  $u\ln \frac K{u}-K\ln \frac KK=c(x)(u-K)$.  Then
\[
\begin{cases}-d\Delta (u-K)+m(x)(u-K)-s(x)c(x)(u-K)=-m(x)K\le 0 &\text{ in } \Omega\\
u-K\leq 0 &\text{ in }  \Omega.
\end{cases}\]
The Strong Maximum Principle then implies that either  $u-K\equiv 0$ in $\Omega$ or $K-u>0$ in $\Omega$.  When $m(x)\neq 0$ in $L^\infty$ then the equation gives $0=-m(x)K$ which is a contradicton.\\
Next we show that the solution is unique. 
Let us first observe that  \eqref{eq:bound} implies that $f(x,u)\in L^\infty(\Omega)$ and then, by regularity theory $u\in C^{1,\alpha}(\Omega)$ for every $\alpha\in (0,1)$.

Now, assume  by contradiction there exists two positive solutions $u_1,u_2$ of \eqref{P}.  Then $u_i\geq \bar \e>0$ in $\Omega$ for $i=1,2$. 
Following \cite{Berestycki-Hamel-Roques} we define
\[\gamma^*:=\sup \{\gamma\in \R^+: u_2>\gamma u_1 \text{ in } \Omega\}.\]
Assume $0<\gamma^*<1$ and define $z:=u_2-\gamma^* u_1$. Then $z\geq 0$ and by the definition of $\gamma^*$ there exists a sequence of points $x_n\in \Omega$ such that $z(x_n)\to 0$ as $n\to \infty$. Up to a subsequence, $x_n\to x_0\in \bar \Omega$ and $z(x_0)=0$.  
Moreover $z$ satisfies
\[-d\Delta z=-d\Delta u_2+\gamma^*d \Delta u_1=s(x)\left(u_2\ln \frac K{u_2}-\gamma^* u_1\ln \frac K{u_1}\right)-m(x)\left(u_2-\gamma^* u_1\right)
\]
and, since $\gamma^*<1$ 
\[
-d\Delta z+m(x) z>s(x)\left(g(u_2)-g(\gamma^*u_1)\right)
\]
where $g(s)=s\ln \frac Ks$. Since $g(s)$ is Lispchitz continuous in  $[\bar \e,K]$ then there exists $L$ such that 
\begin{equation}\label{eq:a}
-d\Delta z+m(x) z+L s(x)z> 0.
\end{equation}
First we show that $x_0\in \Omega$.  Assume not, then $x_0\in \partial \Omega$ and, since $\Omega$ is smooth we can apply the Hopf Lemma at the function $z$ in $x_0$ getting that $\partial_\nu z(x_0)<0$. But this is impossible since $\partial_\nu z=\partial _\nu u_2-\gamma^* \partial _\nu u_1=0$. Then $x_0\in \Omega$. 
Further $z\geq 0$ in $\Omega$, $z(x_0)=0$ which, together with \eqref{eq:a} imply $z(x)\equiv 0$ in $\Omega$ by the strong maximum principle. But this is impossibile due to the strict inequality in \eqref{eq:a}. Then $\gamma^*\geq 1$ and $u_2\geq u_1$. Interchanging the role of $u_2$ and $u_1$ one can prove also that $u_1\geq u_2$ so that $u_1=u_2$. 

\end{proof}
\begin{remarks}\label{rmk:ipom}
Let us make a couple of remarks concerning the above proposition.
\begin{enumerate}
\item
Note that the result holds for every $m$ such that $-  m_0\leq m\leq \overline{m}(x)$, with $m_{0}\in \R^{+}$,   choosing as upper bound $Ke^{ \frac{m_0}{\underline{s}}}$ instead of $K$. 
\item
Let us observe that the proof of the uniqueness property can be handled arguing by contradiction (see \cite{bros}), supposing that there exist $u_{1}$ and $u_{2}$ two positive solution and  testing the equation of $u_{1}$ with with $u_{2}^{2}/u_{1}$  and the equation solved by  $u_{2}$ with $u_{2}$.
\end{enumerate}
\end{remarks}

Let us establish  the conditions under which there exists  
a constant solution.

\begin{proposition}\label{u-costante}
Let $m\in \Mcal$ and assume \eqref{ipo:s}. Then, Problem \eqref{P}
has a constant solution if and only if $m(x)=\alpha s(x)$ for  $\alpha=\frac{M}{\int_{\Omega}s(x)}$ and the only constant solution is given by  
\begin{equation}\label{eq:Uinfty}
U_\infty: =Ke^{-\frac{M}{\int_{\Omega}s(x)}}.
\end{equation}

\end{proposition}
Let us observe that the notation $U_{\infty}$ is motivated by the fact that $U_{\infty}$ turns out to be the limit of the  solution to \eqref{P}, as $d\to +\infty$ (see Proposition \ref{pro:dinfinito} ahead).
\begin{proof}
If  $m(x)=\alpha s(x)$ for $\alpha$ given above, then it is very easy to see that  Problem \eqref{P} has  a constant solution $u$, and by uniqueness  $u\equiv U_\infty\equiv Ke^{-\alpha}$.
Let us now suppose that the unique solution $u  $  is  constant. 
Multiplying the equation in \eqref{P} by $1$ and integrating by parts, one immediately derives that there exists $a\in \R^+$ such that $m(x)=as(x)$, 
and the only possible $a$ is given by $\alpha=M/\int_\Omega s(x)$.
\end{proof}
\begin{remark}
The above result implies that,  if $s(x)$ is constant   Problem \eqref{P}
has a constant solution if and only if $m$ is constant.  
\end{remark}
\begin{remark}\label{rem-par}
Proposition \ref{existence} yields the global  existence of the solution $U$ of 
Problem \ref{PP} and the convergence to $u$ as $t\to+\infty$ for every $u_{0}\in 
C^{0}(\overline{\Omega})$ such that $u_{0}>0$ in $\overline{\Omega}$. Indeed, 
this can be obtained by observing that every positive constant $\eps<\overline{\eps}$ is a 
 positive sub-solution to \eqref{P}, so that we can fix 
$\eps_{0}<\min\{\overline{\eps}, \min_{\overline{\Omega}} u_{0}\}$ as a positive time-independent sub-solution to \eqref{PP}.
By arguing in an analogous way, we can construct a time independent super-solution,  and obtain the existence of the 
positive globally defined solution to problem \eqref{PP} (see e.g. \cite{SV}). 
In addition the positive solution to \eqref{PP}  converges to $u$ the 
unique  positive solution to \eqref{P} (see e.g. \cite[Theorem 10.22 ]{SV}).
If $u_{0}\geq 0$ and $u_{0}\not \equiv 0$, we can exploit \cite[Theorem 3]{aroser}  to deduce that there exists $t_{0},$ and  $g_{0}>0$  such that $u(x,t_{0})\geq g_{0}>0$. This is sufficient to follow the argument above to construct a positive, time independent sub-solution and obtain the desired asymptotic behavior.
\end{remark}

In view of proposition \ref{existence}  the functional $J_{d}: \Mcal\mapsto \R$ introduced in \eqref{eq:Jd}, 
where $u$ is the unique  solution to \eqref{P},  is well defined.
The rest of the paper is devoted to the study of the optimization problems 
\eqref{eq:defmax}.
Let us first show that there exist optimal controls, namely,   $m^*$ and $m_{*}$ 
maximizer and  minimizer respectively.

\begin{proposition}\label{sup-attained}
For every $d>0$, 
there exist optimal controls $m^*,m_* \in \mathcal M$.
\end{proposition}
\begin{proof}
The proof relies on a straightforward application of global minimization (or maximization) arguments, exploiting the bound \eqref{eq:bound} and taking into account  that the class 
 $\Mcal$ is closed with respect to the weak-$*$ topology in $L^\infty$.
\end{proof}
\section{The linearized equation}\label{se:lin_eq} 
Let us start this section, by proving some properties
 concerning the linearized problem, which will be repeatedly used in the rest of the paper.

First of all, let us introduce the linearized operator:  let  $u$ be the solution to \eqref{P},  we denote by
\begin{equation}\label{lin}
L_u\psi :=-d\Delta \psi -s(x)\ln \frac K{u} \psi+s(x)\psi +m(x)\psi
\end{equation}
the linearized operator at $u.$ 
The following result will be crucial in our study.
\begin{proposition}\label{prop-lin-inv}
Let $t(x)\in L^\infty(\Omega)$ be such that $\inf_{\Omega}t(x)=\underline{t}>0$ in $\Omega$.  Let $u$ be the unique positive solution to \eqref{P}. Then, for every $d>0$, the linear operator 
\begin{equation}\label{operator-T}
T_u\psi :=-d\Delta \psi -s(x)\ln \frac K{u} \psi +m(x)\psi+t(x)\psi\end{equation}
with Neumann boundary conditions is invertible. Moreover,   the first eigenvalue $\lambda_{1}(T_{u})$ is positive,  and satisfies
\[\lambda_1(T_u):= \inf_{\psi\in H^1(\Omega), \psi\neq 0}
\frac{d\int_\Omega |\nabla \psi|^2- s(x)\ln \frac K{u} \psi^2 +m(x)\psi ^2 +t(x)\psi^2}{\int_\Omega \psi^2}\geq \underline{t}>0.
\]
\end{proposition}
\begin{proof}
Let us consider first the linear operator 
\[A_u\psi:=-d\Delta \psi-s(x)\ln \frac K{u} \psi +m(x)\psi\]
with Neumann boundary conditions. The first eigenvalue for $A_u$ is given by
\[\lambda_1(A_u)
:= \inf_{\psi\in H^1(\Omega), \psi\neq 0}
\frac{d\int_\Omega |\nabla \psi|^2- s(x)\ln \frac K{u} \psi^2 +m(x)\psi ^2}{\int_\Omega \psi^2}=0
\]
since $u$ is a positive solution to \eqref{P}, and hence a first positive eigenfunction for $A_u$.
Next, we observe that the first eigenvalue for $T_u$ satisfies 
\[
\begin{split}
\lambda_1(T_u)&\geq \inf_{\psi\in H^1(\Omega), \psi\neq 0}
\frac{d\int_\Omega |\nabla \psi|^2- s(x)\ln \frac K{u} \psi^2 +m(x)\psi ^2 }{\int_\Omega \psi^2}+\inf_{\psi\in H^1(\Omega), \psi\neq 0}
\frac{\int_\Omega t(x)\psi ^2 }{\int_\Omega \psi^2}
\\
&\ge \lambda_1(A_u)+\inf_{\psi\in H^1(\Omega), \psi\neq 0}
\frac{\underline{t}\int_\Omega \psi ^2 }{\int_\Omega \psi^2}=\underline{t}.
\end{split}
\]
\end{proof}
 \begin{remark}\label{rem:ln}
Note that $u$ is a solution to \eqref{P} if and only if $v:=\ln u$ is a solution to the problem 
\begin{equation} 
\begin{cases}
-d\Delta v-d|\nabla v|^{2}=-m(x)+s(x)(\ln K-v)  &\text{in } \Omega
\\
\partial_{\nu}v=0 &\text{on } \partial \Omega,
\end{cases}
\end{equation}
and the optimization problems in \eqref{eq:defmax} become 
\[
\min_{{m\in \Mcal}}\into e^{v}, \qquad \max_{{m\in \Mcal}}\into e^{v}.
\]
This approach has been recently used in \cite{FMP} to show the bang-bang property of maximizer together with regularity property of its positivity set. 
Our context is not exactly included in their setting because of the spatial dependence of the upper bound $\overline{m}$, the opposite sign in the control,
and as our  nonlinearity  does not fit the hypotheses on $Q$ (see \cite[Section 1.1.2.2]{FMP}).
Here, we work directly on the semilinear problem, showing the  bang-bang property of the maximizer following the approach in \cite{manaprcpde}. 
The study of the regularity of the free boundary, possibly using the  approach in \cite{FMP}, remains, in our context,  an interesting  open problem.
\end{remark}

The proof of Theorem \ref{thm:mbang}  is based on the study of the first and second derivatives of $J_{d}(m)$ at the global maximum point $m^{*}$.
First of all,  we  obtain a Taylor expansion for the solution $u$  associated with a fixed $m$.
In order to do this, we have to take into account   the definition of  admissible directions as introduced in \cite[Section 7]{HP} (see also \cite{manaprjmpa}).
An easy adaptation of \cite[lemma 7.2.24]{HP} yields the following characterization.

In our assumptions, the function $\overline m(x)$ may vanish somewhere in $\Omega$. This corresponds to assuming that there are regions inside  the region $\Omega$  that cannot be reached by the treatment $m$,  or that are preferable to avoid.  In the following analysis, we must take these regions into account.

Here, the dependence on $d$ will be omitted.
\begin{lemma}\label{le:gadm}
Let $\Omega'$ be defined in \eqref{eq:omegaprimo}.
A function $g$ is an admissible perturbation if and only if $g$ is  such that
\begin{itemize}
\item[i)] $g\in L^\infty(\Omega),\, \int _{\Omega}g=0$, $g\equiv 0$ on the set $\Omega'$,
\item[ii)]  $\| \chi_{Q^0_n} g^-\|_{\infty}\to 0$ as $n\to \infty$, where $Q^0_n:=\{x\in \Omega\setminus \Omega'  : m(x)\leq \frac 1n \}$ 
\item[iii)]  $\| \chi_{Q^{\overline{m}}_n} g^+\|_{\infty}\to 0$ as $n\to \infty$, where $Q^{\overline{m}}_n:=\{x\in \Omega\setminus \Omega'  : m(x)\geq \overline{m} -\frac 1n \}$.
\end{itemize}
\end{lemma}
In $\Omega'$ every $m\in \mathcal M$ should coincide with $0=\overline m(x)$, due to the definition of $\mathcal M$,  and any admissible perturbation $g$ must be $0$.  Observe that $|\Omega'|<|\Omega|$.

We can now show that for any $d>0$ the map $m\in \mathcal M \ \mapsto 
u:=u(m)$ is twice Gateaux-differentiable in the direction $g$ as stated in the 
following theorem. 

\begin{theorem}\label{teo-expansion}
Let $m\in \mathcal M$ and $g$ be an admissible perturbation for $m$. 
Let $u_\e$ be a solution to \eqref{P} corresponding to $m_\e(x)=m(x)+\e g$.
Then 
\[
u_\e=u+\e \dot{u}+\frac12\e^2 \ddot{u}+o(\e^2),\quad \text{for $\e\to 0$}
\] 
with respect to the  $H^1$ convergence,  where  $ \dot{u}$ is the unique solution to 
\begin{equation}\label{u1}
\begin{cases}
-d\Delta \dot{u}-s(x)\dot{u}\ln \frac K{u}+m(x) \dot{u}+s(x)\dot{u}=-gu 
& \text{ in } \Omega\\
\partial _\nu \dot{u}=0 & \text{ on } \partial \Omega
\end{cases}\end{equation}
and $\ddot{u}$ is the unique solution to
\begin{equation}\label{u2}
\begin{cases}
-d\Delta \ddot{u}-s(x)\ddot{u}\ln \frac K{u}+m(x) \ddot{u}+s(x)\ddot{u}=
-2g\dot u-  s(x)\frac{\dot u ^2}{u}
& \text{ in } \Omega\\
\partial _\nu \ddot{u}=0 & \text{ on } \partial \Omega.
\end{cases}\end{equation}
\end{theorem}
The derivative $\dot{u}$ is usually called the sensitivity (see \cite{LenhartWorkman2007}, \cite{ding-finotti-lenhart}).
\begin{proof}
Since $m_\e\in \mathcal M$ for every $\e$ small enough, Proposition \eqref{existence} yields the existence of a unique positive solution $u_{\eps}$ of \eqref{P}. Taking $u_{\eps}$ as test function  and exploiting \eqref{eq:bound},
one has
\begin{equation}\label{eq:1-3}
d\int |\nabla u_\e|^2 +\into m_\e(x)u_\e^2=\into s(x) (u_\e^2) \ln \frac K{u_\e}  \leq C 
\end{equation}
so that, up to a subsequence as $\e \to 0$, $u_{\eps}\to u$, weakly in $H^1(\Omega)$, strongly in $L^2(\Omega)$ and  a.e. in $\Omega$.
In addition, passing to the limit in \eqref{eq:1-3}, taking into account that
$u$ is the unique positive solution to \eqref{P}, 
 and using the weak lower semicontinuity of the $L^{2}$ norm of the gradient, we deduce
\[\begin{split}
d\|\nabla u\|_{2}^{2}&\leq \liminf d\|\nabla u_{\e}\|_{2}^{2}\leq \limsup d\|\nabla u_{\e}\|^{2}=\into\left(s(x)u^{2}\ln\frac{K}{u}-m(x)u^{2}\right)
\\
&=d\|\nabla u\|_{2}^{2}.
\end{split}\]
So that, $u_{\e}\to u$ strongly in $H^{1}(\Omega)$.
\\
{\bf Step 1.} In this step we prove that the sequence $w_\e=\frac{u_\e-u}\e$  
converges to $\dot u$ strongly in $H^{1}(\Omega)$.
\\
Let us first observe that $w_{\e}$  satisfies 
\[-d \Delta w_\e+m(x)w_\e-s(x)\frac 1\e\left[u_\e\ln \frac K{u_\e}-u\ln \frac K{u}\right]=-g(x)u_\e.\]
that we rewrite as
\[
-d \Delta w_\e+m(x)w_\e-s(x)\ln \frac K{u} w_\e-s(x)
\frac {u_\e}\e\left[\ln \frac K{u_\e}-\ln \frac K{u}\right]=-g(x)u_\e .
\]
The mean value theorem yields
\begin{equation}\label{diff-log}
\frac {1}\e\left[\ln \frac K{u_\e}-\ln \frac K{u}\right]=-\left(\int_0^1 \frac{1}{tu_\e+(1-t)u}dt \ \right) w_\e=-t_\e(x)w_\e\end{equation}
with $\frac {1}{\bar \e}\geq t_\e(x)\geq \frac {1}{K}>0$ and $t_\e(x)\to \frac 1{u(x)}$ pointwise.
Then  $w_\e$ satisfies
 \begin{equation} \label{w-epsilon}
-d \Delta w_\e+m(x)w_\e-s(x)\ln \frac K{u} w_\e+s(x)u_\e t_\e(x)w_\e=-g(x)u_\e
\end{equation}
with $s(x)u_\e t_\e(x)\geq \frac{\underline{s} \bar \e} K$.
By Proposition \ref{prop-lin-inv}, we have
 \[
\begin{split}
 \frac{\underline{s} \bar \e} K\int_\Omega w_\e^2 &\leq  
d\int_\Omega |\nabla w_\e|^2- \int_{\Omega}w_\e^2 \Big(s(x)\ln \frac K{u} -m(x)  -s(x)u_\e t_\e(x)\Big)\\
&=
-\int_\Omega g(x) u_\e w_\e.
\end{split}\]
Then,   the $L^{2}$ boundedness of  $w_{\e}$ follows by applying
H\"older  inequality and taking into account that $g(x)$ and $u_\e(x)$ are uniformly bounded.
Then \eqref{w-epsilon} implies that also the $H^1(\Omega)$ norm of $w_\e$ is 
bounded; so  $w_\e\rightharpoonup  \dot u$ weakly in $H^1$, strongly in $L^2$ 
and   almost everywhere in $\Omega$.  
In view of \eqref{eq:bound} one deduces that  $t_\e(x)\to \frac 1{u}$ almost everywhere in $\Omega$.
Passing to the limit in \eqref{w-epsilon} we obtain that  $\dot u$ solves  \eqref{u1} and  Proposition \ref{prop-lin-inv} imply  that $\dot u$ is unique and given by
$\dot u:=(L_{u})^{-1}(-gu)$ and,  $\dot u\neq 0$ for any $g\neq 0$.
The $H^{1}$ convergence can be obtained by using \eqref{w-epsilon}, \eqref{u1} 
and  arguing as before.  Finally, we deduce that $w_\e$ is uniformly bounded 
in $L^\infty(\Omega)$,  by observing  that $g(x)$,  $u_\e$, and all the 
coefficients in \eqref{w-epsilon}  are uniformly bounded in $L^\infty(\Omega)$ and 
by exploiting Proposition \ref{prop-lin-inv}.
\\
{\bf Step 2.} In this step we obtain the second order expansion of $u$ with respect to $m$.
\\
Using \eqref{w-epsilon} we have that the function $z_\e=\frac{w_\e-\dot{u}}\e$  satisfies
\beq \label{equation-z}
-d \Delta z_\e+m(x)z_\e-s(x)z_\e \ln \frac K{u}+s(x)z_\e=-g(x)w_\e+s(x)\frac {w_\e} {\e} \left( 1- u_\e t_\e(x) \right).
\eeq
Next, recalling that $u_{\e}=\dot u+\e w_{\e}$, \eqref{diff-log} yields
\[
\begin{split}
\frac 1\e
\left(1-t_\e(x)u_\e\right)&=\frac 1{\e} \left(1-\int_0^1\frac {u_\e}{tu_\e+(1-t)u}dt\right)=w_\e\int_0^1\frac {1-t}{tu_\e+(1-t)u}dt
\\
&=w_\e \tilde t_\e(x)
\end{split}
\]
where $\tilde t_\e(x)\to \frac 1{u}\int_0^1t dt=\frac 1{2u}$.  
In view of  Proposition \ref{prop-lin-inv}, we have
\[
\begin{split}
\underline{s}\int_\Omega z_\e^2
\leq
& d \int_\Omega |\nabla z_\e|^2- s(x)\ln \frac K{u} z_\e^2 +m(x)z_\e^2 +s(x)z_\e^2=\\
&-\int_\Omega g(x) w_\e z_\e+\int_\Omega s(x)w_\e^2 z_\e 
\tilde t_\e(x).
\end{split}\]
Note that $g\in L^{\infty}(\Omega)$,  $w_\e$ is uniformly bounded in $L^{\infty}(\Omega)$, by Step 1, and 
 $\tilde t_\e(x)\leq \frac C{\bar \e} $, so that   H\"older inequality yields the
 $L^{2}$ boundedness of $z_{\e}$. Then, we deduce that $z_{\e}$ is bounded in
  $H^1(\Omega)$ taking into consideration \eqref{equation-z}.
Then, up to a subsequence, $z_\e \rightharpoonup \frac 12 \ddot u$ weakly in $H^1(\Omega)$, strongly in $L^2(\Omega)$ and almost everywhere in $\Omega$.
The conclusion follows arguing as in Step 1, also exploiting   the positiveness of the 
first eigenvalue of $L_{u}$ in proposition \ref{prop-lin-inv} to obtain  the 
uniqueness of  $\ddot u$.
\end{proof}

As a consequence of the previous results we deduce that the functional $J(m)$ is twice  Gateaux-differentiable at $m$ in the direction $g$, for any $d>0$, as stated in the following result.
\begin{corollary}\label{cor}
Let $m\in \mathcal M$ and $g$ be an admissible perturbation for $m$. 
Let $u_\e$ be a solution to \eqref{P} corresponding to $m_\e(x)=m(x)+\e g$.
It holds, as $\e\to 0$
\begin{equation}\label{conv-int}
J(m_\e)=\int_\Omega  u_\e  =\int _\Omega u+\e \int _\Omega \dot u+\frac 12\e^2\int_\Omega \ddot u+o(\e^2).\end{equation}
\end{corollary}
\begin{proof}
By the previous theorem $z_\e \to \frac 12  \ddot u$ weakly in $H^1$ and strongly in $L^2$. 
Since $\Omega$ is bounded \eqref{conv-int} holds.
\end{proof}
In the following results we  introduce the adjoint state $p$ at $m$.

\begin{lemma}\label{lem:p-pos}
Assume \eqref{ipo:s}.
For any $m\in \mathcal M$ and for every $d>0$ there exists a unique solution to 
\begin{equation}\label{adjapp}
\begin{cases}
-d\Delta p-s(x)p\ln \frac K{u}+m(x)p+s(x)p=1 
& \text{ in } \Omega\\
\partial _\nu p=0 & \text{ on } \partial \Omega.
\end{cases}\end{equation}
 Moreover $p\in C^{1,\alpha}(\Omega)$ for every $\alpha \in (0,1)$ and, there exists
a positive constant $C=C(\underline{s}, \overline{s}, K,  |\Omega|)$ such that 
\begin{equation}\label{eq:prop}
\inf_{\Omega} p >0, 
\quad \|p\|_{H^1}\leq \frac C{(\min\{d,\underline{s}\})^\frac 12}, \quad \|p\|_{\infty}\leq C.
\end{equation}
\end{lemma}
\begin{proof}
Since the first eigenvalue of the operator $L_{u }$ is positive by proposition \ref{prop-lin-inv}, the problem \eqref{adjapp} admits a unique weak solution in $H^1(\Omega)$.   By regularity theory $p\in C^{1,\alpha}(\Omega)$ for every $\alpha \in (0,1)$.\\
 We multiply  \eqref{adjapp} by $-p^-$ and integrate over $\Omega$.  Then,  Proposition \ref{prop-lin-inv} yields
\[0\leq \lambda_1(L_{u })\int_\Omega \left(p^-\right)^2\leq -\int_\Omega  p^-\leq 0\
\]
showing that $p^-=0$ in $\Omega$.  The strict inequality then follows from the Strong maximum principle.  
Taking as test function $p$ in \eqref{adjapp} we deduce
\beq\label{ff1}
d\int_{\Omega}|\nabla p|^2+\int_{\Omega} p^2\left(m(x)+s(x)-s(x)\ln \frac K{u}\right)=\int_\Omega p\leq |\Omega|^{\frac 12} \|p\|_2\eeq
Moreover the coercivity of the operator, see Proposition \ref{prop-lin-inv},  gives 
\[\underline{s}\|p\|_2^2\leq \lambda_1(L_u)\|p\|_2^2\leq d\int_{\Omega}|\nabla p|^2+\int_{\Omega} p^2\left(m(x)+s(x)-s(x)\ln \frac K{u}\right)\]
so that
\begin{equation}\label{eq:pl2}
\|p\|_2\leq \frac 1{\underline{s}}|\Omega|^{\frac 12}.
\end{equation}
Then,  \eqref{ff1} implies that
\[
\begin{split}
 \min\{d,\underline s\}\|p\|^2_{H^1}\leq d\int_{\Omega}|\nabla p|^2+\int_\Omega s(x) p^2
\leq |\Omega|^{\frac 12} \|p\|_2+\overline{s}\ln \frac K{\bar \e}\|p\|_2^2\leq C
\end{split}\]
yielding the second conclusion. 
To conclude the proof, let us observe that $p$ solves the equation
\[-d\Delta p+s(x)p=f(x)p+1 \]
with $f(x):=s(x)\ln \frac K{u}-m(x)\in L^{\infty}(\Omega)$.
Then, taking into account \eqref{ipo:s} and \eqref{eq:pl2} the conclusion follows by exploiting classical elliptic regularity estimates.
\end{proof}

\section{Bang-bang property of the maximizer} \label{se:bang_bang_max}
Let us first make some observation enlightening the importance of the $L^{1}-
L^{\infty}$ constraint imposed in the class $\Mcal$. Indeed, note that if $m_{1}
<m_{2}$ then the positive solution to \eqref{P} associated with $m_{1}$ turns out 
to be a positive super-solution to \eqref{P} with $m_{2}$. Then, arguing as in 
proposition \ref{existence}, we deduce that $u_{1}\geq u_{2}$. As a consequence,
if we maximize the functional $J(m)$ without the constraint on the average $
\int_\Omega m=M$, we have that $\max J(m)$ is attained at   $m=0$.

Moreover, the maximization problem over the larger class $\mathcal M_{\geq}:=\{m\in L^\infty(\Omega) \ : \ 0\leq m(x)\leq \overline{m}(x) \text{ and }\int_\Omega m(x)\geq M\}$ is also attained in the class $\mathcal M$   due to the saturation of the constraint.
\\
On the other hand, let us show that the maximum problem on the class $\Mcal_{M}$
analogous to \eqref{eq:minclassebig} has no solution  as the following result shows.
\begin{proposition}\label{pro:nomax}
The supremum
\begin{equation}\label{eq:maxclassebig}
\sup_{m\in\Mcal_{M}} J(m), \quad \Mcal_{M}:=\left\{m\in L^{\infty}(\Omega) : m\geq 0, \int_{\Omega}m=M\right\} 
\end{equation}
is not achieved
\end{proposition}
\begin{proof}
 Let $x_0\in \Omega$ and $m_n(x)=M\rho_n(x-x_0)$ in $\Omega$, where $\rho_n\in 
C_c^\infty(\R^n)$ is a sequence of mollifliers such that $\rho_n\geq 0$, supp$
(\rho_n)\subset \overline{B(0,\frac 1n)}$ and $\int_{\R^n} \rho_n =1$. Then $m_n\in 
\Mcal_{M}$ for every $n$.  Since $0\leq m_n(x)\leq \|m_{n}\|_{\infty}\leq C$ we can argue as in Proposition \ref{existence} getting the existence of a solution $u_n$ to \eqref{P} corresponding to 
$m_n$ that satisfies $\overline{\eps}_{n}\leq u_n<K$   where $\overline{\eps}_{n}=Ke^{-\frac{\|m_{n}\|_{\infty}}{\underline{s}}}$.
So that $1\leq \frac{K}{u_{n}}$, and 
$u_{n}\ln\frac{K}{u_{n}}\geq 0$ is uniformly bounded from above by a positive constant. Then
\begin{equation}\label{eq:enun}
d\int_\Omega |\nabla u_n|^2+\int_\Omega m_n(x) u_n^2=\int_\Omega s(x) u_n^2 \ln \frac K{u_n}\leq C.\end{equation}
Then, up to a subsequence, $u_n\to u$ weakly in $H^1(\Omega)$ and strongly in $L^2(\Omega)$ and pointwise a.e. in $\Omega$. 

Let $\varphi \in C^\infty (\Omega)$ such that  $\varphi$ satisfies supp$\varphi \subset \Omega\setminus\{x_0\}$. It results
\[\left|\int_\Omega m_n u_n \varphi\right|< K \int_\Omega m_n |\varphi|\to KM|\varphi(x_0)|=0.\]
Since in $\R^n$, with $n\geq 2$, the capacity of a point is zero, the set $\{\varphi \in C^\infty(\Omega):\varphi(x_0)=0\}$ is dense in $H^1(\Omega)$. This implies that $u$ is a weak solution to 
\beq\label{pas-1}
\begin{cases}
-d\Delta u=u s(x)\ln \frac Ku 
& \text{ in } \Omega\\
\partial _\nu u=0 & \text{ on } \partial \Omega.
\end{cases}\eeq
Moreover, we have that for every $n$ the energy $\mathcal E_n(u_{n})$ (see 
\eqref{eq:defE}) associated with $m_n$ is strictly negative; indeed taking into
account \eqref{eq:enun}, one has
\[
\Ecal_{n}(u_{n})=\Ecal_{n}(u_{n})-\langle\Ecal'_{n}(u_{n}),u_{n}\rangle=-\frac14\into s(x)u_{n}^{2}\leq -\frac{1}4\underline{s}\overline{\eps}_{n}^{2} |\Omega|.
\]
 Recalling that $u_n$ is a sequence of minumum of the energy $\mathcal E_n(u)$,  and $u_n\to u$, then $u$ has to be a minimum of the energy associated to problem \eqref{pas-1}. This implies that $u=K$ in $\Omega$ and  $J(m_n)=\int_\Omega u_n\to K|\Omega|$
as $n\to \infty$ and shows that $
\max_{m\in\Mcal_{M}} J(m)$ is not attained. Indeed, as in Proposition \ref{existence} we have that $u<K$ for every $m\in\Mcal_{M}$; moreover $J(m)=K|\Omega|$ if and only if $u=K$ is a solution to \eqref{P}
which implies $m\equiv 0\not \in \Mcal_{M}$. 
Ley us briefly discuss the case $n=1$, considering 
 $\Omega =(0,1)$ and $x_0=\frac 12$.   Take 
\[m_{n}=M \begin{cases}
n^{2}(x-\frac12+\frac1n)& \frac12-\frac1n\leq x\leq \frac12
\\
-n^{2}(x-\frac12-\frac1n)& \frac12< x\leq \frac12+\frac1n
\\
0 & \text{elsewhere.}
\end{cases}\]
Integrating the equation satisfied by $u_n$ we have 
\[
-u'_n(x)=\int_0^x \left( s \ln \frac{K}{ u_n} -m_n   \right) u_n \]
Moreover, by applying  Sobolev embedding we obtain that   $u_n\to u$ in $C^1(\Omega)$.  We take two sequences $x_n=\frac12-\frac2n$ and $x'_n\in \frac12+\frac2n$.  We have
\[-u'_n(x_n)=\int_0^{\frac12-\frac2n} s \ln \frac{K}{ u_n}  u_n\to \int_0^{\frac12} s \ln \frac K u   u\]
while
\[-u'_n(x'_n)=\int_0^{\frac12+\frac2n} \left( s \ln \frac{K}{ u_n} -m_n\right) u_n\to \int_0^{\frac12} \left( s \ln \frac K u \right) u-u\left(\frac 12\right).\]
The convergence of $u_n\to u$ in $C^1$ then implies that $ u\left(\frac 12\right)=0$.  We can then pass to the limit in the equation for $u_n$ getting that $u$ is a solution to \eqref{pas-1} and we can conclude as in the previous case.
\end{proof}

\subsection{Proof of Theorem \ref{thm:mbang} }

Let us now focus our attention on the proof of Theorem \ref{thm:mbang}.
\\
As already observed, the proof strictly follows the argument in \cite{manaprcpde,FMP}.  
Let us first prove some  auxiliary but crucial results.  Some of these  are standard in optimal control theory,  but we prefer to include the details to make the work more accessible to everyone. 
A crucial tool in this and the fore-coming sections relies on the switching function $\vfi$ introduced in \eqref{eq:switch}. Let us observe that, in our context, 
$\vfi:=up$, as the problems \eqref{P} and \eqref{adjapp} admit a unique solution.

\begin{lemma}\label{le:firstinfo}
Let $m\in \Mcal$ and  $\vfi$  be the corresponding solutions to  \eqref{eq:switch}. Then it results,
\beq\label{switch}
\int_\Omega \dot u=-\int_\Omega g \vfi ,
\eeq
for every admissible perturbation $g$.
Moreover, if $m^{*}$ is a solution to the maximization problem in \eqref{eq:defmax} and,   $\vfi^{*}$ is the corresponding solution to \eqref{eq:switch}, it holds
\beq\label{dis}
\int_\Omega g \vfi^{*}\geq 0 
\eeq
for every admissible perturbation $g$ for $m^*$. 
\end{lemma}
\begin{proof}
We recall that the map $m\mapsto u_m$ is differentiable at the second order in the sense of Gateaux at $m$ in the direction $g$, and the derivatives $\dot{u}$ and $\ddot{u}$  satisfy \eqref{u1} and \eqref{u2} respectively.\\
Let $\dot{u}^{*}$ the derivative of $u^{*}$ with respect to $m$ along the direction $g$. In the sequel we will simply write $\dot{u}$ in place of
$\dot{u}^{*}$ to simplify the notation.

Let $u^*$ be the solution to \eqref{P} associated with $m^*$ and $p^*$ the corresponding adjoint state (solution of \eqref{adjapp}).
Multiplying  \eqref{adjapp} (written for $u^{*}$)
by $\dot{u}$, and \eqref{u1} by $p^*$ and integrating by parts immediately give \eqref{switch}. Moreover,
let $g$ be an admissible direction. 
 Exploiting  the  optimality of $m^{*}$, 
 \eqref{conv-int} and \eqref{switch}, we obtain
 \[
 0\geq\lim _{\e \to 0}\frac{J(m^{*}+\e g)-J(m^*)}{\e}=\int_\Omega \dot u=-\int_\Omega g u^{*}p^{*}=-\int_\Omega g \vfi^{*} .
\]
\end{proof} 
Associated with $m^*$ we can define the sets
\beq\label{insiemi}
\begin{split}
&
\Omega_0:=\{x \in \Omega\setminus \Omega': m^*(x)=0\}
 \\ 
&
\Omega^*:=\{x \in \Omega: 0<m^*(x)<\overline m(x\}
\\& 
\Omega_1:=\{x \in \Omega\setminus \Omega': m^*(x)=\overline m(x)\}
\end{split}
\eeq
up to a set of zero measure,   where $\Omega'$ is  defined in \eqref{eq:omegaprimo}.

As a consequence of the first order information included in Lemma \ref{le:firstinfo} we can prove the following result. 

 \begin{proposition}\label{prop:11.2}
Assume
$m^*$ is a solution to the maximization problem in \eqref{eq:defmax},  and let $\vfi^*$  be the corresponding solutions to  \eqref{eq:switch}.  Then there exists a constant $c>0$ such that 
\beq \label{insiemilivelloapp}
\vfi^*=c \ \hbox{in } \Omega^*
\eeq
and 
\beq \label{inequalities}
\forall \ (x_0,x^*,x_1)\in \Omega_0\times \Omega^*\times \Omega_1  \hbox{ we have } \vfi^{*}(x_1)\leq   \vfi^{*}(x^*)\le  \vfi^{*}(x_0).
\eeq
\end{proposition}
\begin{proof}
The proof closely follows the one of \cite[Theorem 7.2.26]{HP}; we include it here for the sake of clarity.\\
Define, for $n\in \N$ the set 
\begin{equation}\label{eq:omegastarn}
\Omega^*_n:=\{x\in \Omega: \frac 1 n\le m^*\le \overline{m}-\frac 1n\}.
\end{equation} 
We will prove that $\vfi^*$ is constant in $\Omega_n^*$. 
Since $\Omega^*=\cup_{n\in \N} \Omega_n^*$ this will prove the first assertion.\\
By contradiction, assume that $\vfi^*$ is not constant in $\Omega_n^*$.  Then it is possible to find two measurable sets $E_1$, $E_2$ contained in $\Omega_n^*$ such that 
\beq\label{rel-ins}
|E_1|=|E_2| \  , \ \ \ \ \ \int_{E_1} \vfi^*<\int_{E_2} \vfi^*.
\eeq
We can then choose $g$ defined by
\beq\label{gg}
g=\begin{cases}
1 & \text{ in } E_1\\
-1  & \text{ in } E_2\\
0 & \text{ in } \Omega \setminus \left(E_1\cup E_2\right).
\end{cases}
\eeq
By applying Lemma \ref{le:gadm}, $g$ is an admissible perturbation for $m^*$ and we have
\beq\label{dis-int}
<J'(m^*),g>=-\int_\Omega g \vfi^*= - \int_{E_1} \vfi^*+ \int_{E_2} \vfi^*>0.
\eeq

This  contradicts the optimality condition \eqref{dis} and concludes the proof of the first assertion. The positivity of $c$ follows from proposition \ref{existence} and Lemma \ref{le:firstinfo}.
\\
To prove the second one, let us assume, by contradiction, that there exist two sets 
of positive measure $ E_1\subset \Omega_0$ and $  E_2\subset \Omega^*_n$   such 
that   \eqref{rel-ins} holds. Then, using \eqref{gg} as before, we have 
\eqref{dis-int} and a contradiction arises. The other inequalities in \eqref{inequalities} can be proved analogously.
\end{proof}

\begin{remark}
Let us observe that arguing as in 
\cite[Theorem 7.2.26]{HP} it is also possible  to show that, if $m^*\in \mathcal M$ satisfies \eqref{insiemilivelloapp} and \eqref{inequalities} then it satisfies \eqref{dis} for every admissible perturbation $g$, so that $m^*$ is a local maximizer for $J(m)$.
\end{remark}
\begin{remark}
Analogous results hold  for a solution $m_{*}$ of the minimization problem in \eqref{eq:defmax}. In this case we obtain that
\[
\int_\Omega g \vfi_{*}\leq 0 \]
for every admissible perturbation $g$ for $m_*$.  We can then obtain \eqref{insiemilivelloapp} as in the previous case, while \eqref{inequalities} becomes
\[
\forall \ (x_0,x_*,x_1)\in \Omega_0\times \Omega^*\times \Omega_1\,  \Rightarrow\, 
\vfi_*(x_0)\leq   \vfi_*(x_*)\le  \vfi_*(x_1).
\]
\end{remark}

\begin{proof}[Proof of theorem \ref{thm:mbang}]
The proof follows a mix of the arguments in \cite[Theorem I]{manaprcpde} and \cite[Theorem 3]{FMP}.
Multiplying \eqref{adjapp} by $\ddot{u}$, \eqref{u2} by $p^{*}$ and integrating by parts, gives
\beq\label{secordapp}
\int_\Omega \ddot u=-2\int_\Omega g \dot u p^{*}-\int_\Omega s(x)\frac{\dot u^2}{u^*}p^{*} 
\eeq
By \eqref{u1}, we have
\[-g=\frac{-d\Delta \dot{u}-s(x)\dot{u}\ln \frac K{u^*}+m^*(x) \dot{u}+s(x)\dot{u}}{u^*}\]
which, together with \eqref{secordapp} gives
\[
\int_\Omega \ddot u=2\int_\Omega \frac {p^{*}}{u^*}\left(-d\dot u \Delta \dot u\right)+
\int_\Omega \frac{p^{*}}{u^*}\dot u^2\left(-2s(x)\ln \frac K{u^*}+2m^*(x)+s(x)\right)
\]
that we rewrite as
\beq 
\int_\Omega \ddot u=2d\int_\Omega \frac {p^{*}}{u^*}|\nabla  \dot u|^2+
\int_\Omega \dot u^2\underbrace{\left(\frac{p^{*}}{u^*}\left(-2s(x)\ln \frac K{u^*}+2m^*(x)+s(x)\right) -d\Delta \left( \frac {p^{*}}{u^*}\right)\right)}_{:=V_{m^*}(x)}.
\eeq
As $u^{*}$ and $p^{*}$ are $C^{1,\alpha}(\Omega) $ and $u^{*}\geq \bar{\eps}$, we deduce that  $V_{m^*}(x), \frac {p^{*}}{u^*}$ are uniformly bounded in $L^{\infty}(\Omega)$. In addition, there exist $A_1,A_2>0$ such that 
\[ J''( m^*)[g,g]=\int_\Omega \ddot u\geq A_1 d\int_\Omega |\nabla  \dot u|^2-A_2
\int_\Omega \dot u^2.
\]
Taking into account \eqref{u1} and applying  \cite[Proposition A.2]{FMP}  we have
\[ J''(m^*)[g,g]\geq A_1\|g\|^2_{W^{-1,2}(\Omega)} -B_2\|g\|^2_{W^{-2,2}(\Omega)}.\]
Let us now conclude the argument first for dimension $n\geq 2$.
Note that, if $g$ is supported in a sufficiently small set then an application of  \cite[Proposition 2.1.6]{FMP} yields,
\beq\label{sec-ord} 
J''(m^*)[g,g]\geq \frac 12 A_1\|g\|^2_{W^{-1,2}(\Omega)}.
\eeq
Now assume, by contradiction, that the set $ \Omega^{*}:=\{x\in \Omega: 0<m^*(x)<\overline{m}(x)\}$  has positive measure, so that the sets 
$\Omega^{*}_{n}$ (see \eqref{eq:omegastarn}) have positive measure for $n$ sufficiently large. 
To reach a contradiction,  it is enough to produce an admissible perturbation $g$ as in \eqref{gg}, with $\int_\Omega g=0$, where 
 $E_1\cup E_2\subset \Omega_{n}^{*}$ and $|E_1\cup E_2|$ so small that \eqref{sec-ord} holds.
Indeed,  by proposition \ref{prop:11.2} we have
\[ J'(m^*)[g]=\int_\Omega \dot u=-\int_\Omega gu^*{p^{*}} =-c \int_\Omega g=0\]
so that
\beq\label{sec-ord2}
J(m)-J(m^*)=\frac 12 \e^2  J''(m^*)[g,g] \geq \frac 14 \e^2 A_1\|g\|^2_{W^{-1,2}(\Omega)}>0  
\eeq
which gives a contradiction.  In the one dimensional case, we can show that
\eqref{sec-ord} holds as well, by arguing as in \cite[Proposition 2.1.8]{FMP} choosing $r^{2}\phi_{r}$ in the place of the cut-off function $\phi_{r}$ introduced in (2.1.18) and exploiting the embedding $W^{2,2}(\R)$ in $L^{\infty}(\R)$.
\\
As a conclusion, we deduce that 
\[\text{meas}\left( \{x\in \Omega\setminus \Omega': 0<m^*(x)<\overline{m}(x)\}\right)=0.\]
Then there exists a measurable set $\omega^*\subset \Omega\setminus \Omega'$ such that 
 $m^*(x)=\overline{m}(x)\chi_{\omega^*}$ and
\[\int_{\omega^*} \overline{m}=\int_\Omega  m^*=M.\]

Now assume that the set $\{x\in \Omega\setminus \Omega': \varphi^*=c^*\}$
has positive measure, with $c^*:=\sup_{x\in \Omega_1} \varphi^*(x)$.  Suppose by contradiction that both the sets 
$F:=\{x\in \Omega\setminus \Omega': \varphi^*=c^*\}\cap \Omega_0$ and $G:=\{x\in \Omega\setminus \Omega': \varphi^*=c^*\}\cap \Omega_1$ have positive measure,  so that the sets 
$F_n:=\{x\in \Omega\setminus \Omega': \varphi^*=c^*\}\cap \Omega_0\cap \Omega_{n}^{*}$ and $G_n:=\{x\in \Omega\setminus \Omega': \varphi^*=c^*\}\cap \Omega_1\cap \Omega_{n}^{*}$  (see \eqref{eq:omegastarn}) have positive measure for $n$ sufficiently large.  Let $x_F$ and $x_G$ be Lebesgue points of $F_n$ and $G_n$ respectively and $E_1:=B(x_F,r_F)\cap F_n$, $E_2:=B(x_G,r_G)\cap G_n$ where $r_F$ and $r_G$ are so small such that estimate \eqref{sec-ord} holds and such that $|E_1|=|E_2|.$ We take $g$ as in \eqref{gg}. It is an admissible perturbation and, as before, by \eqref{sec-ord}, we obtain \eqref{sec-ord2} which gives a contradiction. Then, either the set $\{x\in \Omega\setminus \Omega': \varphi^*=c^*\}$ is either contained in $\omega^*$ or in $\Omega\setminus (\Omega'\cup \omega^{*})$, yielding the conclusion. 
\end{proof}

\begin{remark}\label{11.7}
Since $m^*(x)=\overline{m}(x)\chi_{\omega^*}\in \mathcal M$ and $\omega^*\subset \Omega\setminus \Omega'$,  then 
\[\int_\Omega m^*=M<\int_{\Omega\setminus \Omega'} \overline m (x) .\]
This implies that $|\omega^*|<|\Omega\setminus \Omega'|.$

\end{remark}

 \begin{remark}
In the context of logistic nonlinearities and assuming $\overline{m}=1$, 
it has been shown that \eqref{inequalities} holds with strict inequalities (see \cite[Proposition 3]{manaprjmpa}), by exploiting the equation satisfied by the switching function $\varphi^*:=u^{*}p^*$.
\\
Proving that this property holds  in the case of Gompertz nonlinearities and under our general assumptions on $s$ and $\overline{m}$ remains an interesting open problem (see also \cite[Theorem 3]{FMP}). 
\end{remark}
\section{Qualitative properties of  minimizers}\label{se:miniprop}
In this Section the dependence on  $d$ is not relevant, so that it will be omitted.
Let us first observe that arguing as at the beginning of Section \ref{se:bang_bang_max}, minimizing $J$ without the constraint on the average $\int_\Omega m=M$, one  has that $\min J(m)$ is attained at $m= \overline{m}$. 
\\
In addition, the minimization problem for $J $ over the larger class $\mathcal M_{\leq}:=\{m\in L^\infty(\Omega) :  0\leq m(x)\leq \overline{m}(x) \text{ and }\int_\Omega m(x)\leq M\}$ is attained on the set $\mathcal M$ due to saturation of the constraint.

Let us first tackle the minimization problem 
 \begin{equation}\label{eq:minclassebig}
\min_{m\in\Mcal_{M}} J(m), \quad \Mcal_{M}:=\left\{m\in L^{\infty}(\Omega) : m\geq 0, \int_{\Omega}m=M\right\} 
\end{equation}

When $f(x,t)$ is of logistic type, the minimum problem in \eqref{eq:minclassebig} is solved by the constant weight $M/|\Omega|$ (see \cite{Lou}). The counterpart of this result in
our context is obtained, assuming that $s(x)\equiv s$ is constant,   when minimizing the cost functional $\widetilde{J}: \Mcal_{M}\to (0,+\infty)$ 
 
\begin{equation}\label{eq:Jtilde}
\widetilde J(m)=\int_\Omega \ln u \ dx.
\end{equation}
Indeed, the following result holds.

\begin{lemma}\label{lemma:minlogu}
Assume that $s(x)\equiv s$ is constant. Then, $M/|\Omega|\in \Mcal_{M}$ and 
it is the unique  global minimum point for the functional $\widetilde J(m)$ with
the corresponding solution $U_{\infty}$, given in \eqref{eq:Uinfty}, namely 
\[
\min_{m\in \mathcal M_M}\widetilde J(m)=\widetilde J\left(\frac{M}{|\Omega|}\right)=\int_{\Omega} \ln U_\infty \ dx=\ln K |\Omega|-\frac M{s}.
\]
\end{lemma}
\begin{proof}
Multiplying the equation in Problem  \eqref{P} by $\frac 1u$ and
integrating by parts yields
\[ 
-d \int_\Omega \frac{|\nabla u|^2}{u^2}\ dx=s|\Omega|\ln K-s
\into \ln u dx-M =s\into \ln U_{\infty}-s\into \ln u dx
\]
so that, one has
\[
\widetilde J(m)=\widetilde J\left(\frac{M}{|\Omega|}\right)+\frac{d}{s}\into
\frac{|\nabla u|^2}{u^2}\ dx,
\]
for every $m\in \Mcal_M$. As a consequence, the  only   global minimum point is given 
by $M/|\Omega|$ as the unique solution associated with it, is the constant one.
\end{proof}
As a consequence of the above lemma, we obtain that the solution to the minimization problem in \eqref{eq:minclassebig} is given by the constant weight $m$ when $s$ is constant, as shown in the following result.
\begin{theorem}\label{th:minscost}
Assume $s(x)\equiv s$ is constant in $\Omega$. Then, $M/|\Omega|$
is the unique  optimal control $m_{*}$, for  problem \eqref{eq:minclassebig}.
\\
If in addition 
\begin{equation}\label{ipo:M}
\overline{m}\geq \frac{M}{|\Omega|},
\end{equation}
then the minimum problems \eqref{eq:defmax} and \eqref{eq:minclassebig} are equivalent and solved by the constant weight $M/|\Omega|$.
\end{theorem}
\begin{proof}
Let $m_{*}$ be a solution to the minimum problem in \eqref{eq:defmax}, with corresponding solution $u_{*}$. Then it results 
\begin{equation}\label{eq:1}
K|\Omega| e^{-\frac M{s|\Omega|}}=\int_{\Omega} U_\infty \ dx\geq \int_{\Omega} u_{*} \ dx=\int_{\Omega} e^{w_{*}} \ dx
\end{equation}
where $w_{*}=\ln u_{*}$.
Applying Jensen's inequality to the strict convex function $\varphi(t)=e^t$ and taking into account Lemma \ref{lemma:minlogu}, we obtain
\begin{equation}\label{eq:2}
\into U_{\infty}\geq \into u_{*}=\int_{\Omega} e^{w_{*}} \geq |\Omega|e^{\frac 1{|\Omega|}\int_\Omega w_{*}}>|\Omega|e^{\frac 1{|\Omega|}\int_\Omega w_{\infty}}=\into U_{\infty},
\end{equation}
where $w_\infty=\ln U_\infty$. Then, the above inequalities  has to be equalities. 
This immediately implies that the minimum of $J$ is achieved only by the constant 
weight $M/|\Omega|$ with associated solution $U_{\infty}$.
To conclude the proof, it is sufficient to observe that, if \eqref{ipo:M} holds, 
then $M/|\Omega| $ belongs to $\Mcal$.
\end{proof}
In view of the above result, it is natural to wonder which  should be the optimal weight when $s$ or $\overline{m}$ is not constant. 
In the following result we show that $\alpha s$ may not be a minimizer if $s$ is not constant.
\begin{proposition}\label{pro:minsnocostante}
Let $s$ be a non-constant function and suppose that 
\begin{equation}\label{eq:ipombarrato}
\overline{m}>\frac{M}{\into s}s.
\end{equation}
Then the function $m=\alpha s$ with $\alpha=M/\into s$ does not solve the minimization problem in \eqref{eq:defmax}.
\end{proposition}
\begin{proof}
First note that by hypothesis $m\in \Mcal$. Then, we argue by contradiction, 
and suppose that $m$ is a minimum point of $J$. Proposition 
\eqref{u-costante} imply that the solution $u$ associated with $m$ is the   
constant$u=U_{\infty}=Ke^{-M/\into s}$. 
Then the equation for the  adjoint state $p$ (see \eqref{adjapp}) reduces to 
\[
-d\Delta p+s(x)p\left(-\frac{M}{\into s}+\alpha+1\right)=1\,\;\Leftrightarrow\;
-d\Delta p+s(x)p=1.
\]
So that, the switching function $\vfi:=up=U_{\infty}p$ solves 
\begin{equation}\label{eq:vfinuova}
-d\Delta \vfi+s(x)\vfi=U_{\infty}.
\end{equation}
Note that, by hypothesis 
\[
\Omega^{*}=\{0<m<\overline{m}\}=\Omega.
\]
Then, Proposition \ref{prop:11.2} implies  that $\vfi $ is constant in $\Omega$ 
which, in view of \eqref{eq:vfinuova},  yields that $s$ is constant.
\end{proof}

\begin{remark}
The regularity of the domain $\Omega$ in Theorem \ref{thm:mbang}  can be weakened in some cases, provided that the regularity $W^{2,p}$-estimates and the Sobolev embedding theorem still hold.
\end{remark}

\begin{remark}\label{rem:minlogus}
The same proof of Lemma \ref{lemma:minlogu} shows that
\[
\min_{m\in \mathcal M_M}\into s(x) \ln u=\into s(x) \ln U_{\infty}.
\]
Analogously, if one consider the logistic nonlinearity $f(x,t)=m(x)t-s(x)t^{2}$ one obtains, exploiting the same argument, that
\[
\min_{m\in \mathcal M_M}\into s(x)  u=\into s(x)  \widetilde{U}_{\infty}.
\]
where $\widetilde{U}_{\infty}:=\into m(x)/\into s(x)$. In both cases, it is not clear how to recover information on the population size, i.e. the 
functional $J$. 
\\
A comparison between the situation for a constant $s$ and a piece-wise one is 
performed in Section \ref{se:num_results} via numerical simulations. The results 
are shown in  Figures ~\ref{fig:gompertz_1D_min} and 
~\ref{fig:gompertz_1D_min_piecewise_heatmaps_profiles}. In particular, the red 
graphs in Figure ~\ref{fig:gompertz_1D_min} shows the results proved in Theorem 
\ref{th:minscost}, while the blue ones reveal a non-constant value of the minimum 
of the cost functional $J(m_*)$  depending on $d$. This becomes more evident in 
Figure ~\ref{fig:gompertz_1D_min_piecewise_heatmaps_profiles} where (on the 
bottom left side) the graphs of the optimal control $m_*$ are represented for 
different values of the diffusion rate $d$. All of them show a piece-wise behavior 
of the optimal control; moreover,  the discontinuity point of $m_*$ appears at the 
same point at which $s$ is not continuous. At the same time, Proposition 
\ref{pro:minsnocostante} establishes that $m_*$ is not proportional to $s$ as, in 
this example, \eqref{eq:ipombarrato} is verified.
\end{remark}

\begin{remark}\label{rem:min_s_const}
Lemma \ref{lemma:minlogu} and Theorem \ref{th:minscost} show that, when $s$ 
is constant, the minimization problem in $\Mcal$ is equivalent to problem 
\eqref{eq:minclassebig} if $M/|\Omega|\in \Mcal$ that is if $\overline{m}(x)\geq 
M/|\Omega|$. On the other hand, the case in which $\overline{m}$ vanishes in a 
set of positive measure is open even in the logistic model. 
In this framework,  a natural candidate to solve the minimum problem
in \eqref{eq:defmax} might be a function piecewise constant function in $\Omega\setminus \Omega'$.\\
This would yield two interesting facts: in this situation  the optimal control is not constant and the associated state does not correspond to the limit of $u$ as $d$ tends to plus infinity (see Proposition \ref{pro:dinfinito} ahead).
The results obtained in Section \ref{se:num_results} show the interplay between $\overline{m}$ and $M$. In particular, in Figure~\ref{fig:gompertz_1D_i_ii} (top left and right) it is represented the graph of the optimal control $m_*$ for $\overline{m}:=\chi_{(\frac12,\frac34)}$ and $M=0.2$ (top, left), and $M= 0.03$ (top right). In both the situations hypothesis \eqref{ipo:M} is not globally satisfied. Moreover, this seems to allow $m_*$ to be not constant  in $\Omega\setminus \Omega'$ and to saturate the constraint $\overline{m}$ in a set of small measure depending on the value of $M$ compared to $\into \overline{m}$. 
\end{remark}

\section{Asymptotical study  with respect to $d$}\label{se:asym_d}
In this section we  consider $m$ fixed and we study the behavior of the functional $J_{d}$ with respect to $d\in (0,+\infty)$. 
This study will be crucial in the proof of Theorem \ref{teo-concentration}.
\\
In this study  section we will denote by $u_{d}$   the solution of \eqref{P} for a given $d$.
In the following results,  we compute the limit of $u_{d}$ as $d\to 0^{+}$
and as $d\to\infty$.  The first one turns out to be $U_{0}=Ke^{-\frac{m(x)}{s(x)}}$,
the latter is $U_{\infty}$ which we recall  is not a solution of \eqref{P}, unless  $m(x)$ and $s(x)$ are proportional (see Proposition \ref{u-costante}).
\begin{proposition}\label{pro:dzero}
Let $u_{d}$ the unique solution to problem \eqref{P}. 
Then 
\begin{equation}\label{eq:Uzero}
\lim_{d\to0^{+}}u_{d}=U_0:=Ke^{-\frac{m(x)}{s(x)}}, \qquad  \text{in $L^{q}(\Omega)$ for $q\in [1,+\infty)$}.
\end{equation} 
\end{proposition}
\begin{proof}
We follow the argument of \cite{Lou, deanizha}.
\\
{\bf Step 1.} Let us  first assume that 
$m, \,s \in  C^{2}(\Omega)$ and  show that there exists 
$C>0$ and $d_{0}>0$ such that,
\begin{equation}\label{eq:msreg}
\|u_{d}-U_{0}\|_{\infty}\leq Cd, \qquad \text{for every $d\in (0,d_{0})$.}
\end{equation} 
Let us consider $\overline{u}:= U_0+C^*d$ for $C^*\in \R^+$ to be detected.
Recalling  that $s(x)\ln\frac{K}{U_0}-m(x)=0$, one obtains
\begin{align*}
-d\Delta \overline{u}-\overline{u}\left(s(x)\ln\frac{K}{\overline{u}}-m(x)\right)&
=-d\Delta U_0-\overline{u}\left(s(x)\ln\frac{K}{\overline{u}}-m(x)\right)
\\
&=-d\Delta U_0-\overline{u}s(x) \ln \frac{U_0}{U_0+C^*d}
\\
&=-d\Delta U_0-\overline{u}s(x)\ln\left(1-\frac{C^*d}{U_0+C^*d}\right).
\end{align*}
Since $U_0>0$ and does not depend on $d$, there exists $d_1$ sufficiently small such that
\[
\dfrac{\ln\left(1-\frac{C^*d}{U_0+C^*d}\right)}{-\frac{C^*d}{U_0+C^*d}}\in (\frac34,\frac54), \text{ for every $d\in (0,d_1)$.}
\]
Then
\[
-d\Delta \overline{u}-\overline{u}\left(s(x)\ln\frac{K}{\overline{u}}-m(x)\right)\geq -dU_0 H(x)+\frac{\underline{s}}2 C^* d  
\]
where $H(x)$ is a bounded function in $\Omega$, thanks to the regularity assumption on $m$ and $s$.
So that, taking into account that $\inf U_0>0$ we can find $C^*$ such that $\overline{u}$ is a super-solution to Problem \ref{P}.
Analogous computations shows that there exists a constant $C_*$ such that $\underline{u}:=U_0-C_*d$ is a sub-solution. Since $\underline{u}<\overline{u}$, Proposition \ref{existence} implies that $\underline{u}\leq u_d\leq\overline{u}$. So that $u_d\to U_0$ in $L^\infty(\Omega)$.
\\
{\bf Step 2.}
Let us now consider  $s$ satisfying assumption \eqref{ipo:s} and $m\in \Mcal$,
and denote with $s$ and $m$ their trivial extension to $\R^{N}$.
Then, there exist 
 $m_{n},\, s_{n}\in C^{2}(\overline{\Omega})$ such that
\begin{equation}\label{eq:msmoll}
\begin{split}
\|m_{n}\|_{L^{\infty}(\Omega)}\leq \|m\|_{L^{\infty}(\Omega)},\; & \into m_{n}>0,
\; \underline{s}\leq s_{n}\leq \overline{s},
\\
m_{n}\to m, \;&  s_{n}\to s, \; \text{in $L^{1}(\Omega)$},
\end{split}
\end{equation}
with $\underline{s}$, $\overline{s}$ given in \eqref{ipo:s}.
Then, applying Proposition \eqref{existence}, for every $d>0$ there exists $u_{n}$ solution to  Problem \eqref{P} (with $f(x,u)=f_{n}(x,u)=s_{n}\ln\frac{K}{u}-m_{n}(x)u$ for $u>0$); moreover $\overline{\eps}_{n}\leq u_{n}\leq K$
 where $\overline{\eps}_{n}=Ke^{-2\frac{\|m\|_{L^{\infty}(\Omega)}}{\underline{s}}}$.
As a consequence, $u_{n}$ satisfies
\[\begin{split}
d\|\nabla u_{n}\|_{L^{2}(\Omega)}^{2} &=
\into u^{2}_{n}\left(s_{n}
\ln\frac{K}{u_{n}}-m_{n}\right)\leq K^{2}\left(|\Omega|\|\overline{m}\|_{L^{\infty}(\Omega)}+\into \overline{s}\ln\frac{K}{u_{n}}\right)
\\
&\leq
3K^{2}|\Omega|\|\overline{m}\|_{L^{\infty}(\Omega)}.
\end{split}\]
So that, there exists $u\in H^{1}(\Omega)\cap L^{\infty}(\Omega)$ such that $u_{n}\to u$ weakly in $H^{1}(\Omega)$ and strongly in $L^{q}(\Omega)$ for every $q\in [1,+\infty)$. This and \eqref{eq:msmoll} imply that $f_{n}(x,u_{n})\to f(x,u)$ in
$L^{q}(\Omega) $ for every $q\in [1,+\infty)$.
Passing to the limit in the equation satisfied by $u_{n}$, we obtain, by uniqueness, that $u$ is the solution to Problem \eqref{P}. 
Then,  for every $\sigma >0$, we can fix $n_{1}$ such that
$\|u_{d}-u_{n_{1}}\|_{L^{q}(\Omega)}+\|U_{n_{1}}-U_{0}\|_{L^{q}(\Omega)}\leq \sigma$. Moreover, we can exploit \eqref{eq:msreg} proved in Step 1 to obtain that 
$\|u_{n_{1}}-U_{n_{1}}\|_{L^{\infty}(\Omega)}\leq \sigma$, for $d\in (0,d_{\sigma})$, where $U_{n}=Ke^{-\frac{m_{n}(x)}{s_{n}(x)}}$. All in all, we deduce 
\[\begin{split}
\|u_{d}-U_{0}\|_{L^{q}(\Omega)}
&\leq \|u_{d}-u_{n_{1}}\|_{L^{q}(\Omega)}+
\|u_{n_{1}}-U_{n_{1}}\|_{L^{q}(\Omega)}+\|U_{n_{1}}-U_{0}\|_{L^{q}(\Omega)}
\\
&\leq 2\sigma+\|u_{n_{1}}-U_{n_{1}}\|_{L^{q}(\Omega)}\leq 3\sigma \quad \text{for $d\in (0,d_{\sigma})$},
\end{split}\]
yielding the conclusion.
\end{proof}
\begin{proposition}\label{pro:dinfinito}
Let $u_{d}$ be the unique solution to problem \eqref{P}. 
Then 
\begin{equation}\label{eq:ulimUinfty}
\lim_{d\to+\infty}u_{d}=U_\infty, \qquad  \text{in $H^{1}(\Omega)$}
\end{equation} 
where $U_{\infty}$ is defined in \eqref{eq:Uinfty}.
\end{proposition}
\begin{proof}
Let us first observe that Proposition \ref{existence} yields $\overline{\varepsilon}\leq u_{d}(x)\leq K$, so that, assumptions \eqref{ipo:s} and \eqref{M-set} yield
\[
\|\nabla u_{d}\|_{2}^{2}=\dfrac1d \int_{\Omega}u^{2}_{d}\Big(s(x)\ln\frac{K}{u_{d}}-m(x)\Big)\leq \frac{C}d.
\]
As a consequence, there exists a positive constant $C$ such that $u_{d}\to C$  
strongly in $H^{1}(\Omega)$ as $d\to +\infty$.
Moreover, integrating the equation satisfied by $u_{d}$ and exploiting the homogeneous Neumann boundary conditions, one obtains
\[
0=\int_{\Omega}u_{d}\Big(s(x)\ln\frac{K}{u_{d}}-m(x)\Big).
\]
Then, taking the limit as $d\to \infty$ we infer
\[
0=\int_{\Omega}\Big(s(x)\ln\frac{K}{C}-m(x)\Big)=\ln\frac{K}{C}\into s(x)-M.
\]
which gives the exact value of $C$, yielding the conclusion. 
\end{proof}
Let us deepen the asymptotical study with respect to  $d$ large by studying a first order expansion of $u_{d}$. We will perform this expansion, first for every $m$ fixed for $u_{d}$, the adjoint state $p$ and the switching  function $\vfi=up$. Then for the corresponding quantities associated with $m^{*}_{d}$. 
As already observed, these expansion will be crucial in proving Theorem \ref{teo-concentration}.

\begin{lemma}\label{le:expd}
For every $m$ fixed, let $u_{d}$ be the solution to problem \eqref{P}.
Then, up to a subsequence, 
\[
\lim_{d\to+\infty}d(u_{d}-U_{\infty})=\eta, \quad \text{in $H^{1}(\Omega)$ }
\]
where $U_{\infty}=Ke^{-\frac{M}{\int_{\Omega}s(x)}}$ 
is defined in \eqref{eq:Uinfty}
and $\eta$ is the unique solution to the problem
\begin{equation}\label{eq:eta}
\begin{cases}
-\Delta \eta=U_{\infty}\left(s(x)\frac{M}{\into s}-m(x)\right)
\\
\frac{\partial \eta}{\partial \nu}=0
\end{cases}
\end{equation}
that satisfies
\beq\label{2}
\left(\frac{M}{\into s(x)}-1\right)\int_\Omega  s(x)\eta =\int_\Omega m(x) \eta.
\eeq
\end{lemma}
\begin{proof}
Let $\eta_{d}=d(u_{d}-U_{\infty})$, then $\eta_{d}$ is such that $0<U_{\infty}+\frac1d\eta_d<K$ and it  solves
\begin{equation}\label{eq:etad}
\begin{split}-\Delta \eta_{d}&
=-d\Delta u_{d}=s(x)u_{d}\ln\frac{K}{u_{d}}-m(x)u_{d}
\\
&= s(x)(U_{\infty}+\frac1d \eta_{d})\ln\frac{K}{U_{\infty}+\frac1d \eta_{d}}-m(x)\left(U_{\infty}+\frac1d \eta_{d}\right).
\end{split}
\end{equation}
Then $\eta_{d}$ is bounded in $H^{1}(\Omega)$, once we show that $\eta_{d}$ is bounded in $L^{2}(\Omega)$.

In order to do this, let us argue by contradiction
and suppose that, up to a subsequence, $\|\eta_{d}\|_{2}\to +\infty$, then the function
\[
z_{d}:=\dfrac{\eta_{d}}{\|\eta_{d}\|_{2}}
\]
is such that $\|z_{d}\|_{2}=1$ and $z_{d}$ satisfies 
\beq\label{eq:tzetad}
-\Delta z_{d}=\frac{u_d}{\|\eta_d\|_{2}}\left(s(x)\ln\frac{K}{u_{d}}-m(x)\right).
\eeq
Taking $z_d$ as test function in \eqref{eq:tzetad} one obtains
\[\int_\Omega |\nabla  z_d|^2=\frac{1}{\|\eta_d\|_{2}}\int_\Omega\left(s(x)\ln\frac{K}{u_{d}}-m(x)\right)u_d z_d\leq \frac{C}{\|\eta_d\|_{2}}.
\]
As a consequence, there exists $z\in H^1(\Omega)$ such that $\|z\|_{L^2 }=1$, $z_d$ converges to $z$ strongly in  $H^1(\Omega)$ and $z$ is constant. So that
 $ z=1/\sqrt{|\Omega|}$. 
On the other hand, integrating  \eqref{eq:tzetad} one deduces that
\[\begin{split}
0=&\int_\Omega\left(s(x)\ln\frac{K}{u_{d}}-m(x)\right)u_d=U_\infty\into \left(s(x)\ln\frac{K}{u_{d}\pm U_\infty}-m(x)\right)
\\
&+
\into (u_d-U_\infty)\left(s(x)\ln\frac{K}{u_{d}}-m(x)\right)
\\
=&-U_\infty \into s(x)\ln\left(1+\frac{u_{d}- U_\infty}{U_\infty}\right)+
\into (u_d-U_\infty)\left(s(x)\ln\frac{K}{u_{d}}-m(x)\right).
\end{split}\]
Note that
\[
\frac{u_d-U_\infty}{U_\infty}=\frac{d\eta_d}{\|d\eta_d\|_2}\frac{\|u_d-U_\infty\|_2}{U_\infty}
=z_d\frac{\|u_d-U_\infty\|_2}{U_\infty}
\]
so that, we have
\[
\begin{split}
0=-\frac{U_\infty}{\|u_d-U_\infty\|_2} \into s(x)\ln\left(1+z_d\frac{\|u_{d}- U_\infty\|_2}{U_\infty}\right)+\into z_d \left(s(x)\ln\frac{K}{u_{d}}-m(x)\right).
\end{split}
\]
Passing to the limit, and taking into account that $U_\infty=Ke^{-\frac{M}{\into s}}$, we obtain
\[
0=-\frac1{\sqrt{|\Omega|}}\into s(x)dx
\]
which  contradicts \eqref{ipo:s}.

So that there exists $\eta \in H^{1}(\Omega)$ such that $\eta_{d}\rightharpoonup \eta$ weakly in $H^{1}(\Omega)$ as $d\to +\infty $. Then, Proposition \ref{u-costante} and  Sobolev embeddings imply that $\eta$ solves \eqref{eq:eta}.
The strong convergence in $H^{1}$ immediately follows once one take as test function $\eta_{d}-\eta$ in \eqref{eq:etad}.

In order to obtain \eqref{2}, we integrate 
\eqref{P} and use  the expansion 
$u_d=U_\infty+\frac 1d \eta_{d}$,  to have
\[
0=\into\left(U_{\infty}+\frac1d\eta_{d}\right)\left(s(x)\ln\frac{K}{U_{\infty}+\frac1d\eta_{d}}-m(x)\right).
\]
Taking into account \eqref{eq:Uinfty} one obtains
\[\begin{split}
0&=-d\into U_{\infty}s(x) \ln\left(1+\frac1{dU_{\infty}}\eta_{d}\right)-\into \eta_{d}s(x)\ln\left(1+\frac1{dU_{\infty}}\eta_{d}\right)
\\
&+\into\eta_{d}\left(s(x)\ln\frac{K}{U_{\infty}}-m(x)\right).
\end{split}\]
As $\eta_{d}\to \eta $ strongly in $H^{1}(\Omega)$, 
we can pass to the limit and obtain that  $\eta$ satisfies \eqref{2}. Further, we observe that equation \eqref{eq:eta} has infinitely many solutions that differs for a constant,  namely $\eta_1=\beta+\eta_2$ where $\beta\in \R$ for any  $\eta_1,\eta_2$  solutions to \eqref{eq:eta}. 

Moreover, there exists only one solution that satisfies \eqref{2}. Indeed,
let  $\eta_1$ satisfying \eqref{2}, then it holds
\[\left(\frac{M}{\into s(x)}-1\right)\int_\Omega  s(x)\eta_2 =\int_\Omega m(x) \eta_2 +\beta \into s(x)\neq \int_\Omega m(x) \eta_2
\]
if $\beta\neq 0$, 
concluding the proof.
\end{proof}

 \begin{corollary}
Suppose that $m\neq \alpha s$ for $\alpha=M/\into s$, then it results   
\[
J_d(m)>J_\infty(m),\qquad \text{ for $d$ sufficiently large.}
\]
\end{corollary}
 \begin{proof}
By applying Lemma \ref{le:expd} we can expand the functional $J_{d}$ when $d\to \infty$ as
\[J_d(m)=\into u_d=\into U_\infty +\frac 1 d\into \eta_d=U_\infty |\Omega|+\frac 1 d\into \eta +O(\frac 1{d^2})
\]
where $\eta_{d}=d(u_{d}-U_{\infty})$ and  $\eta$ satisifies \eqref{eq:eta} and \eqref{2}.  Multiplying equation \eqref{eq:eta} by $\eta$  integrating and using \eqref{2} one gets
\[\into |\nabla \eta|^2=U_\infty\left(\frac M{\into s}
\into s \eta-\into m\eta\right)=U_\infty\into s \eta>0,
\]
as $m\neq \alpha s$. As a consequence, we obtain
\[\into \eta =\into \frac {s\eta }{s}>\frac 1{\overline s}\into s \eta\geq 0,\]
yielding the conclusion.
\end{proof}
\begin{remark}
 We expect that  combining the argument of the proof of Lemma \ref{le:expd} and of \cite[Theorem 1]{manaprjmpa} it is possibile to obtain that $d\to J_{d}(m)$ is convex for $d$ sufficiently large, which in turn would yield  
that $d\to J_{d}(m)$ is strictly decreasing. 
This property is related with the exclusion
of the occurrence of  fragmentation phenomena, at least in dimension one.
We will obtain this result in a different way in Theorem \ref{teo-concentration}
(see Remark \ref{rmk:dsmall} also related to the numerical results). 
\end{remark}
\begin{lemma}\label{le:expdstar}
For every $m$ fixed, let $p_d$ be the unique solution to \eqref{adjapp}.
Then
\beq
\lim_{d\to+\infty}d(p_{d}-p_{\infty})=w
\ \ \ \hbox{ in }H^1(\Omega),
\eeq
where $p_\infty=\frac {|\Omega|}{\int_\Omega s(x)} $ and $w$ is the unique solution to the problem
\begin{equation}\label{eq:w}
\begin{cases}
-\Delta w=p_{\infty}\left(s(x)\frac{M}{\into s}-m(x)-s(x)\right)+1
\\
\frac{\partial w}{\partial \nu}=0
\end{cases}
\end{equation}
that satisfies
\beq\label{2-bis}
\left(\frac{M}{\into s(x)}-1\right)\int_\Omega  s(x)w -\int_\Omega m(x) w=\frac {p_\infty}{U_\infty}\into s(x) \eta
\eeq
where $\eta$ is defined in Lemma \ref{le:expd}.

\end{lemma}
\begin{proof}
By Proposition \ref{lem:p-pos} we already know that $\|p_d\|_{H^1}\leq C$ for $d>\delta>0$. Then, as $p\to \infty$,  $p_d\to p_\infty$ weakly in $H^1$ and strongly in $L^2$ as $d\to \infty$.  Moreover $p_\infty$ is constant and the convergence is strong in $H^1(\Omega)$.  Further, integrating the equation satisfied by $p_d$, exploiting the boundary condition, and passing to the limit we get
\[p_\infty\int_\Omega s(x) =|\Omega|.\]
Let $w_d=d(p_{d}-p_{\infty})$, then $w_d$ solves
\[-\Delta w_d=p_d\left(s(x)\ln \frac K{u_d}-m(x)-s(x)\right)+1.
\]
Then $w_d$ is bounded in $H^1(\Omega)$, once we show that $w_d$ is bounded in 
$L^2(\Omega)$.  As in the proof of Lemma \ref{le:expd}, we argue by 
contradiction, and we assume that, up to a subsequence,  $\|w_d\|_2\to +\infty$.  
We define $\widetilde w_d:=\frac{w_d}{\|w_d\|_2}$ so that 
$\|\widetilde w_d\|_2=1$ and   $\widetilde w_d$ satisfies
\beq\label{eq:tildewd}
-\Delta \widetilde w_d=\frac 1{\|w_d\|_2}p_d\left(s(x)\ln \frac K{u_d}-m(x)-s(x)\right)+\frac 1{\|w_d\|_2}.
\eeq
Taking $\widetilde w_d$ as test function, we have
\[
\int_\Omega |\nabla \widetilde  w_d|^2=\frac 1{\|w_d\|_2}\int_\Omega \widetilde 
w_d p_d\left(s(x)\ln \frac K{u_d}-m(x)-s(x)\right)+\frac 1{\|w_d\|_2}\int_\Omega 
\widetilde w_d\leq \frac C{\|w_d\|_2}.
\]
This implies that there exists $\widetilde w\in H^1$ such that, $w$ is constant, 
 $\|\widetilde w\|_2=1$ 
and $\widetilde w_d\to \widetilde w$ strongly in $H^1(\Omega)$.
So that $\widetilde w=\frac 1{\sqrt{|\Omega|}}$. 
On the other hand, integrating \eqref{eq:tildewd}, and recalling that $u_d= U_\infty(1+\frac 1{dU_\infty}\eta_d)$, by Lemma \ref{le:expd}, we have
\beq\label{w-unica}\begin{split}
0&=\int_{\Omega}p_d\left(s(x)\ln \frac K{u_d}-m(x)-s(x)\right)+|\Omega|\\
&=\int_\Omega (p_d-p_\infty)\left(s(x)\ln \frac K{u_d}-m(x)-s(x)\right)+|\Omega|\\
&+p_\infty \int_\Omega
\left(s(x)\ln \frac K{U_\infty}-s(x) \ln (1+\frac 1{dU_\infty}\eta_d)-m(x)-s(x)\right)\\
&=\int_\Omega (p_d-p_\infty)\left(s(x)\ln \frac K{u_d}-m(x)-s(x)\right)
-p_\infty\int_\Omega s(x)\ln (1+\frac 1{dU_\infty}\eta_d).
\end{split}
\eeq
Multiplying by $\frac d{\|w_d\|_2}$, we obtain
\[p_\infty \frac d{\|w_d\|_2} \int_\Omega  s(x) \ln (1+\frac 1{dU_\infty}\eta_d)=\int_\Omega \widetilde w_d\left(s(x)\ln \frac K{u_d}-m(x)-s(x)\right).
\]
Passing to the limit 
\[0=-\widetilde w\int_\Omega s(x)<0\]
which gives a contradiction, showing that $\|w_d\|_2\leq C$. 
This proves that $w_d$ is bounded in $H^1(\Omega)$, so that there exists $w\in 
H^1$ such that $w_d\to w$ weakly in $H^1(\Omega)$ and strongly in $L^2(\Omega)
$. Passing to the limit into the weak formulation of the equation satisfied by $w_d$ 
we obtain \eqref{eq:w}. Then, it is easy to see that the convergence is strong in 
$H^1(\Omega)$.
 From \eqref{w-unica} we have 
\[\begin{split}
0&=\int_\Omega w_d\left(s(x)\ln \frac K{U_\infty}-s(x) \ln(1+\frac 1{dU_\infty}\eta_d)-m(x)-s(x)\right)
\\
&-dp_\infty\int_\Omega s(x) \ln (1+\frac 1{dU_\infty}\eta_d).
\end{split}\]
Recalling the definition of $U_\infty$ and the fact that $\eta_d\to \eta$ in $H^1$ and $w_d\to w$ in $H^1$ this gives
\[\frac{p_\infty}{U_\infty}\into s(x) \eta=\into m(x) w+s(x) w-s(x) w \frac M{\into s(x)}
\]
and implies that $w$ satisfies \eqref{2-bis}. This condition, together with \eqref{eq:w} implies that $w$ is uniquely determined.

\end{proof}
\begin{remark}\label{rm:aggiunti}
Let us observe that, given $\eta$ the solution to \eqref{eq:eta} and \eqref{2},  
it is possibile to obtain $w$ as an adjoint state of $\eta$ with respect to the optimization problem
\begin{equation}\label{eq:defJ1}
\max_{m\in \Mcal}J_1(m),\quad\text{where $J_{1}(m):=\int _\Omega \eta$},
\end{equation}
by considering  the augmented form
\[\begin{split}
H(m)=&\into \eta+p_{1}\left(\Delta \eta+U_{\infty}\Big(s(x)\frac{M}{\into s(x)}-m(x)\Big)\right)
\\
&+ p_{2}\left(\frac{M}{\into s(x)}-1\right)\into s(x)\eta-p_{2}\into m(x)\eta.
\end{split}\]
where $p_1\in H^1(\Omega)$ is an adjoint state and $p_2\in \R$.

Indeed, by computing the optimality conditions, one deduces that 
$p_{1}$ is a solution of   \eqref{eq:w} with $p_{2}$ in place of $p_{\infty}$. 
In addition, by integrating \eqref{eq:w} one immediately obtains that $p_{2}=p_{\infty}$. 
\end{remark}

\begin{lemma}\label{le:expd3}
For every $m$ fixed, let $\varphi_d:=p_du_d$.
Then
\beq\label{eq:limzd}
\lim_{d\to+\infty}d(\varphi_d-p_{\infty}U_\infty)=p_\infty\eta+U_\infty w
\quad  \hbox{ in }H^1(\Omega),
\eeq
where $\eta$ and $w$ are  defined in Lemma \ref{le:expd} and Lemma \ref{le:expdstar}. Moreover $z:=p_\infty\eta+U_\infty w$ is a solution to the problem
\begin{equation}\label{eq:zinfty}
\begin{cases}
-\Delta z =p_{\infty}U_\infty\left(2 s(x)\frac{M}{\into s}-2m(x)-s(x)\right)+U_\infty
\\
\frac{\partial z_\infty}{\partial \nu}=0.
\end{cases}
\end{equation}
\end{lemma}
\begin{proof}
By the definiton of $\varphi_d$, using Lemma \ref{le:expd} and Lemma \ref{le:expdstar} we have that 
\[d(\varphi_d-p_{\infty}U_\infty)=p_\infty\eta_d+U_\infty w_d+\frac 1{d^2} \eta_dw_d.\]
We can then pass to the limit getting that $d(\varphi_d-p_{\infty}U_\infty)\to p_\infty\eta+U_\infty w$ in $H^1(\Omega)$. The fact that $z$ solves \eqref{eq:zinfty} is a consequence of \eqref{eq:limzd}.
\end{proof}
In the sequel we will also analize the asymptotical behavior of the optimal weight
$m^*$ with respect to $d$. For the sake of clearness, we  will emphasize the dependence on $d$ by writing $m^{*}_{d}$ in place of $m^{*}$.
\begin{corollary}\label{cor:expdstar}
Let $m^*_{d}=\overline{m}(x)\chi_{\omega^*_{d}}$ the solution to the maximum problem \eqref{eq:defmax}.  Then the following conclusions hold. 
\begin{enumerate}
\item
There exists $m_\infty\in \mathcal M$ such that
\[
m^*_{d} \overset{\ast}{\rightharpoonup} m_{\infty}, \quad \text{in $L^{\infty}(\Omega)$}.
\]
\item
Denoting with $u^*_d $ and $p^*_d$ the unique solution to \eqref{P} and \eqref{adjapp} associated with $m^*_{d}$,  it results
\[
\begin{split}
\lim_{d\to+\infty}d(u^*_{d}-U_{\infty})&=\eta_{\infty} \quad  \hbox{ in }H^1(\Omega),
\\
\lim_{d\to+\infty}d(p^*_{d}-p_{\infty})&=w_{\infty} \quad  \hbox{ in }H^1(\Omega),
\end{split}
\]
where $p_\infty=\frac{|\Omega|}{\int_\Omega s(x)}$,  $\eta_{\infty}$ and $w_{\infty}$ are solutions to \eqref{eq:eta},  \eqref{2} and to \eqref{eq:w}, \eqref{2-bis} corresponding to $m_\infty$ respectively.
\item
The function $\varphi^*_d=u^*_d p^*_d$ converges to $ \varphi_\infty:=p_{\infty}U_\infty$ in $H^{1}(\Omega)$ and
\[
\lim_{d\to+\infty}d(\varphi^*_{d}-p_{\infty}U_\infty)=z_{\infty}:=p_\infty \eta_{\infty}+U_\infty w_{\infty}
\quad \hbox{ in }H^1(\Omega).
\]
\end{enumerate}
\end{corollary}
\begin{proof}
Since $m^*_{d}$ is uniformly bounded in $L^\infty(\Omega)$, there exists 
$m_\infty\in L^\infty(\Omega)$ such that, up to a subsequence,  $m^{*}_{d} 
\overset{\ast}{\rightharpoonup} m_{\infty}$ in $L^{\infty}(\Omega)$. The fact 
that $m_\infty\in \mathcal M$ is a consequence of the weak-star convergence. This 
convergence is enough to repeat the proof of Proposition \ref{pro:dinfinito}, as $u^{*}_{d}\to U_{\infty}$ strongly in $L^{2}(\Omega)$. This crucial property also
holds for $z_{d}$ and $\eta_{d}$ in Lemma \ref{le:expd}, for $\widetilde{w}_{d}$ and $p_{d}$ in Lemma \ref{le:expdstar}. 

It is then easy to pass to the limit in the weak formulations  of the equations under consideration to obtain the desired conclusions.
\end{proof}

\begin{remark}
The same result holds also for the solution $(m_{d})_*$ of the minimum problem \eqref{eq:defmax}.
\end{remark}
\section{Proof of Theorem \ref{teo-concentration}}\label{se:thm13}
This section is devoted to the proof of Theorem \ref{teo-concentration}, namely we will show that, in dimension one, the positivity set of 
$m^{*}$ is an interval located at one of the extrema of $\Omega$.
We first obtain some preliminary results which might be of independent interest in the general framework, both respect to the dimension and to the assumption on $s$ and $m$. 

\begin{lemma}\label{le:J1convex}
Let $\eta$  be the solution to \eqref{eq:eta} that satisfies \eqref{2}. Then,
the functional $J_{1}$ defined in \eqref{eq:defJ1} is convex.
\end{lemma}

\begin{proof}
Let us decompose $\eta$ by means on the projection on the eigenspace associated 
with the zero eigenvalue, namely the constant functions, and the set of
functions with zero average. 
That is, let  $\eta=\widetilde \eta+\beta$ with $\beta\in \R$ and $\widetilde \eta\in H^1(\Omega)$ with zero mean.  With this choice $\widetilde \eta$ is the unique solution to \eqref{eq:eta}, namely
\begin{equation}\label{eq:eta-hat}
\begin{cases}
-\Delta \widetilde\eta= U_{\infty}\left(s(x)\frac{M}{\into s}-m\right)
\\
\frac{\partial \widetilde\eta}{\partial \nu}=0\\
\int_\Omega \widetilde \eta =0.
\end{cases}
\end{equation}
In order to prove that $J_{1}$ is convex, we take 
 $m\in \mathcal M$ and $g$ an admissible perturbation for $m$.  We want to show that $\beta$ is twice Gateaux-differentiable in the direction $g$ and that it holds $\ddot{\beta}>0$ whenever $g\neq 0$.
\\
{\bf Step 1.} 
Let $m\in \mathcal M$ and $g$ be an admissible perturbation for $m$. 
Let $\widetilde\eta_\e$ be the solution to \eqref{eq:eta-hat} corresponding to $m_\e(x)=m(x)+\e g$. 
Let us first prove that the following expansion holds
\begin{equation}\label{eq:etaexp}
\widetilde\eta_{\eps}=\widetilde\eta +\eps \dot{\widetilde{\eta}}+o(\eps^{2}),
\end{equation}
where $\dot{\widetilde{\eta}}$ is a solution to the problem
\begin{equation}\label{eq:etapunto}
\begin{cases}
-\Delta \dot{\widetilde{\eta}}=-U_{\infty}g &\text{in $\Omega$}
\\
\into \dot{\widetilde{\eta}}=0
\\
\partial_{\nu }\dot{\widetilde{\eta}}=0 &\text{on $\partial \Omega$}.
\end{cases}
\end{equation}
Arguing as in the proof of Theorem \ref{teo-expansion} and in view of \eqref{eq:eta-hat}, we first deduce that
$\widetilde\eta_{\eps}\to \widetilde\eta $ strongly in $H^{1}(\Omega)$ as $\eps\to 0$.
Let us now consider $w_{\eps}:=\frac{\widetilde\eta_{\e}-\widetilde\eta }{\eps}$. The function
$w_{\e}$ belongs to $\widetilde X$ and solves
\[
-\Delta w_{\e}=-U_{\infty}g.
\]
 Using the Poincaré-Wirtinger inequality we have
\[c\int_\Omega w_{\e} ^2\leq \int_\Omega |\nabla w_{\e}|^2=-U_\infty \int_\Omega  gw_{\e} \leq C\|w_{\e}\|_2
\]
then it is easy to obtain that $w_{\e}$ strongly converges in $H^{1}(\Omega)$
to $\dot{\widetilde{\eta}}$, which solves \eqref{eq:etapunto}.
The conclusion follows by observing that \eqref{eq:etapunto} implies that
$\ddot{\widetilde\eta}\equiv0$.
\\
{\bf Step 2.}
Let us first observe that \eqref{2} implies that $\beta$ is given by
\begin{equation}\label{eq:defbeta}
\beta=\dfrac1{\into s(x)}\left[\left(\dfrac{M}{\into s}-1\right)\into s \widetilde{\eta}-\into m\widetilde{\eta}\right].
\end{equation}
Let us now consider the incremental quotient of 
$\beta$. It results
\[
\frac{\beta_{\e}-\beta}{\e}=\frac1{\into s}\left[\left(\dfrac{M}{\into s}-1\right)\into s \frac{\widetilde{\eta}_{\e}-\widetilde{\eta} }\e -\into m\frac{\widetilde{\eta}_{\e}-\widetilde{\eta} }\e- \into g\widetilde{\eta}_{\e}\right].
\]
Then, taking into consideration what we proved in Step 1, we obtain
\[
\dot{\beta}=-\frac1{\into s }\left\{-\left(\dfrac{M}{\into s}-1\right) \into s\dot{\widetilde \eta} +\into m\dot{\widetilde \eta}+g\widetilde{\eta}\right\}.
\]
Let us now compute $\ddot{\beta}$. It results
\[
\begin{split}
\frac{\dot\beta_{\e}-\dot\beta}{\e} &=-\frac1{\into s }
\left\{-\left(\dfrac{M}{\into s}-1\right) \into s \frac{\dot{\widetilde{\eta}}_{\e}-\dot{\widetilde{\eta}}}{\e}+
\into m \frac{\dot{\widetilde{\eta}}_{\e}-\dot{\widetilde{\eta}}}{\e}+
\into g \dot{\widetilde{\eta}}_{\e}+\into g \frac{\widetilde{\eta}_{\e}-\widetilde{\eta}}{\e} \right\}.
\end{split}\]
Passing to the limit as $\eps\to0$ and observing that 
$\dot{\widetilde{\eta}}_{\e}\to \dot{\widetilde{\eta}}$ strongly in $L^{1}(\Omega)$,
we deduce that
\[
\ddot{\beta}=-\frac2{\into s } \into g \dot{\widetilde{\eta}}.
\]
Finally, taking $\dot{\widetilde{\eta}}$ as test function in \eqref{eq:etapunto} 
yields
\[
\ddot{\beta}=-\frac2{\into s} \into g \dot{\widetilde{\eta}}=\frac2{U_{\infty}\into s }\into |\nabla \dot{\widetilde{\eta}}|^{2}>0
\]
whenever $g\neq 0$, implying the desired conclusion.
\end{proof}

\begin{theorem}\label{teo2MNP}
Let $m^{*}_{d}=\overline{m}(x)\chi_{\omega^*_{d}}$ the solution to the maximum problem \eqref{eq:defmax}.  Then 
\[
\lim_{d\to +\infty}m^{*}_{d}=\overline{m}(x)\chi_{\omega_{\infty}}, \quad \text{in $L^{p}(\Omega)$ for every $p\in [1,+\infty)$}
\]
where $\omega_{\infty}$ is  a measurable set such that either
\[\omega_\infty=\{x\in \Omega\setminus \Omega': z_\infty(x)< t\}  \ \ \ \text{ or } \ \omega_\infty=\{x\in \Omega\setminus \Omega': z_\infty(x)\leq t\}\]
 where $t=\inf_{\omega_\infty} z_\infty(x)$ and  $z_\infty$ is introduced in conclusion (3) of Corollary \ref{cor:expdstar}. 
\end{theorem}
\begin{proof}
Since   $m^{*}_{d}\in \Mcal$ there exists $m_{\infty}$  such that, up to a subsequence,
\[
m^{*}_{d} \overset{\ast}{\rightharpoonup} m_{\infty}, \quad \text{weak-$*$ in $L^{\infty}$}.
\]
Moreover, for every $m\in \Mcal$ it holds, also using Lemma \ref{le:expd},
\[
d\left(J_{d}(m^{*}_{d})-\int_\Omega U_\infty\right)\geq  d\left(J_{d}(m)-\int_\Omega U_\infty\right)=J_{1}(m)+O\left(\frac1d\right)
\]
where $J_{1}$ is the functional defined in \eqref{eq:defJ1}.
Then, passing to the limit in the above inequality and applying  Corollary \ref{cor:expdstar}
we obtain that $m_{\infty}$ is a maximizer of $J_{1}$ on $\Mcal$.
Namely, 
\begin{equation}\label{eq:minftymax}
J_{1}(m_{\infty})=\into \eta_{\infty}\geq \into \eta=J_{1}(m), \quad \text{for every $m\in \Mcal$,}
\end{equation}
as $\eta_{\infty}$ is  the unique solution to \eqref{eq:eta} and \eqref{2}, associated with $m_{\infty}$.
Then, Lemma \ref{le:J1convex} implies that $m_\infty$ is an extremals of the class $\Mcal$, which, in turn, yields that $m_\infty=\overline{m}\chi_{\omega_\infty}$
for some $\omega_\infty\subset \Omega\setminus \Omega'$ (see  \cite[Proposition 7.2.17]{HP}).
Then \cite[Proposition 2.2.1]{HP}
yields $m_{d}\to m_{\infty}$ strongly in $L^{p}(\Omega)$ for every $p\in [1,+\infty)$.
 Using the adjoint state $\widetilde w$, which is the unique the solution to \eqref{eq:w} with zero average, and the equation for $\dot{\widetilde{\eta}}$, it can be derived that 
\[\dot \beta=-\frac 1{|\Omega|}\into \widetilde z_{\infty}g
\]
where $\widetilde z_{\infty}$ is the unique solution to \eqref{eq:zinfty} with zero average.  Then, the conclusion follows reasoning as in Proposition \ref{prop:11.2} and the proof of Theorem \ref{thm:mbang} supposing by contradiction that the sets $\{x\in \Omega\setminus \Omega '  : \widetilde z_{\infty}(x)= \widetilde t\}\cap \omega_\infty$ and $\{x\in \Omega\setminus \Omega '  : \widetilde z_{\infty}(x)= \widetilde t\} \setminus\omega_\infty$ have positive measure.
\end{proof}

 \begin{remark} For the solution $(m_d)_*$ of the minimum problem \eqref{eq:defmax} 
we can obtain that,  up to a subsequence,
$(m_d)_{*} \overset{\ast}{\rightharpoonup} (m_{\infty})_*$,  weak-$*$ in $L^{\infty}$ and $(m_{\infty})_*$ is a minimizer of $J_1$ on $\Mcal$.
\end{remark}

\begin{lemma}
Assume  
\begin{equation}\label{ipo:sconst}
s(x)\equiv s\in (0,+\infty).
\end{equation}
Then, $\overline{m}(x)\chi_{\omega_{\infty}}$ is a solution to the problem
\begin{equation}\label{eq:nuovoprob}
\min_{\Mcal}\min_{  \widetilde{X}}\widetilde\Ecal
\end{equation}
where  $\widetilde\Ecal :  \widetilde{X}:=\{v\in H^1(\Omega): \int _\Omega v=0\} \mapsto \R$ is defined by
\begin{equation}\label{eq:defEt}
\widetilde\Ecal(v):=\frac 12 \int_\Omega |\nabla v|^2+U_\infty\int_\Omega m(x) v .
\end{equation}	
\end{lemma}
\begin{proof}
Note that, in this case \eqref{2} reduces to
\beq\label{eq:22}
s\int_\Omega  \eta =\frac M{|\Omega|}\int_\Omega \eta-\int_\Omega m(x) \eta.
\eeq
Next, using $\eta$ as test function in \eqref{eq:eta}, and using \eqref{eq:22},  we have
\beq\label{1}
\int_\Omega |\nabla \eta |^2=U_\infty \left(\frac M{ |\Omega| }\int_\Omega \eta-\int_\Omega m(x) \eta\right)=U_\infty s\int_\Omega  \eta. 
\eeq
As a consequence,
\beq\label{3}
J_1(m)=\int_\Omega \eta=\frac 1 {s U_\infty}\int _\Omega |\nabla \eta |^2, \quad \text{for every $m\in \Mcal$.}
\eeq
As done in Lemma \ref{le:J1convex} we use the decomposition $\eta=\widetilde \eta+\beta$ with $\beta\in \R$ and $\widetilde  \eta\in \widetilde{X}:=\{v\in H^{1}(\Omega) : \into v=0\}$.  With this choice $\widetilde \eta$ is the unique solution to \eqref{eq:eta-hat} and, by \eqref{3} the constant $\beta$ is given by
\beq
\beta= \frac 1 {sU_\infty |\Omega|}\int _\Omega |\nabla \widetilde \eta |^2.
\eeq
Once one settles the problem \eqref{eq:eta} in $\widetilde{X}$, it is standard to 
see  that $\widetilde \eta$ minimizes the energy $\widetilde\Ecal(w)$ among all the function $w\in \widetilde{X}$. 
Taking as test function $\widetilde{\eta}$ in \eqref{eq:eta-hat} we deduce 
\begin{equation}\label{eq:livelloeta}
\int_\Omega  |\nabla \widetilde\eta |^2=-U_\infty\int_\Omega m (x) \widetilde \eta 
\end{equation}
which implies that 
\[\widetilde\Ecal(\widetilde \eta)=-\frac 12 \int_\Omega  |\nabla \widetilde\eta |^2=\frac12\into U_{\infty}m(x)\widetilde\eta.\]
Then, in view of \eqref{3}
\beq
J_1(m )=\frac 1{sU_\infty}\int_\Omega |\nabla \widetilde\eta |^2=\frac 1{sU_\infty}\left( -2 \widetilde\Ecal(\widetilde \eta)\right)=-\frac 2{sU_\infty} \min _{w\in \widetilde{X}}\widetilde\Ecal(w).
\eeq
In order to obtain \eqref{eq:nuovoprob}, we recall that
 $m_{\infty}$ is a maximizer of $J_{1}$ on $\Mcal$, so that
\[
\begin{split}
{J_{1}(m_{\infty})=\max_{\Mcal }J_{1}(m)=\max_{\Mcal}\left[-\frac{2}{sU_{\infty}}\min _{w\in \widetilde{X}}\widetilde\Ecal(w)\right]
=-\frac2{sU_\infty} \min_{\Mcal}\min _{w\in \widetilde{X}}\widetilde\Ecal(w),}
\end{split}\]
yielding the conclusion.

\end{proof}

From now on we focus our attention to the one-dimensional case.

\begin{corollary}\label{cor:10.4}
Let $\Omega=(0,1)$. Assume \eqref{ipo:sconst} and $\overline{m}\equiv 1$. Then,  $ \chi_{\omega_{\infty}}$ introduced in Theorem \ref{teo2MNP} is either $\chi_{(0,M)}$ or $\chi_{(1-M,1)}$.
\end{corollary}
\begin{proof} 
In the following argument, to simplify the notation,  we will write  $\eta$ in the place of  $\eta_\infty$.  Let us observe that
\[
\begin{split}
\min_{\Mcal}\min_{u\in \tilde{X}}\widetilde{\Ecal}
&=\widetilde{\Ecal}(\widetilde{\eta})=\frac12\int_{0}^{1}| \widetilde{\eta}'|^{2}+U_{\infty}\int_{0}^{1} m_{\infty}(x)\widetilde{\eta}(x)
\\
&=
\frac12\int_{0}^{1}| \widetilde{\eta}'|^{2}-U_{\infty}\int_{0}^{1} (1-m_{\infty}(x))\widetilde{\eta}(x),
\end{split}\]
as $\widetilde{\eta}$ has zero mean.
Then, denoting with $f_{d}$  the monotone decreasing rearrangement of $f$ and applying the Hardy-Littlewood and Polya inequalities (\cite{Kaw}) we obtain
\[
\begin{split}
\widetilde{\Ecal}(\widetilde{\eta})
&\geq  \frac 12\int_{0}^{1}| (\widetilde{\eta}_{d})'|^{2}-U_{\infty}\int_{0}^{1} (1-m_{\infty}(x))_{d}\widetilde{\eta}_{d}(x)
\\
&=\frac 12 \int_{0}^{1}| (\widetilde{\eta}_{d})'|^{2}+U_{\infty}\int_{0}^{1} 
(m_{\infty}(x))_{i}\widetilde{\eta}_{d}(x)
\end{split}
\]
where $f_{i}$ denotes the monotone increasing rearrangement of $f$.
As $(m_{\infty}(x))_{i}\in \Mcal$ we deduce
\[
\begin{split}
\min_{\Mcal}\min_{u\in \tilde{X}}\widetilde{\Ecal}
&=\widetilde{\Ecal}(\widetilde{\eta})=\frac12\int_{0}^{1}| \widetilde{\eta}'|^{2}+U_{\infty}\int_{0}^{1} m_{\infty}(x)\widetilde{\eta}(x)
\\
&\geq \frac 12
\int_{0}^{1}| (\widetilde{\eta}_{d})'|^{2}+U_{\infty}\int_{0}^{1} 
(m_{\infty}(x))_{i}\widetilde{\eta}_{d}(x)
\geq \min_{\Mcal}\min_{u\in \tilde{X}}\widetilde{\Ecal}.
\end{split}
\]
As a consequence, all the above inequalities are equalities and by exploiting 
the equality case in the Polya inequality we deduce that $\widetilde\eta $ is monotone. 
Let us now observe that, under the assumption \eqref{ipo:sconst}, $p_{\infty}$ introduced in Lemma \ref{le:expdstar} is given by $p_{\infty}=1/s$. So that
 the function $w_{\infty}$ introduced in Corollary \eqref{cor:expdstar} is the unique solution
to 
\[
\begin{cases}
-\Delta w_{\infty}=\frac1s\left(\frac{M}{|\Omega|}-m\right)
\\
\frac{\partial w_{\infty}}{\partial \nu}=0
\end{cases}
\]
satisfying \eqref{2-bis}.
Then, recalling \eqref{eq:eta}, we obtain that $w_{\infty}=\frac1{U_{\infty}s}\eta_\infty+c_{0}$ where $c_{0}$ can be obtained exploiting \eqref{2} and \eqref{2-bis}.
As a consequence, $z_{\infty}$ introduced in conclusion {\it (3)} of Corollary \ref{cor:expdstar} is given by $z_{\infty}=p_{\infty}\eta_{\infty}+U_{\infty}w_{\infty} =\frac2s\eta_{\infty}+U_{\infty}c_{0}$. So that,  $\omega_{\infty}$ turns out to be a sublevel set of $\eta_{\infty}$ as well. 
In addition, by applying \cite[Theorem 1.1]{louc} we deduce that the level sets of $\widetilde{\eta}_{\infty}$ have zero measure as $M<1$. As a consequence,  $m_{\infty}(x)=\chi_{\omega_{\infty}}$ with 
$\omega_{\infty}$ an open  sub-level set of $\widetilde{\eta}_\infty$.
Then, the conclusion follows from the monotonicity property of $\eta_\infty$.
 \end{proof}
\begin{remark}\label{rmk:dim2}
We expect that the above results can be extended to the case of  $\Omega=(0,1)^{2}$ as done in \cite{manaprjmpa}. This may suggest 
the right location, at least for $d$ sufficiently large, of the optimal set.
This would be in accordance to the numerical study represented in Figure
\ref{fig:gompertz_2D_max}.
\end{remark}
We are now in the position to prove Theorem \ref{teo-concentration}.
\begin{proof}[Proof of theorem \ref{teo-concentration}]
By applying Lemma \ref{teo2MNP}, we deduce that, up to a subsequence, $m^*_d\to \chi_{\omega_\infty}$ in 
$L^1(\Omega)$, where $\omega_\infty=\{x\in (0,1): \eta_\infty<t\}$ for a suitable value of $t$, and,  Corollary \ref{cor:10.4} yields that  either $\omega_\infty=(0,M)$ or $\omega_\infty=(1-M,1)$.

In view of Corollary \ref{cor:expdstar},  we have that  $z_d:=d(\varphi^*_{d}-p_{\infty}U_\infty)\to z_\infty$ in $H^1(\Omega)$ and $z_\infty$ satisfies 
\[
\begin{cases}
-z_\infty ''=U_\infty \frac 2 {s}\left( M-\chi_{\omega_\infty}\right) & \text{in $(0,1)$}
\\
z'_{\infty}(0)=0, \, z'_{\infty}(1)=0.
\end{cases}\]
Let us assume that $\omega_\infty=(0,M)$. Then, since $M<1$ (by the definition of $\mathcal M$) we have that $z_\infty$ is strictly convex in $(0,M)$ and strictly concave in $(M,1)$. 

Then, the boundary conditions $z'_\infty(0)=z'_\infty(1)=0$ implies that $z'_\infty>0$ in $(0,1)$.
According to Theorem \ref{thm:mbang} we know that $m^*_d=\chi_{\omega^{*}_d}$ where 
\[
\{x\in (0,1): \varphi^*_d<t_d\}\subseteq \omega^{*}_d\subseteq \{x\in (0,1): \varphi^*_d<t_d\}\cup A_d
\]
where $t_d$ is a suitable value and $A_d= \{x\in (0,1): \varphi^*_d=t_d\}$.

Using the notation $\mu_d:=d\left(t_d-\frac {U_\infty}{s}\right)$, we can then rewrite $\omega^{*}_d$ in terms of $z_d$ and $\mu_d$ as follows 
\[
\{x\in (0,1): z_d<\mu_d\}\subseteq \omega^{*}_d\subseteq \{x\in (0,1): z_d<\mu_d\}\cup A_d, \quad \text{where $A_d\subseteq \{x\in (0,1): z_d=\mu_d\}$.}
\] 
Corollary \ref{cor:expdstar}, and Sobolev embedding yield that $z_d\to z_\infty $ 
in $C^{0,\alpha}(0,1)$ for $\alpha<\frac 12$. 
This implies that $\mu_d$ is bounded and, up to a subsequence, $\mu_d\to \mu_\infty$. 
The strict monotonicity of $z_\infty$ implies that there exists a unique point $x_\infty\in (0,1)$ such that $z_\infty(x_\infty)=\mu_\infty$.
First we observe that $|A_d|\to 0$ as $d\to \infty$ since $A_d\to \{x\in (0,1) : z_\infty=\mu_\infty\}=\{x_\infty\}$ by the strict monotonicity of $z_\infty$.
In addition,  $\{x\in (0,1): z_d<\mu_d\}\to \{x\in (0,1): z_\infty\le\mu_\infty\}
=(0,x_\infty]$ and $\{x\in (0,1): z_d>\mu_d\}\to \{x\in (0,1): z_\infty\ge\mu_\infty\}
=[x_\infty,1)$. Moreover, since $|\omega^{*}_d|=M$ for any $d$, then $x_\infty=M$.

Therefore, as $z_\infty$ is increasing, we have $z_\infty(0)<\mu_\infty$ and $z_\infty(1)>\mu_\infty$ and hence, by the uniform convergence of $z_d$
\[z_d<\mu_d \ \text{ in }(0,\e) \ \ ,  z_d>\mu_d \ \text{ in }(1-\e,1) \]
for a suitable value of $\e>0$ when $d$ is large enough.

In order to conclude the proof, let us observe that repeating the argument of Lemma \ref{le:expd3}, Corollary \ref{cor:expdstar} and using elliptic regularity estimates and Sobolev inequalities we obtain that $z_d\to z_\infty$ in $C^{1,\alpha}$ for some $\alpha>0$, so that
\[z'_d>0 \text{ in }(\e,1-\e).\]
This implies the existence of $x_d\in (0,1)$ such that $\{x\in (0,1): z_d<\mu_d\}=(0,x_d)$ and 
$\{x\in (0,1): z_d>\mu_d\}=(x_d,1)$.  
Concluding, the monotonicity of $z_d$ in $(\e,1-\e)$ also implies that $|A_d|=0$
and, as $|\omega^{*}_d|=M$ it results $x_d=M$ completing the proof.
\end{proof}
\begin{corollary}\label{cor:teoest}
Assume $\overline{m}= \chi_{\omega}$ for some $\omega\subset \Omega$ such that $|\omega|>M$. Then, the following conclusions hold for $d$ sufficiently large.
\\
(i) If $(0,M)\subset \omega$ (or if $(1-M,1)\subset \omega$) then the optimal weight $m^*_d$ for $J_{d}(m)$ is equal to $\chi_{(0,M)}$ (or to $\chi_{(1-M,1)}$).
\\
(ii) If both $(0,M)$ and $(1-M,1)$ are subsets of  $\omega$   then, both   $\chi_{(0,M)}$ and  $\chi_{(1-M,1)}$ are maximizers.
\\
(iii) If there exists at least two disjoint subsets 
$\omega_{1},\, \omega_{2}$ such that $\omega=\omega_{1}\cup\omega_{2}$ and 
$M>\max\{|\omega_1|, |\omega_2|\}$ then $\omega^{*}$ is not connected, while 
$\omega^{*}$ is connected if  one between $(0,M)$ or $(1-M,1)$ is contained in either $\omega_1$ or $\omega_2$.
\end{corollary}
\begin{proof}
The first conclusion can be proved by observing that $\chi_{(0,M)}\in \Mcal$ and it holds
\[\max_{m\in \Mcal} J(m)\leq \max_{m\in \Mcal '} J(m)=J(\chi_{(0,M)})\]
for
$
\Mcal':=\{  0\leq m(x)\leq 1, \text{ a.e.  in } \Omega \text{ and }\int_\Omega m(x)\, dx= M\}.$
The same comparison argument can be exploited to show the remaining conclusions.
\end{proof}
 \begin{remark} \label{rem:teoest}
Let us observe that the same result holds assuming  $\overline{m}(x)\equiv\sup \overline{m}=:c$ in 
$(0,\frac{M}c)$ or in  $(\frac{1-M}c,1)$.
\end{remark}
\begin{remark}\label{rmk:dsmall}
Regarding the fragmentation phenomenon observed for $d$ small in \cite{manaprjmpa, mabalet, heokim}, as already mentioned in the introduction, this does not seem to occur in our framework as shown in in Section \ref{se:num_results}  via numerical simulations in the one dimensional case. In Figure \ref{fig:s_and_J_max_1D} (top right hand side) it is represented the behaviour of the cost functional with respect to $d$, for different choices of the growth rate function $s$. In all the cases, 
the numerical evidence shows that $d\mapsto J_{d}(m^{*})$ is decreasing for 
$d\in [10^{-8},1]$. This seems to prevent the fragmentation phenomenon in our
framework, as also highlighted in the graphs of the optimal control included in Figure \ref{fig:s_and_J_max_1D}.

\end{remark}

\section{Numerical results}\label{se:num_results}

In this section we present numerical simulations for the stationary Gompertz control problem. 
Most experiments are performed in one space dimension, with \(\Omega=(0,1)\), while a final example concerns the two-dimensional case \(\Omega=(0,1)^2\). Given a spatial growth rate \(s(x)\), a diffusion coefficient \(d>0\), a carrying capacity \(K>0\), and an admissible control \(m\), the associated state \(u=u_m\) is the unique positive solution to 
\[\begin{cases}
-d u''(x)=s(x)u(x)\ln\left(\frac{K}{u(x)}\right)-m(x)u(x),
& x\in\Omega,
\\
u'(0)=u'(1)=0.
\end{cases}\]

We address both the problems \eqref{eq:defmax} for $m\in \Mcal$ introduced in \eqref{M-set}.

In the full-domain simulations we take \(\overline{m}\equiv 1\), while in the localized-control simulations we impose a restricted admissible region \(\omega\subset\Omega\) by choosing
\[
\overline{m}(x)=\chi_\omega(x).
\]

The numerical experiments below are organized according to the main theoretical results of the previous sections. 
Figures~\ref{fig:gompertz_1D_min}--\ref{fig:gompertz_1D_i_ii} concern the minimization problem and illustrate Theorem~\ref{th:minscost}, Proposition~\ref{pro:minsnocostante}, and the open issues discussed in Remarks~\ref{rem:minlogus}--\ref{rem:min_s_const}. 
Figures~\ref{fig:s_and_J_max_1D}--\ref{fig:m_u_profile_max_1D_centered_rectangle} concern the maximization problem and are compared with the bang-bang property of Theorem~\ref{thm:mbang} and the large-diffusion characterization of Theorem~\ref{teo-concentration}. 
The final two-dimensional experiment illustrates the persistence of the concentration phenomenon beyond the one-dimensional setting.

Throughout this section, when no confusion is possible, we write \(J(\cdot)\) instead of \(J_d(\cdot)\).
\subsection{Numerical implementation}

The problem is discretized by finite differences on a uniform grid of \((0,1)\), using a Neumann discretization of the operator \(-\partial_{xx}\). The integral terms and the mass constraint are approximated by the trapezoidal rule. For each admissible control, the nonlinear stationary Gompertz equation is solved by a damped Newton method, while the reduced gradient is computed through the associated adjoint equation. The resulting finite-dimensional constrained optimization problem is solved by a projected quasi-Newton method of L-BFGS type. At each iteration, the tentative update is projected onto the admissible set defined by the box constraint and the discrete mass constraint. A nonmonotone Armijo line search is used to safeguard the descent or ascent step, depending on whether the minimization or maximization problem is considered (see \cite{Nocedal}).

To improve robustness with respect to possible local optima and to the dependence on the initial guess, the algorithm is run from several admissible initial controls, including warm starts from nearby diffusion values and selected structured initializations. For each value of \(d\), the reported solution corresponds to the best converged run among the tested initializations. Convergence is monitored through the projected gradient, or projected ascent, norm, together with the state residual and a posteriori optimality diagnostics based on the Karush--Kuhn--Tucker conditions.

For the sake of brevity, Algorithm~\ref{alg:minimization} reports only the procedure used for
the one-dimensional minimization problem. The maximization problem and the localized-control
variants are treated analogously, using the same finite-difference discretization, adjoint-based
gradient computation, and projection onto the corresponding admissible set.

\begin{algorithm}[H]
\caption{Numerical computation of a minimizer of \(J(m)\)}
\label{alg:minimization}
\begin{algorithmic}[1]
\Require Grid points \(x_i\), quadrature weights \(w_i\), diffusion coefficient \(d\), carrying capacity \(K\), growth profile \(s_i=s(x_i)\), budget \(M\), admissible region \(\omega\) if prescribed
\Ensure Approximate optimal control \(m_*\), state \(u_*\), adjoint state \(p_*\), and value \(J(m_*)\)

\State Construct the finite-difference matrix \(L_h\) associated with \(-\partial_{xx}\) and homogeneous Neumann boundary conditions.
\State Choose an admissible initial control \(m^0\) satisfying
\[
0\leq m_i^0\leq 1,\qquad \sum_i w_i m_i^0=M,
\]
and, in the localized case, \(m_i^0=0\) for \(x_i\notin\omega\).
\State Choose an initial state \(u^0\), using either the local approximation
\[
u_i^0=K\exp\left(-\frac{m_i^0}{s_i}\right)
\]
or a nearly constant approximation in the large-diffusion regime.

\For{\(k=0,1,2,\ldots\)}
\State Solve, by a damped Newton method, the discrete nonlinear state equation
\[
d(L_hu^k)_i
-
s_i u_i^k \ln\left(\frac{K}{u_i^k}\right)
+
m_i^k u_i^k
=0,
\qquad i=1,\ldots,N .
\]

    \State Compute the discretized cost  functional
    \[
    J(m^k)=\sum_i w_i u_i^k .
    \]

    \State Solve the discrete adjoint equation associated with the linearized state equation.

    \State Compute the reduced gradient \(g^k\) from the state and adjoint variables.

    \State Compute a projected quasi-Newton descent direction using an L-BFGS approximation.

   \State Perform a nonmonotone Armijo line search. For a trial step size \(\alpha>0\),
project the tentative update \(m^k+\alpha q^k\) onto the discrete admissible set by setting
\[
m_i^{k+1}
=
\min\{1,\max\{0,m_i^k+\alpha q_i^k-\lambda w_i\}\},
\]
where \(\lambda\) is chosen so that
\[
\sum_i w_i m_i^{k+1}=M.
\]
In the localized case, the same projection is performed only for \(x_i\in\omega\), while
\(m_i^{k+1}=0\) for \(x_i\notin\omega\).

    \State Update \(m^{k+1}\), \(u^{k+1}\), and the L-BFGS memory.

    \If{the projected gradient norm, state residual, and variation of \(J\) are below the prescribed tolerances}
        \State \textbf{stop}
    \EndIf
\EndFor

\State Set \(m_*=m^k\), \(u_*=u^k\), \(p_*=p^k\), and \(J(m_*)=J(m^k)\).
\end{algorithmic}
\end{algorithm}

\subsection{One-dimensional Gompertz minimization problem}

We first focus on the  minimization problem without imposing any additional localization constraint on the support of the control, namely $\overline{m}\equiv 1$ in the definition of $\Mcal$. This case provides a baseline for understanding how the diffusion coefficient \(d\) and the spatial growth profile \(s(x)\) influence the optimal reduction of the total population. Throughout this subsection we fix \(K=0.01\), $\overline{m}\equiv 1$ and \(M=0.5\), and we consider diffusion values logarithmically distributed in \([10^{-8},1]\), if not otherwise specified.

As shown in Figure~\ref{fig:gompertz_1D_min}, the qualitative behavior of the one-dimensional Gompertz minimization problem depends on both the diffusion coefficient $d$ and the spatial growth profile $s(x)$. In the left panel, we consider the two growth functions 
\[
s(x)= \frac{5}{18} , \qquad x\in(0,1),
\]
in the constant case, and
\[
s(x)= 
\begin{cases}
\frac{1}{3}, & 0 < x < \frac{1}{3},\\[1mm]
\frac{1}{4}, & \frac{1}{3} \le x < 1,
\end{cases}
\]
in the piecewise case. The right panel shows the optimal value $J(m_*)$ as a function of the diffusion parameter $d$, computed over a logarithmically spaced grid of $100$ values in $[10^{-8},1]$. 
In the constant case the optimal control \(m_*\) is constant and does not depend on \(d\), since we are in the setting of Theorem~\ref{th:minscost}; hence \(J(m_*)=\int_\Omega U_\infty=U_\infty\).  In the piecewise case $J(m_*)$ is smaller for low diffusion and increases monotonically with $d$, tending toward the value of the constant case as diffusion becomes larger as proved in Proposition \ref{pro:dinfinito}. 
This indicates that spatial heterogeneity has a nontrivial effect on the optimal outcome in the weak-diffusion regime, whereas this effect becomes less pronounced when diffusion is stronger.

\begin{figure}[H]
\centering
\includegraphics[width=0.45\textwidth]{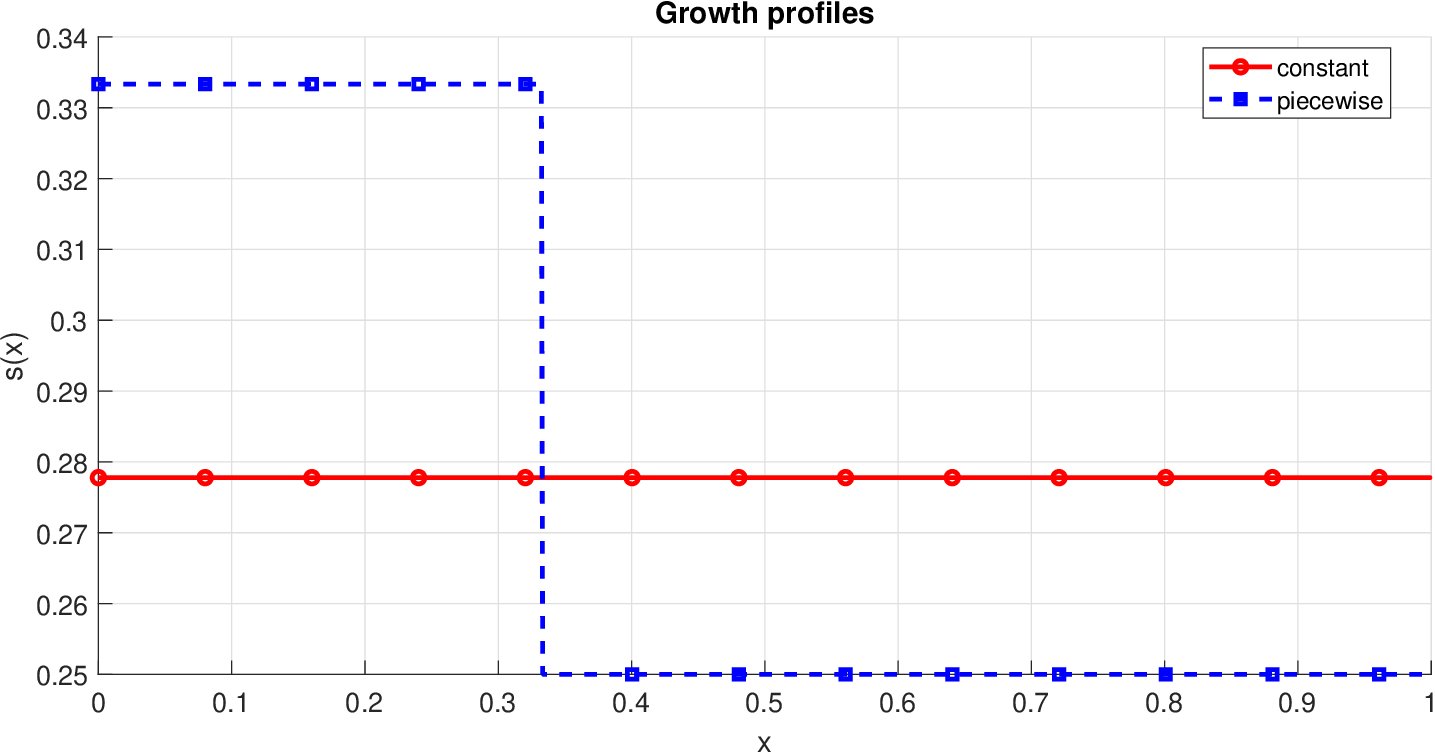} 
\includegraphics[width=0.45\textwidth]{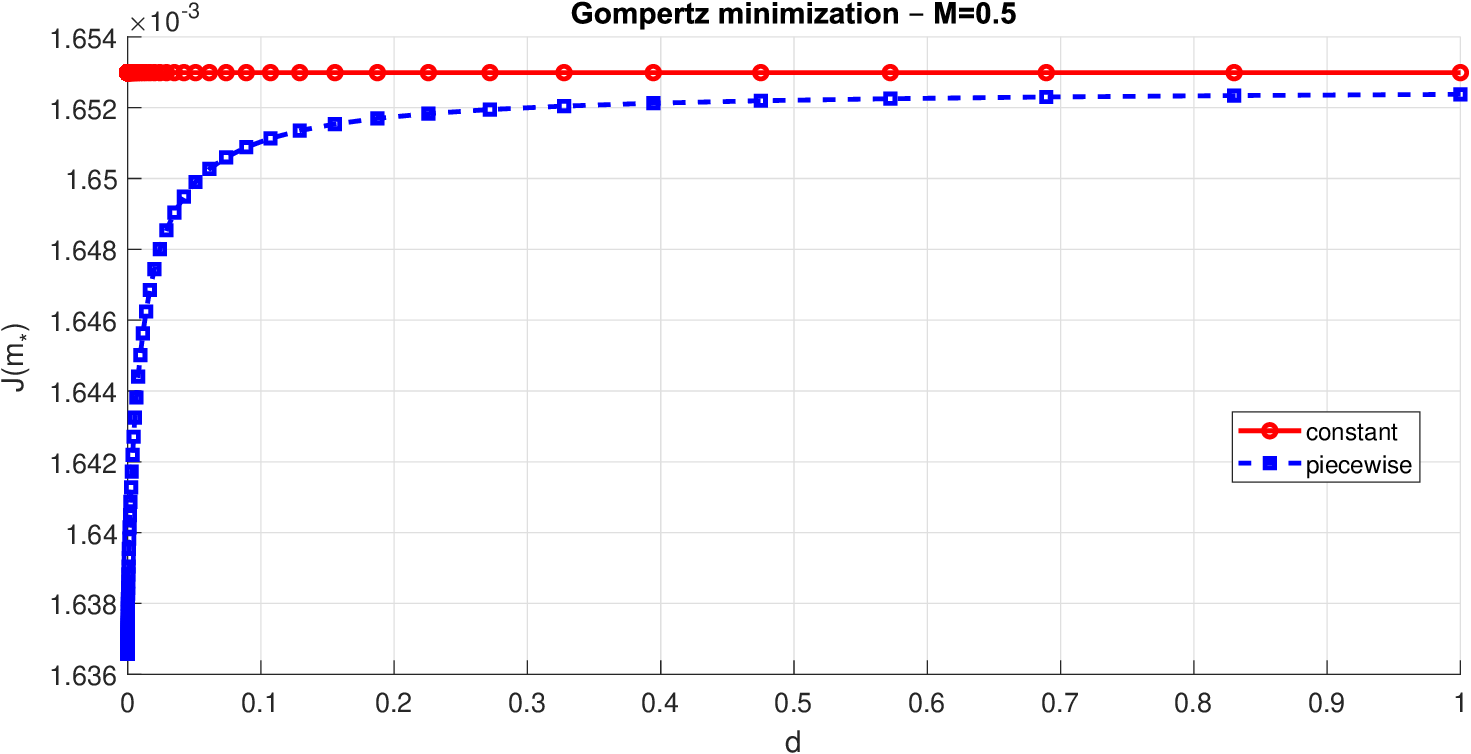} 

\caption{Left: growth profiles $s(x)$ corresponding to the constant and piecewise cases. Right: optimal value $J(m_*)$ as a function of the diffusion parameter $d$, computed for $d \in [10^{-8},1]$ on a logarithmically spaced grid with $100$ points, under the budget constraint $\int_\Omega m(x)\,dx = M$, with $M=0.5$ and $K=0.01$.}
\label{fig:gompertz_1D_min}
\end{figure}
Figure~\ref{fig:gompertz_1D_min_piecewise_heatmaps_profiles} displays the optimal control $m_*$ and the associated optimal state $u_*$ in the piecewise growth setting for $M=0.5$ and $K=0.01$. The heatmap of $m_*(x)$ shows a sharp transition located near the discontinuity point $x=\tfrac13$, separating a region where the control is larger on the interval $(0,\tfrac13)$ from a region where it is smaller on $[\tfrac13,1 )$. Moreover, this interface becomes increasingly steep as the diffusion coefficient $d$ decreases, indicating a stronger localization of the optimal strategy in the low-diffusion regime. The heatmap of $u_*(x)$ exhibits the complementary behavior: the state is higher on the left subinterval, where the growth rate is larger, and lower on the right subinterval, with a transition that again sharpens as $d \to 0$. The representative profiles confirm this trend, showing that for small diffusion both $m_*$ and $u_*$ approach piecewise-structured configurations aligned with the discontinuity of the growth profile, whereas for larger values of $d$ the profiles become smoother due to the homogenizing effect of diffusion.

\begin{figure}[H]
\centering
\includegraphics[width=0.45\textwidth]{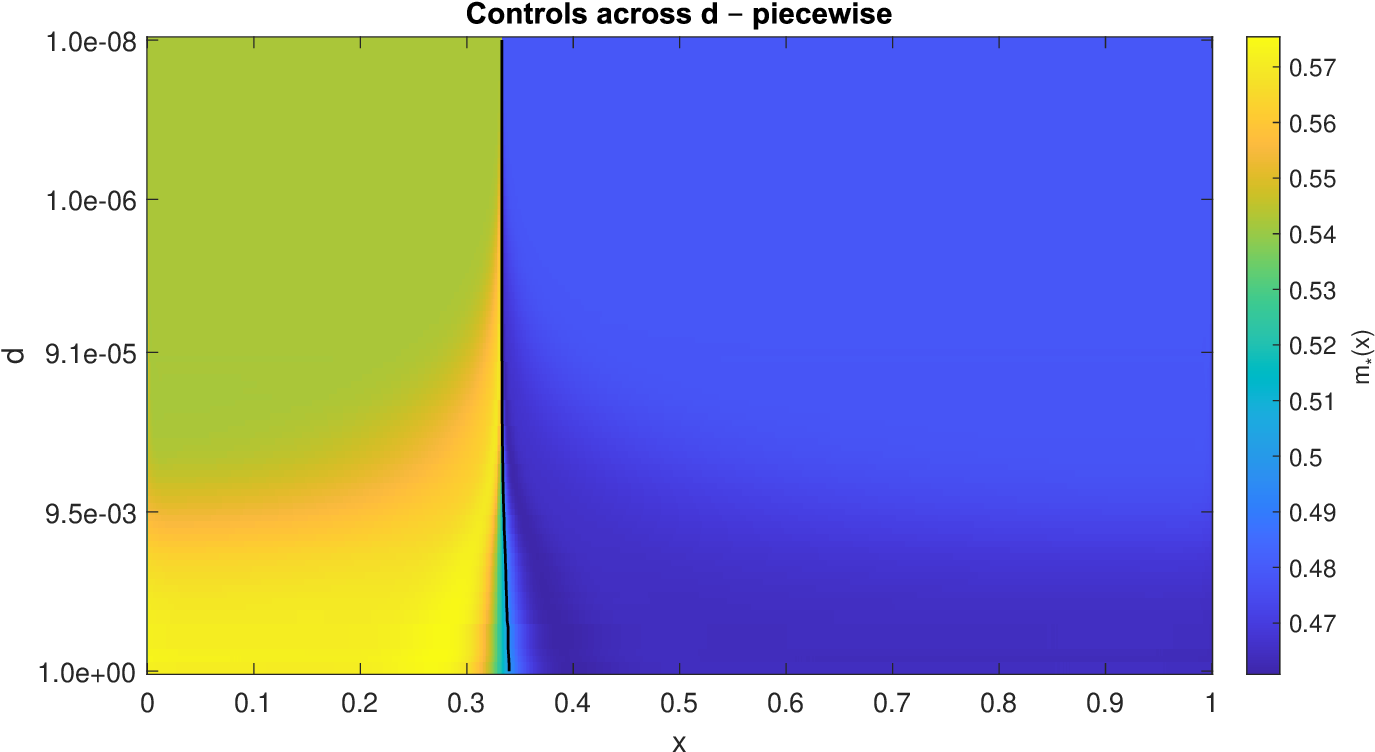}
\includegraphics[width=0.45\textwidth]{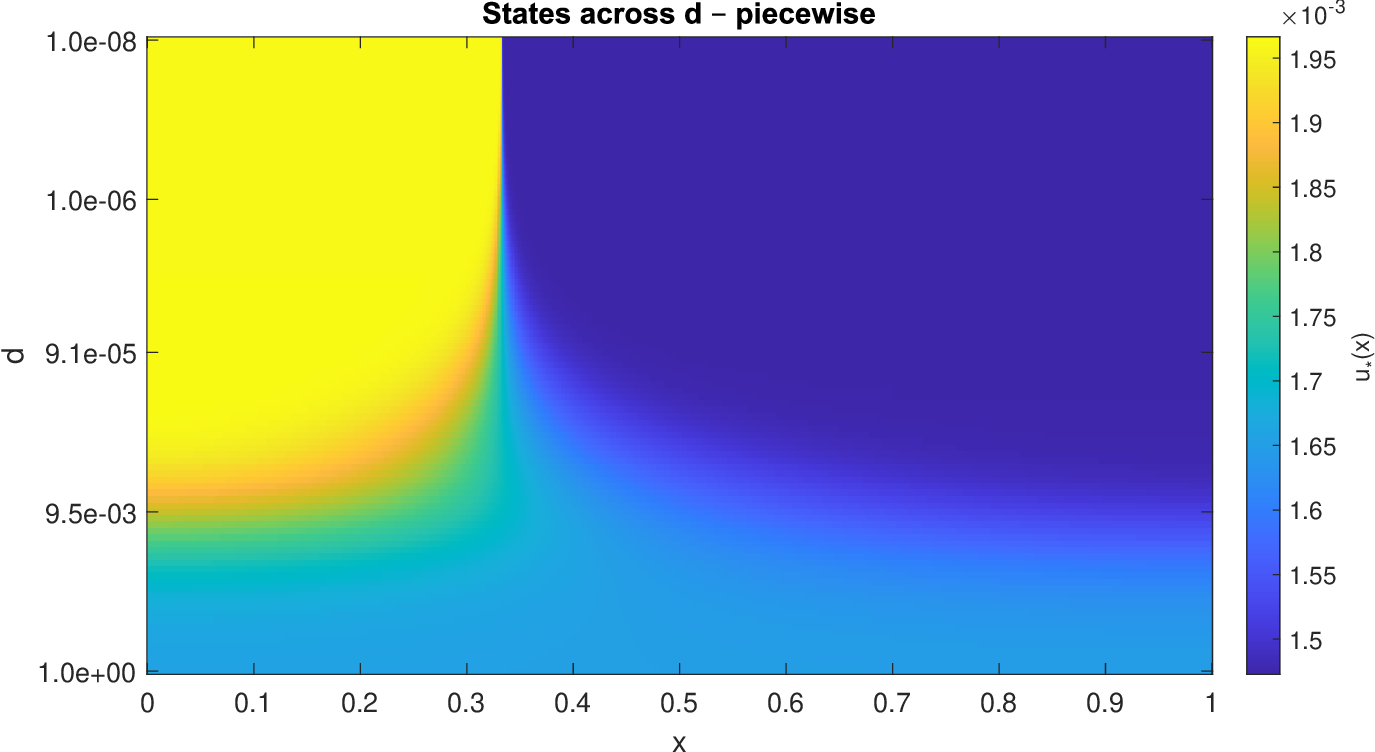}

\includegraphics[width=0.45\textwidth]{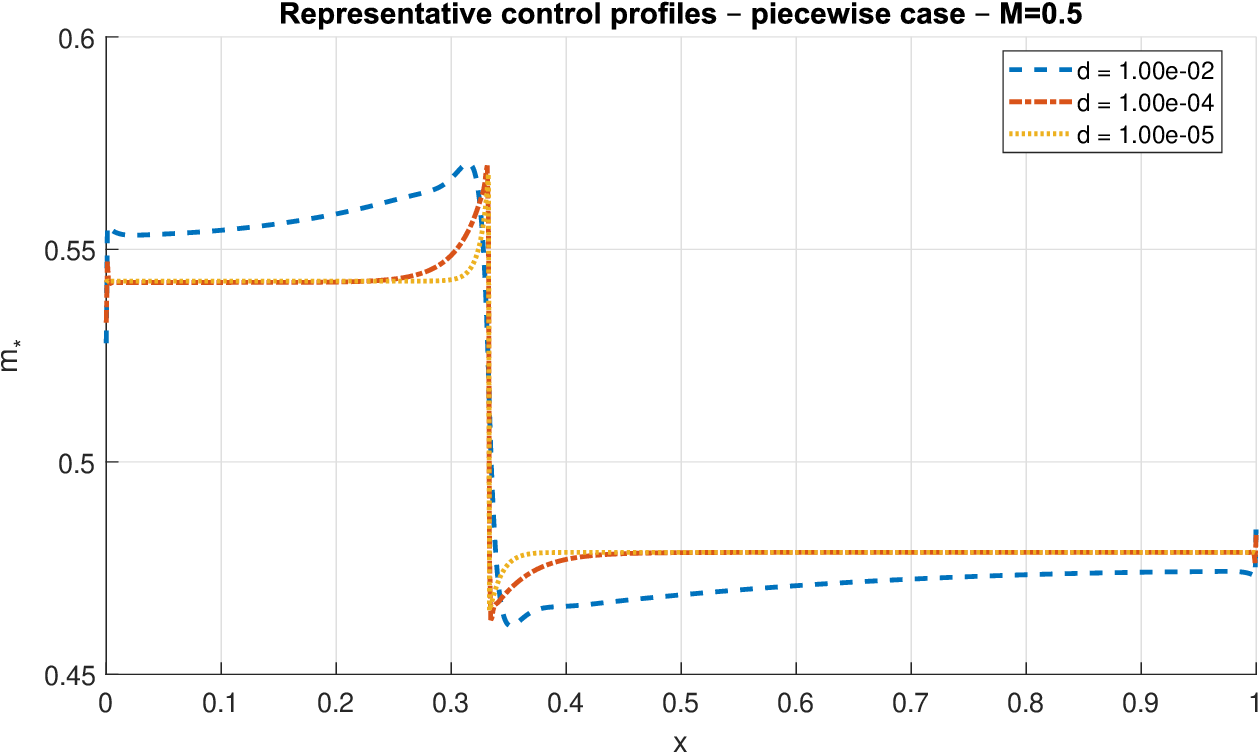}
\includegraphics[width=0.45\textwidth]{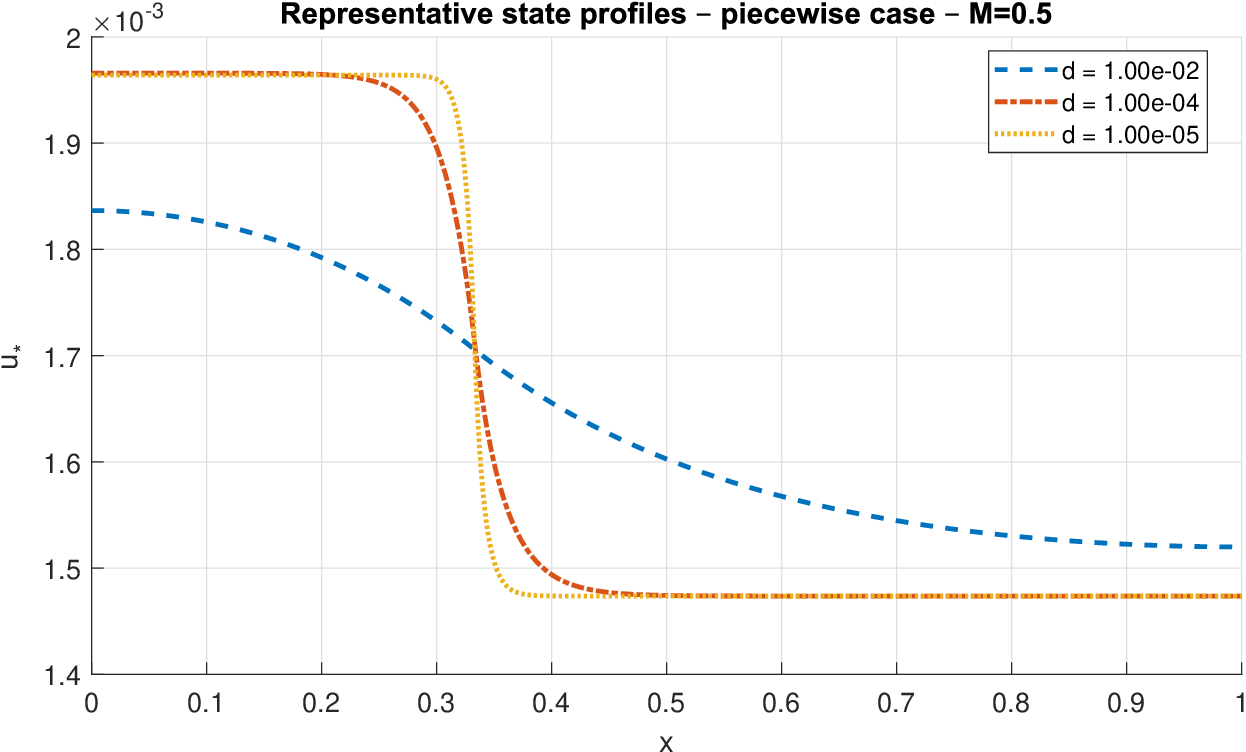}

\caption{Piecewise growth case with budget $M=0.5$ and $K=0.01$. Top left: distribution of the optimal control $m_*(x)$ over $(x,d)$; top right: distribution of the corresponding optimal state $u_*(x)$ over $(x,d)$. Bottom left: selected cross-sections of $m_*(x)$ for representative values of the diffusion parameter $d$; bottom right: the associated profiles of $u_*(x)$. The figure shows the dependence of both the optimal strategy and the resulting state on the diffusion intensity.}

\label{fig:gompertz_1D_min_piecewise_heatmaps_profiles}
\end{figure}

The previous figures describe the behavior of the minimization problem when the control is admissible on the whole interval, namely $\overline{m}\equiv 1$. We now complement this analysis by considering a localized admissible set, where the control is allowed to act only on a prescribed subinterval \(\omega\), namely $\overline{m}(x)=\chi_\omega(x)$.  This setting is useful to visualize how the mass constraint interacts with a restricted control region and how the corresponding first-order optimality conditions are reflected in the switching function $\varphi$, see \eqref{eq:switch} for the definition of $\varphi$.

Figure~\ref{fig:gompertz_1D_i_ii} illustrates the optimal profiles in the constant-growth case for two choices of the control budget  M, with the same admissible region
\(\omega=\left(\tfrac12,\tfrac34\right)\). As expected, the constraint \(m=0\) in \((0,1)\setminus\omega\) confines the optimal control to \(\omega\). Increasing the budget from \(M=0.03\) to \(M=0.2\) enlarges the portion of the admissible interval on which the control is active. The associated states \(u_*\) display the corresponding spatial variation induced by the optimal control. The last row reports the switching function \(\varphi=u_* p_*\), whose behavior characterizes the active and inactive regions of the optimal control through the first-order optimality conditions.

\begin{figure}[H]
\centering

\includegraphics[width=0.45\textwidth]{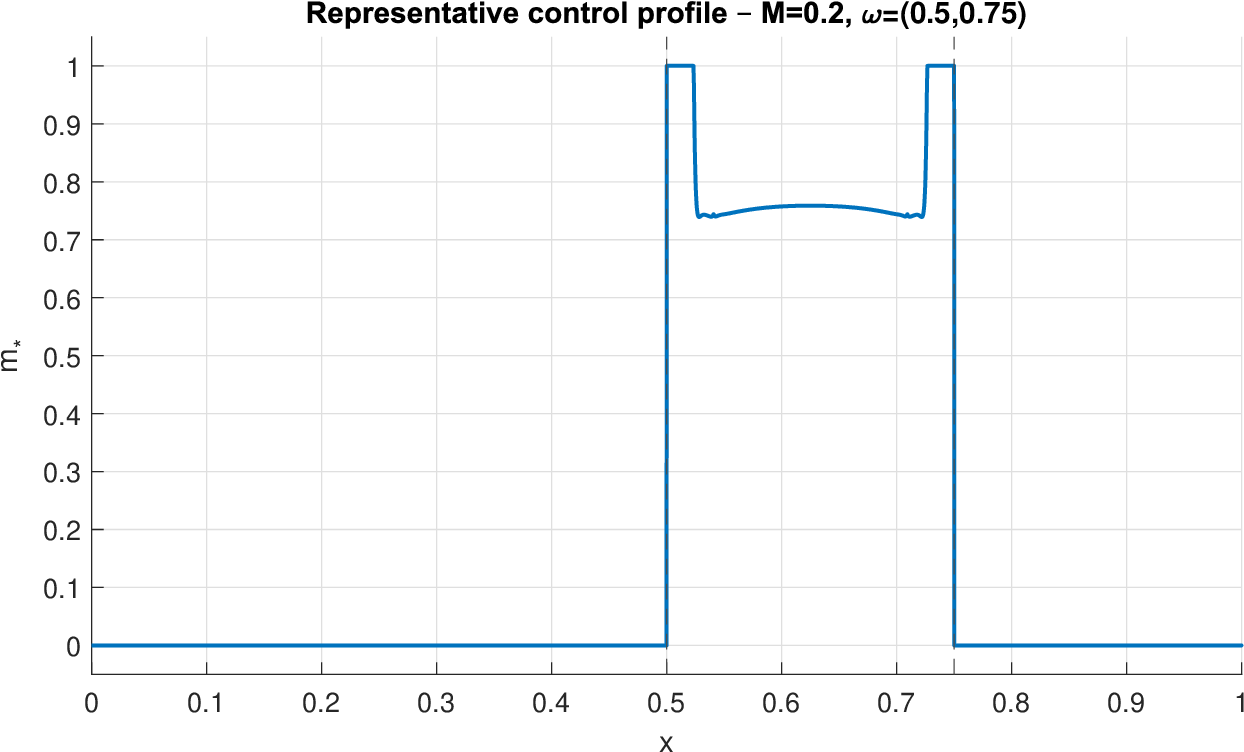}
\includegraphics[width=0.45\textwidth]{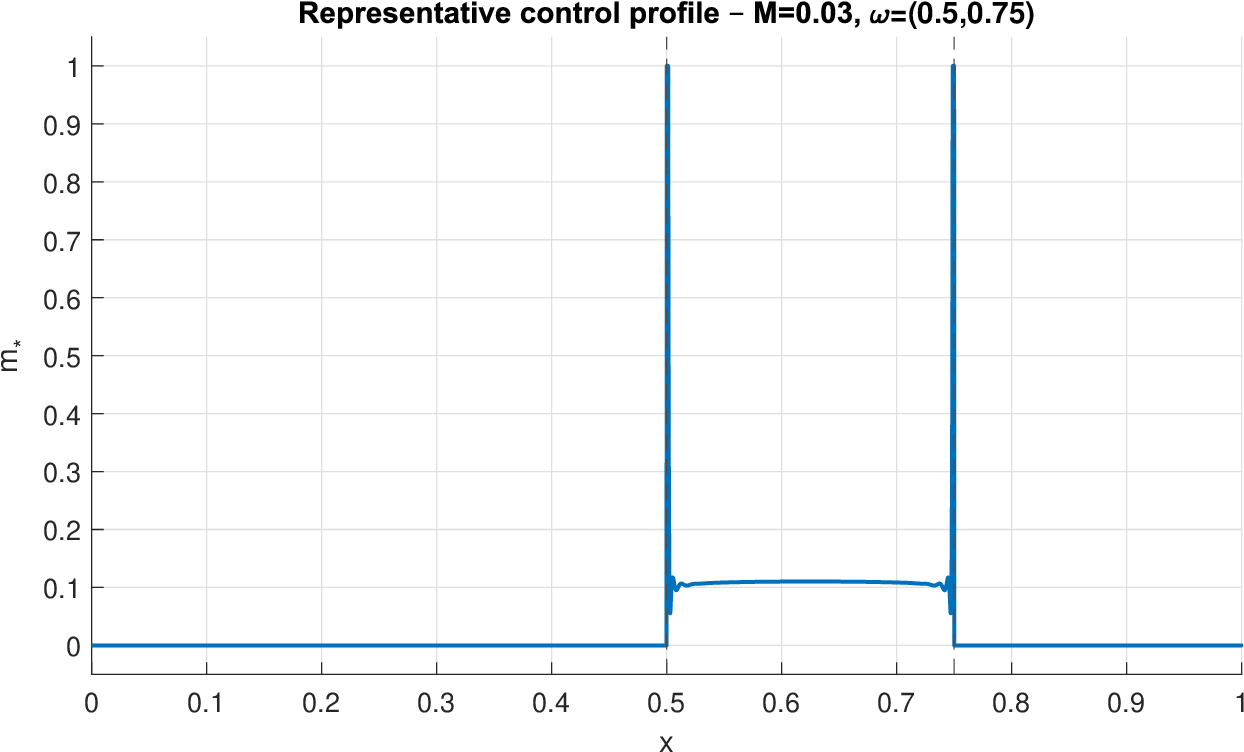}

\includegraphics[width=0.45\textwidth]{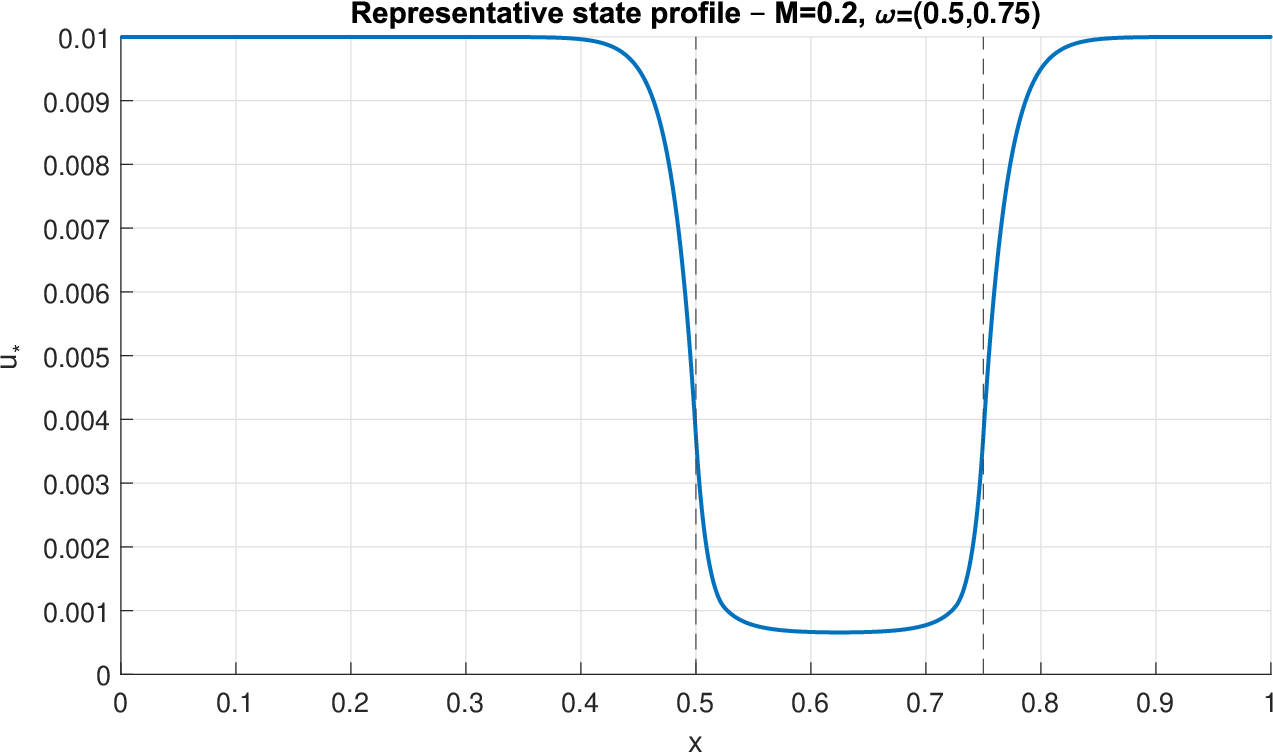}
\includegraphics[width=0.45\textwidth]{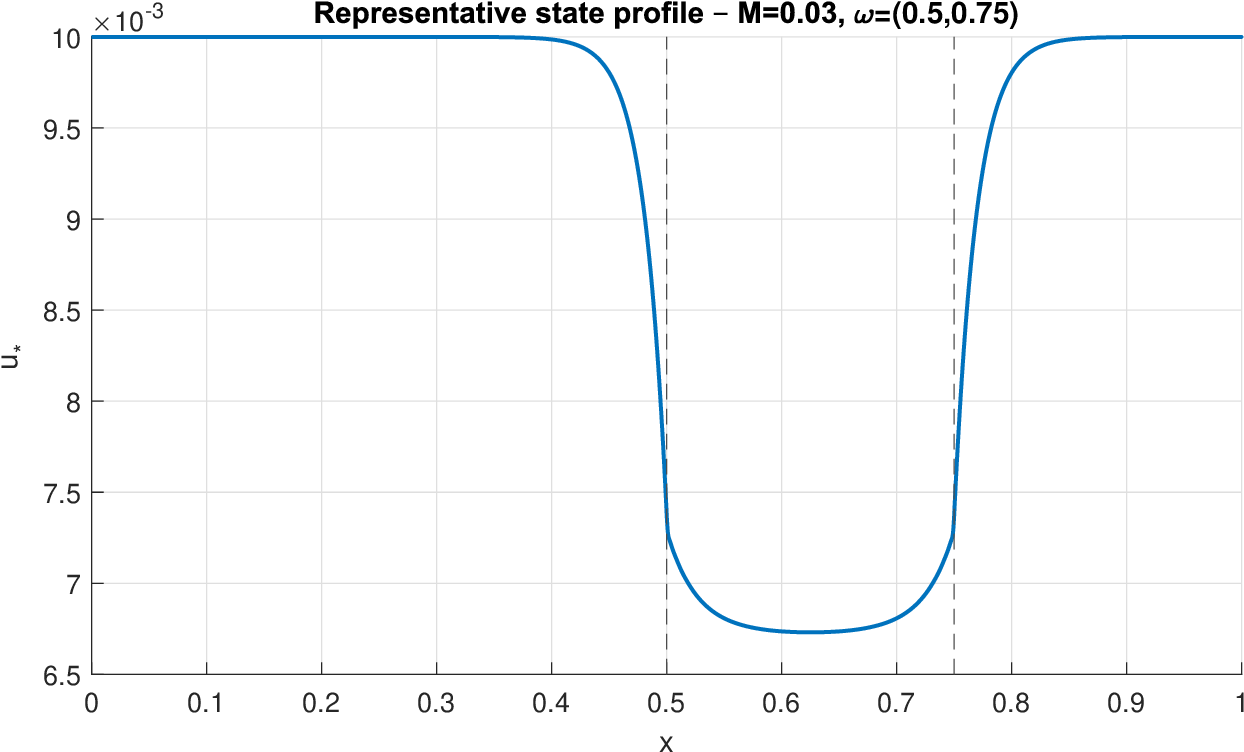}

\includegraphics[width=0.45\textwidth]{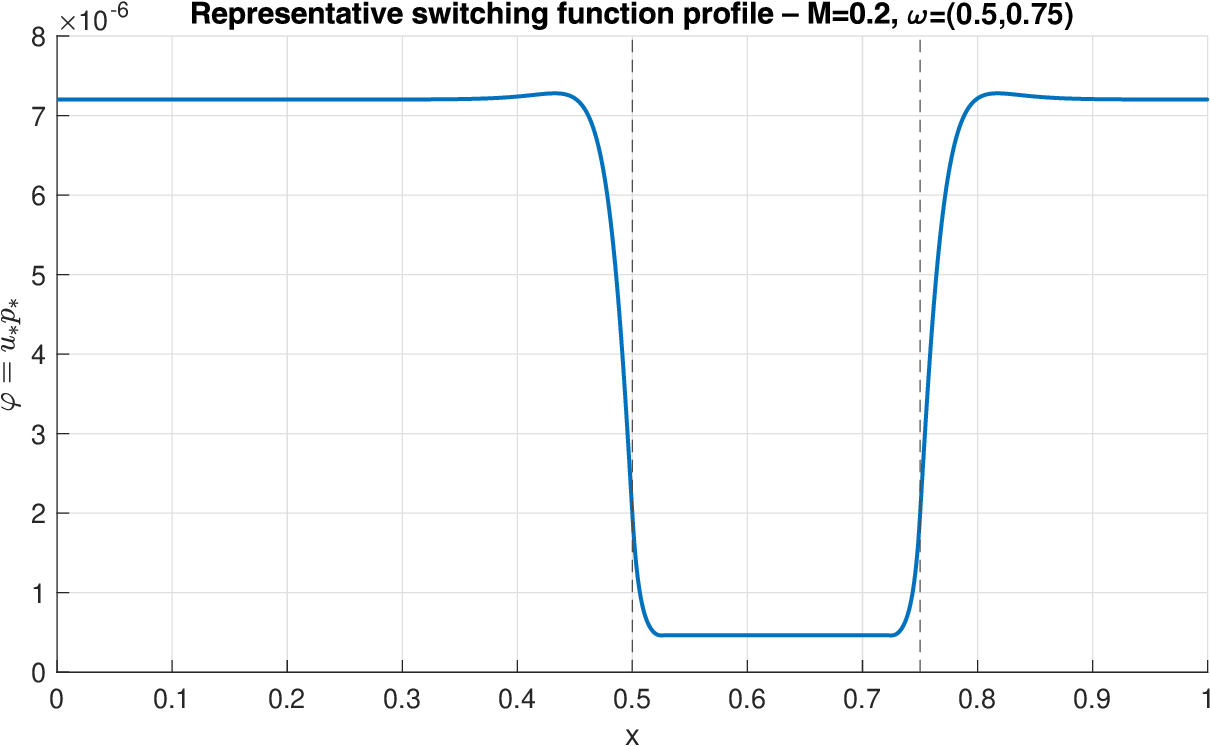}
\includegraphics[width=0.45\textwidth]{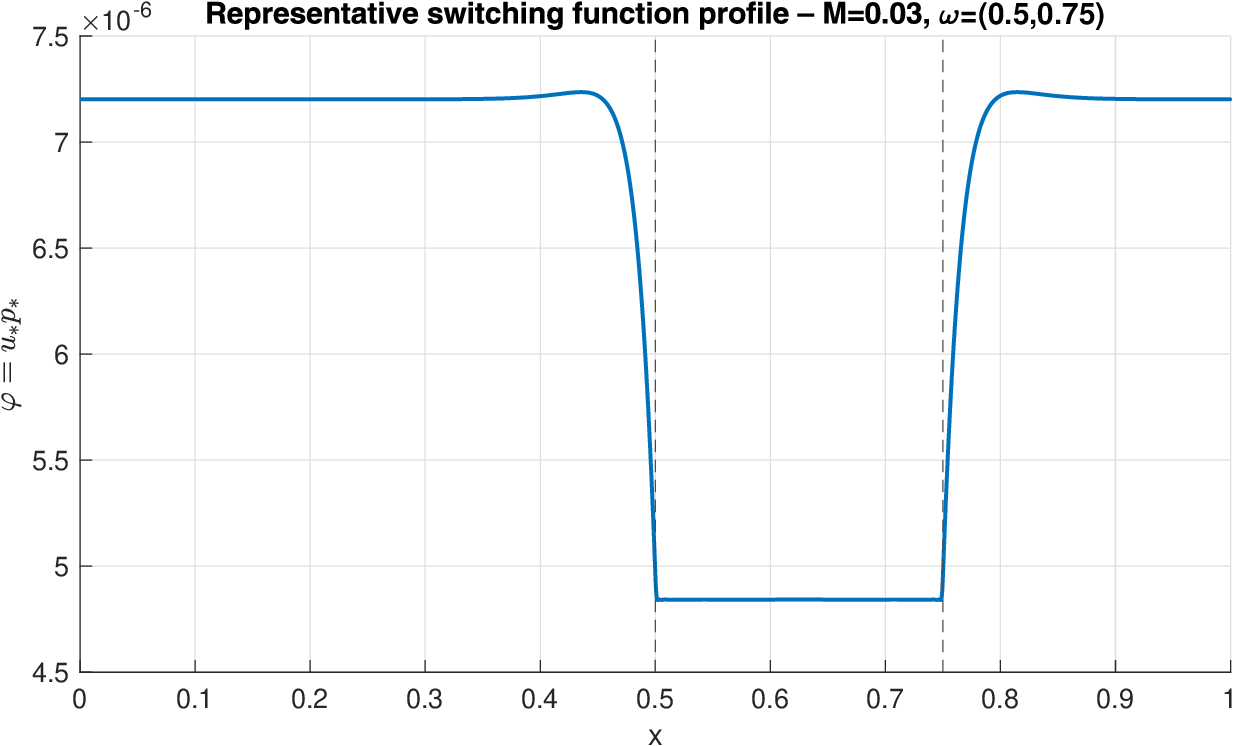}
 
\caption{Profiles associated with the optimal solution in the constant-growth case, for $d=10^{-4}$ and $K=0.01$, under the localized control constraint $m=0$ in $(0,1)\setminus \omega$ with $\omega=\left(\tfrac12,\tfrac34\right)$. The left column corresponds to case (i) $M=0.2$, while the right column corresponds to case (ii) $M=0.03$. The first row shows the optimal control $m_*(x)$, the second row the associated state $u_*(x)$, the third row the switching function \(\varphi(x)=u_*(x)p_*(x)\). The dashed vertical lines mark the boundary of the admissible control region $\omega$.}
\label{fig:gompertz_1D_i_ii}
\end{figure}

These one-dimensional minimization experiments show that the optimal strategy is strongly influenced by the diffusion intensity, the spatial heterogeneity of the growth rate, and the localization of the admissible control region. We now turn to the corresponding maximization problem, where the same structural features are examined under the opposite optimization criterion.

\subsection{Gompertz maximization problem}

Now the goal  is  to maximize the total population size
\[
J(m)=\int_\Omega u_m(x)\,dx
\]
under the same box and mass constraints on the control, namely $m\in \Mcal$.  This provides a complementary perspective on the role of diffusion and spatial heterogeneity: instead of identifying control distributions that reduce the state, we look for admissible configurations that enhance the total population. As before, we fix \(K=0.01\) and \(M=0.5\), and we vary the diffusion coefficient \(d\) in the range \([10^{-8},1]\).

Figure~\ref{fig:s_and_J_max_1D} summarizes the dependence of the one-dimensional
maximization problem on the spatial growth profile and on the diffusion coefficient. The top
left panel shows the growth functions \(s(x)\), while the top right panel reports the
corresponding optimal value \(J(m^*)\) as a function of \(d\). The optimal value is observed to be monotone increasing with respect to \(d\) for all the
considered growth profiles. The bottom panels compare representative optimal controls
\(m^*(x)\) and associated states \(u^*(x)\) at \(d=10^{-4}\). These profiles illustrate that the
bang-bang structure of the maximizer is preserved across the tested growth laws, while the
corresponding state reflects the spatial heterogeneity of \(s(x)\).

\begin{figure}[H]
\centering
\includegraphics[width=0.45\textwidth]{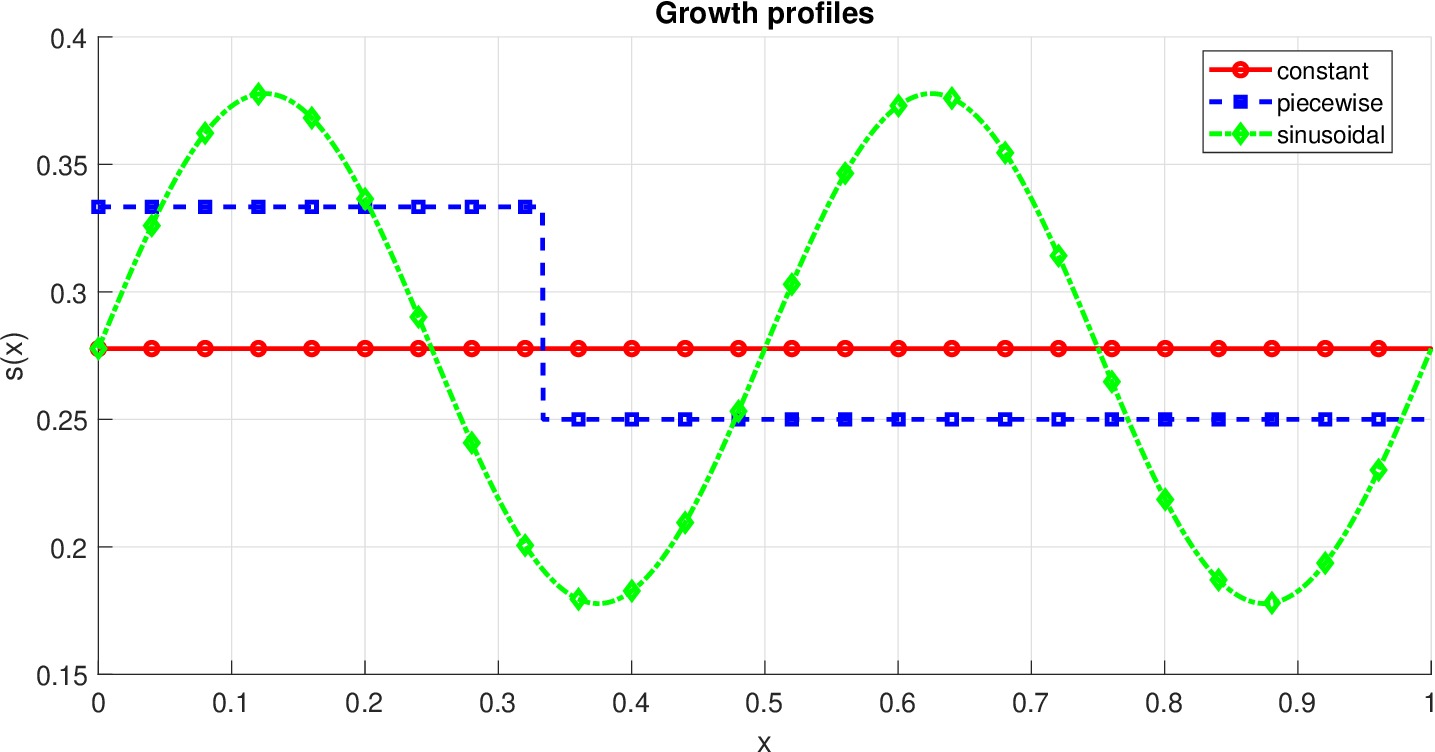} 
\includegraphics[width=0.45\textwidth]{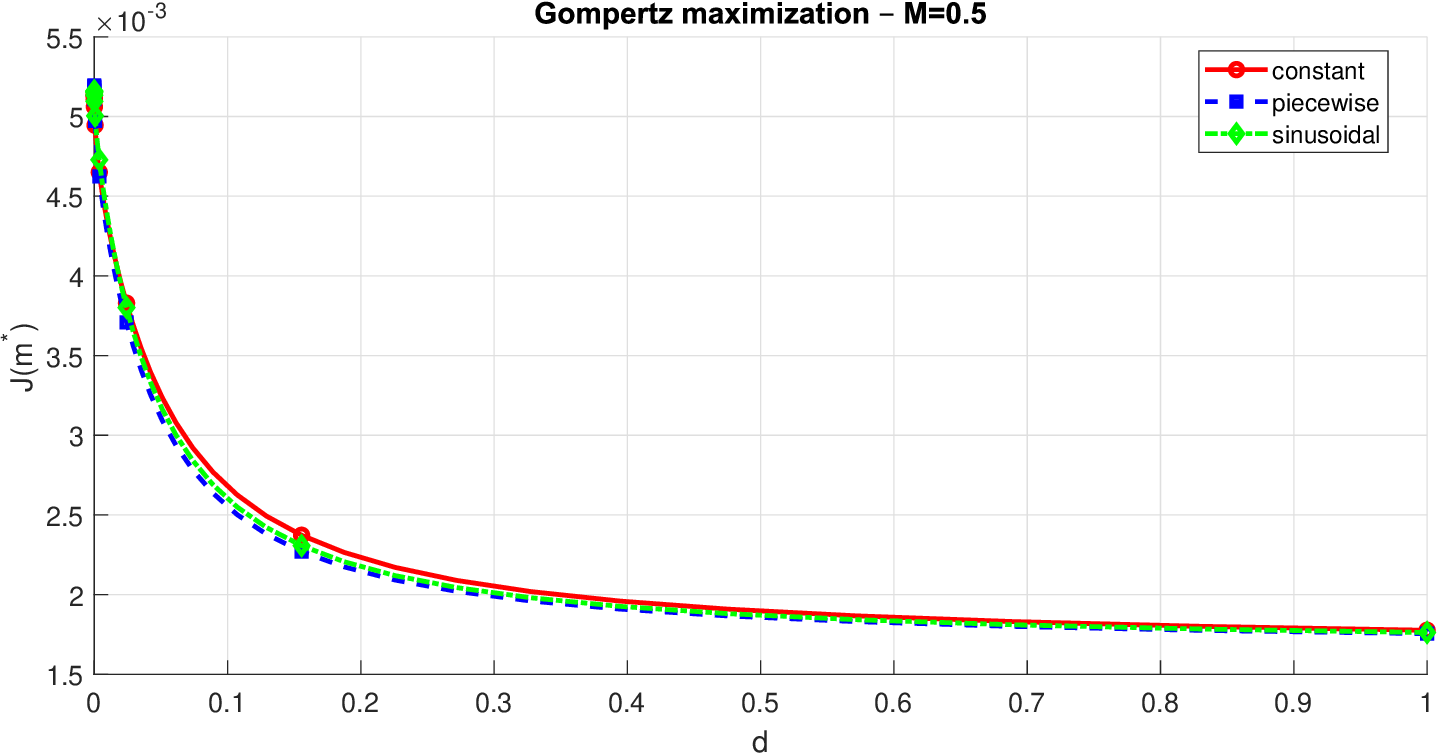}
\includegraphics[width=0.45\textwidth]{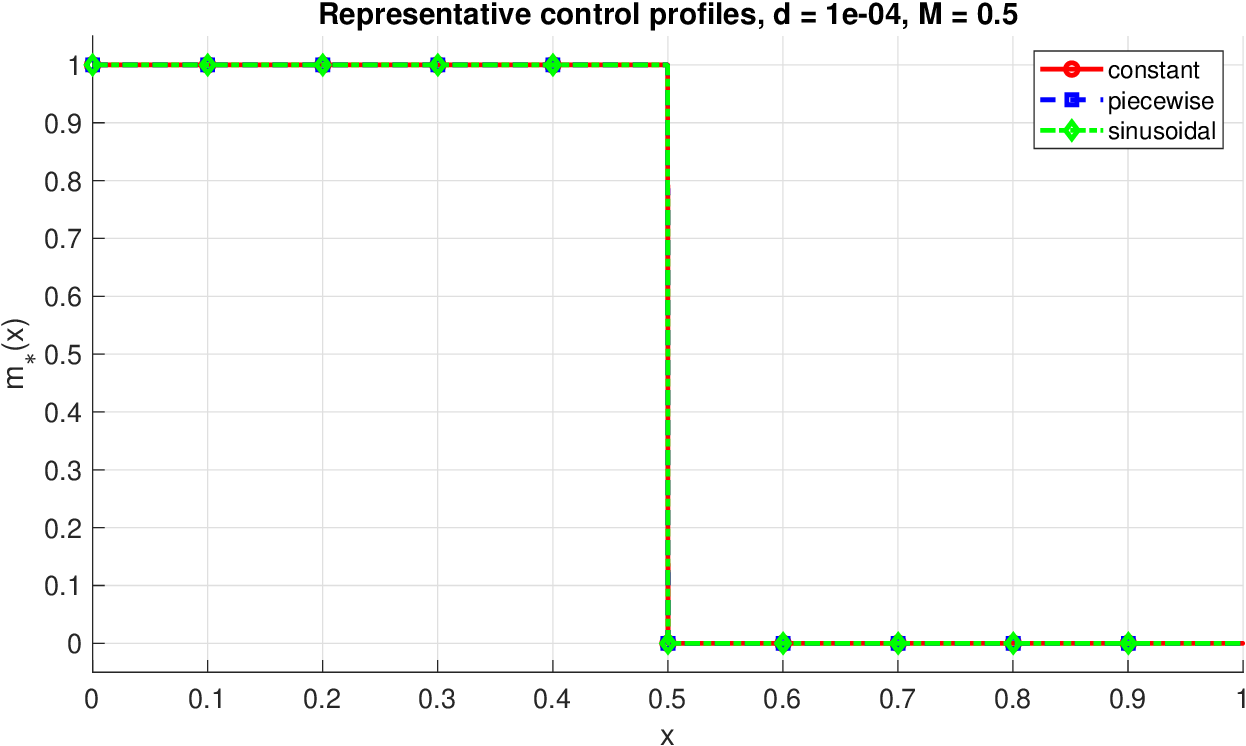}
\includegraphics[width=0.45\textwidth]{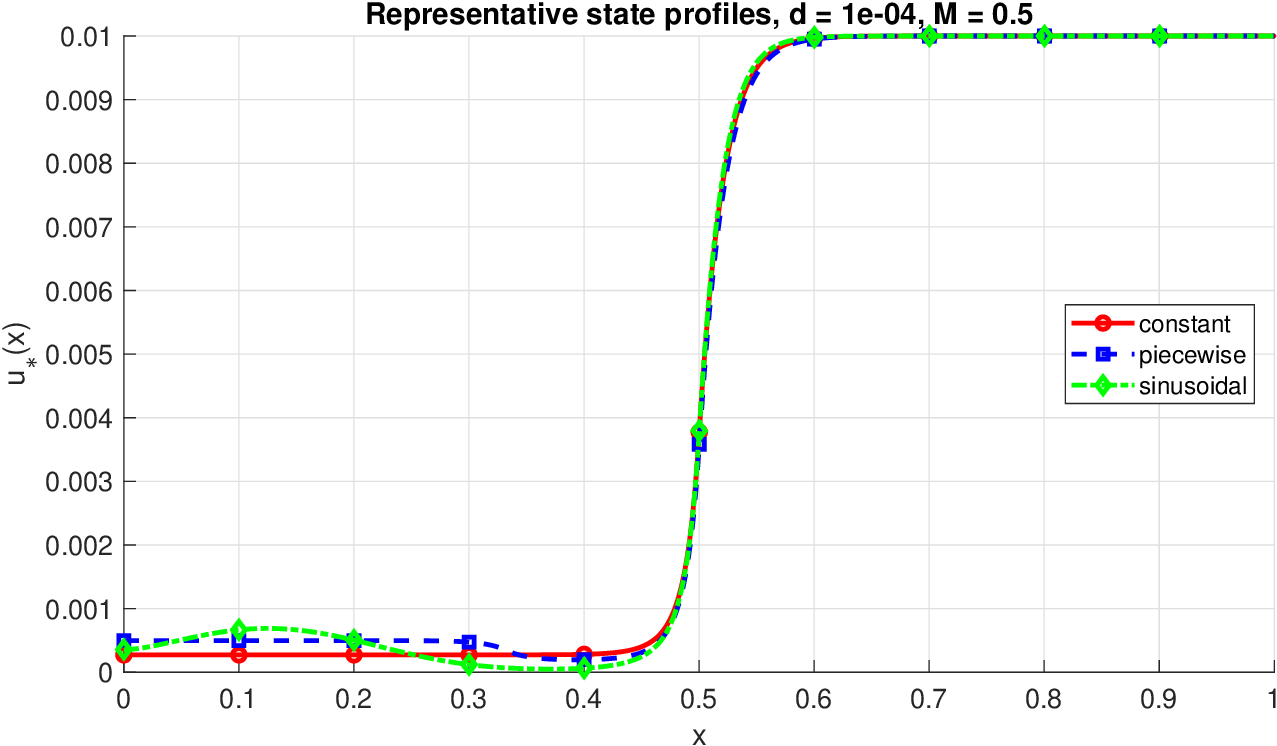}

\caption{One-dimensional stationary Gompertz maximization problem. Top left: spatial growth profiles \(s(x)\) used in the simulations. Top right: optimal value \(J(m^*)\) as a function of the diffusion coefficient \(d\); the horizontal axis is shown on a logarithmic scale for \(d\in[10^{-8},1]\). Bottom left: representative optimal control profiles \(m^*(x)\) at \(d=10^{-4}\). Bottom right: corresponding optimal state profiles \(u^*(x)\) for the considered growth profiles at \(d=10^{-4}\).  The parameters are fixed to \(K=0.01\), \(M=0.5\).}
\label{fig:s_and_J_max_1D}
\end{figure}

Figure~\ref{fig:m_u_profile_max_1D} illustrates the dependence of the optimal pair \((m^*,u^*)\) on the diffusion coefficient \(d\) in the constant-growth case. The upper panels provide a global view through heatmaps in the \((x,d)\)-plane, whereas the lower panels show representative one-dimensional profiles for selected values of \(d\). For the parameters \(K=0.01\) and \(M=0.5\), the computed optimal controls exhibit a bang-bang configuration supported on a left span of the interval. The associated state profiles reflect the smoothing effect of diffusion: as \(d\) increases, the spatial variation of \(u^*\) becomes less pronounced.

\begin{figure}[H]
\centering

\includegraphics[width=0.45\textwidth]{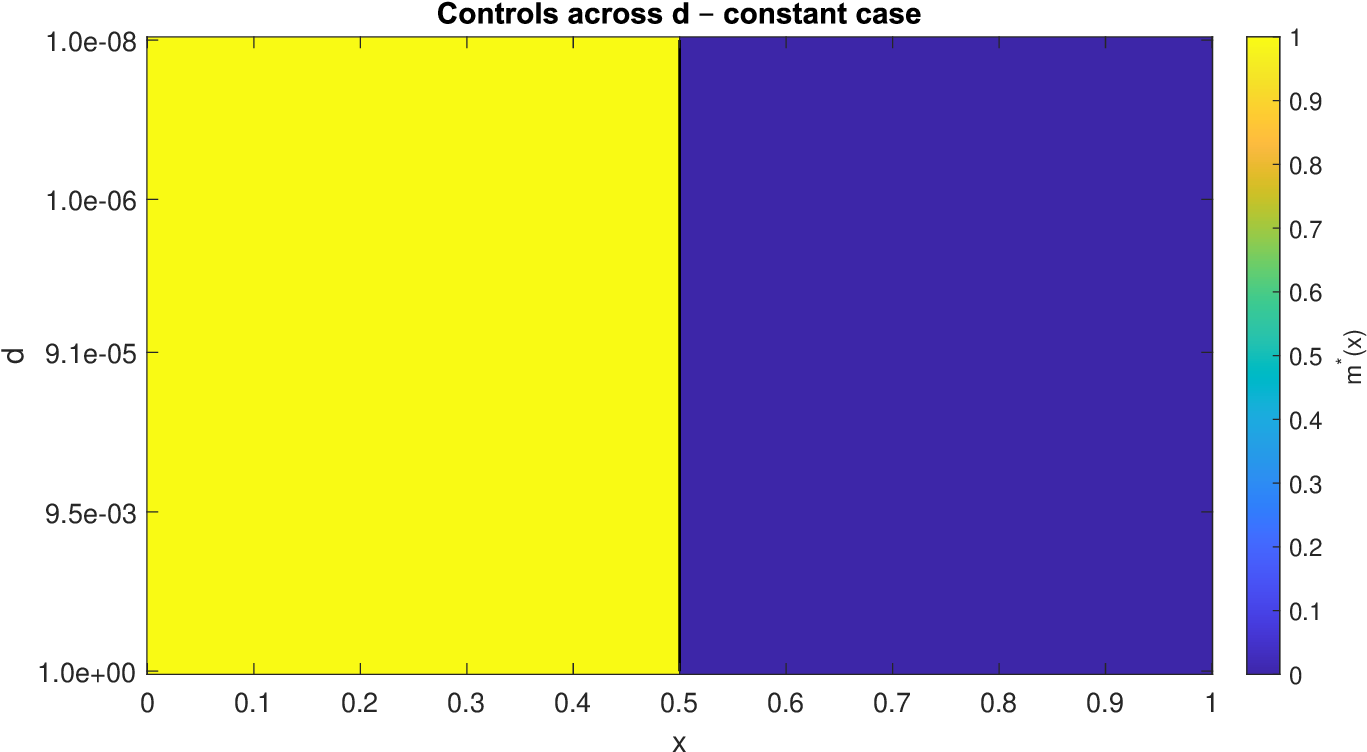} 
\includegraphics[width=0.45\textwidth]{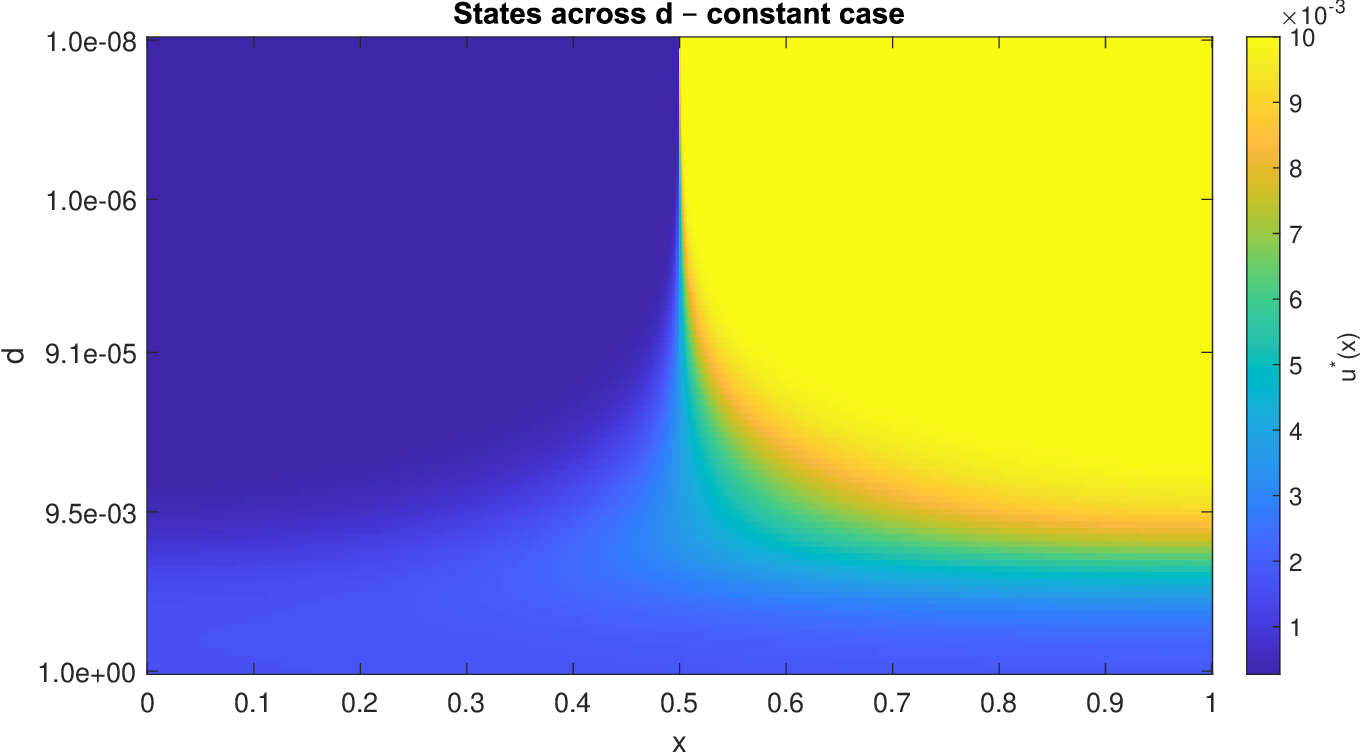}

\includegraphics[width=0.45\textwidth]{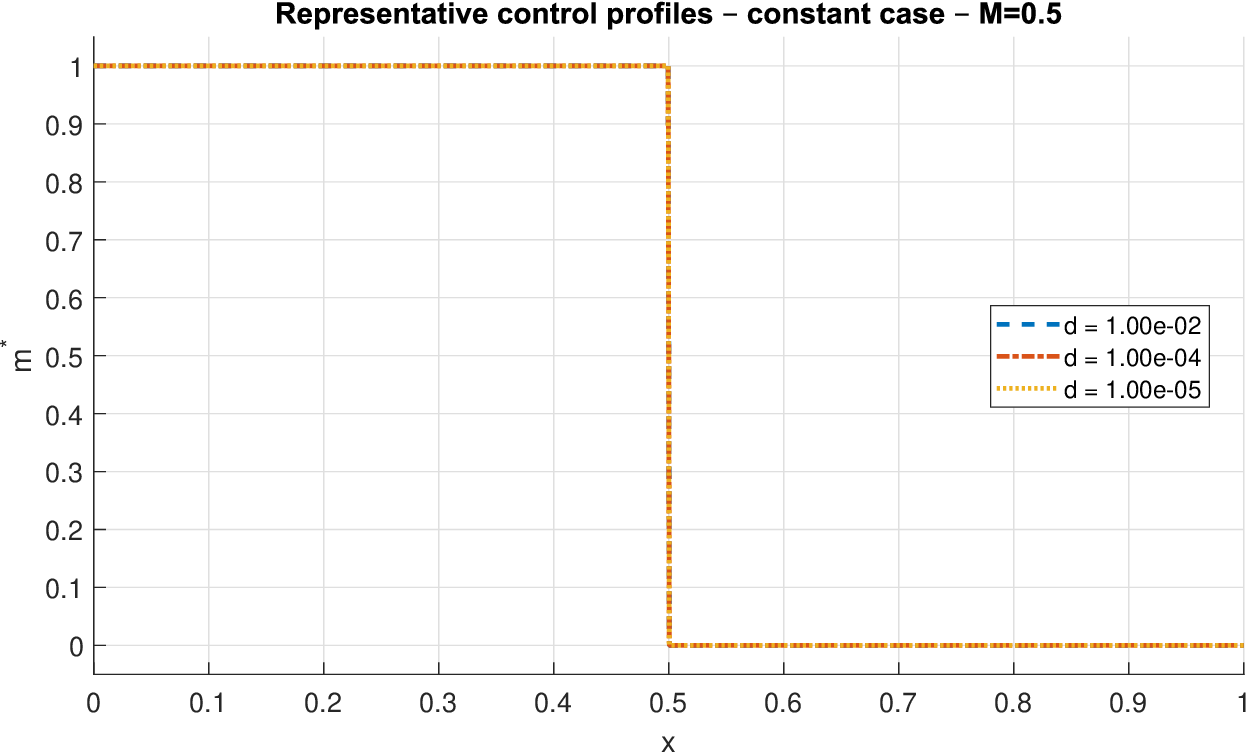} 
\includegraphics[width=0.45\textwidth]{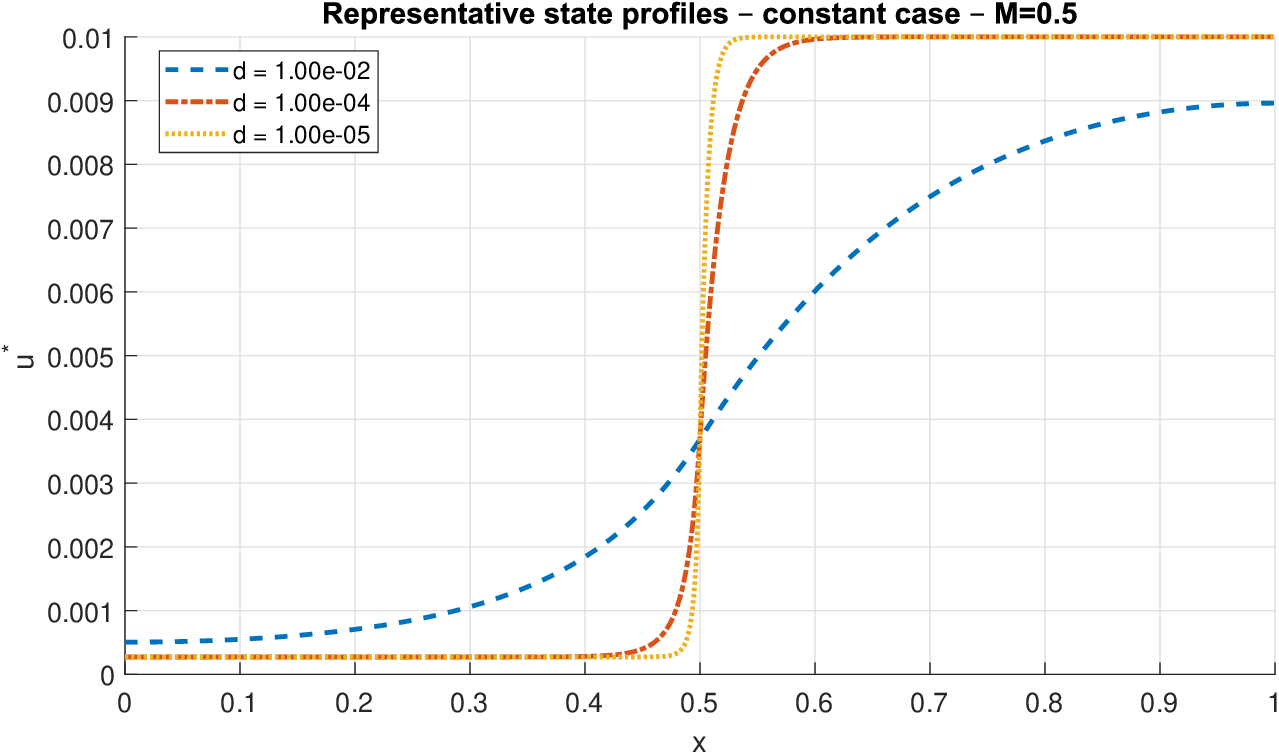}

\caption{Numerical results for the maximization problem in the constant growth case. Top left: heatmap of the optimal control $m^*(x)$ as a function of the spatial variable $x$ and the diffusion coefficient $d$. Top right: heatmap of the corresponding optimal state $u^*(x)$. Bottom left: representative optimal control profiles $m^*(x)$ for selected values of $d$. Bottom right: representative optimal state profiles $u^*(x)$ associated with the corresponding optimal controls. In the examples displayed here, the optimal control preserves a bang-bang structure supported on a left interval of prescribed measure ($M$), while the corresponding state profile varies with the diffusion coefficient ($d$). The parameters are fixed to $K=0.01$, $M=0.5$, and $d\in[10^{-8},1]$.}
\label{fig:m_u_profile_max_1D}
\end{figure}

Figure~\ref{fig:m_u_profile_max_1D_centered_rectangle} reports the numerical results for the maximization problem when the growth function is a centered rectangular profile. 
In this case, \(s(x)\) reaches its maximal value on an interval centered at \(x=1/2\) and takes a smaller constant value near the boundary, while preserving the prescribed spatial average. 
Differently from the previous growth profiles, the location of the active region depends strongly on the diffusion regime. 
For larger values of \(d\), the selected maximizer is concentrated near the boundary, consistently with the boundary-concentration behavior observed in the constant-growth case. 
For smaller values of \(d\), a second bang-bang branch, concentrated in the central high-growth region, gives larger values of the objective. 
For extremely small values of \(d\), the numerical continuation of this central branch becomes delicate; therefore, the apparent return to a boundary-concentrated configuration is not used to draw a definitive conclusion about the limiting small-diffusion regime. 
Overall, the experiment suggests a competition between boundary-concentrated and centrally concentrated bang-bang configurations, with a strong dependence of the active-set location on the diffusion coefficient.
\begin{figure}[H]
\centering

\includegraphics[width=0.45\textwidth]{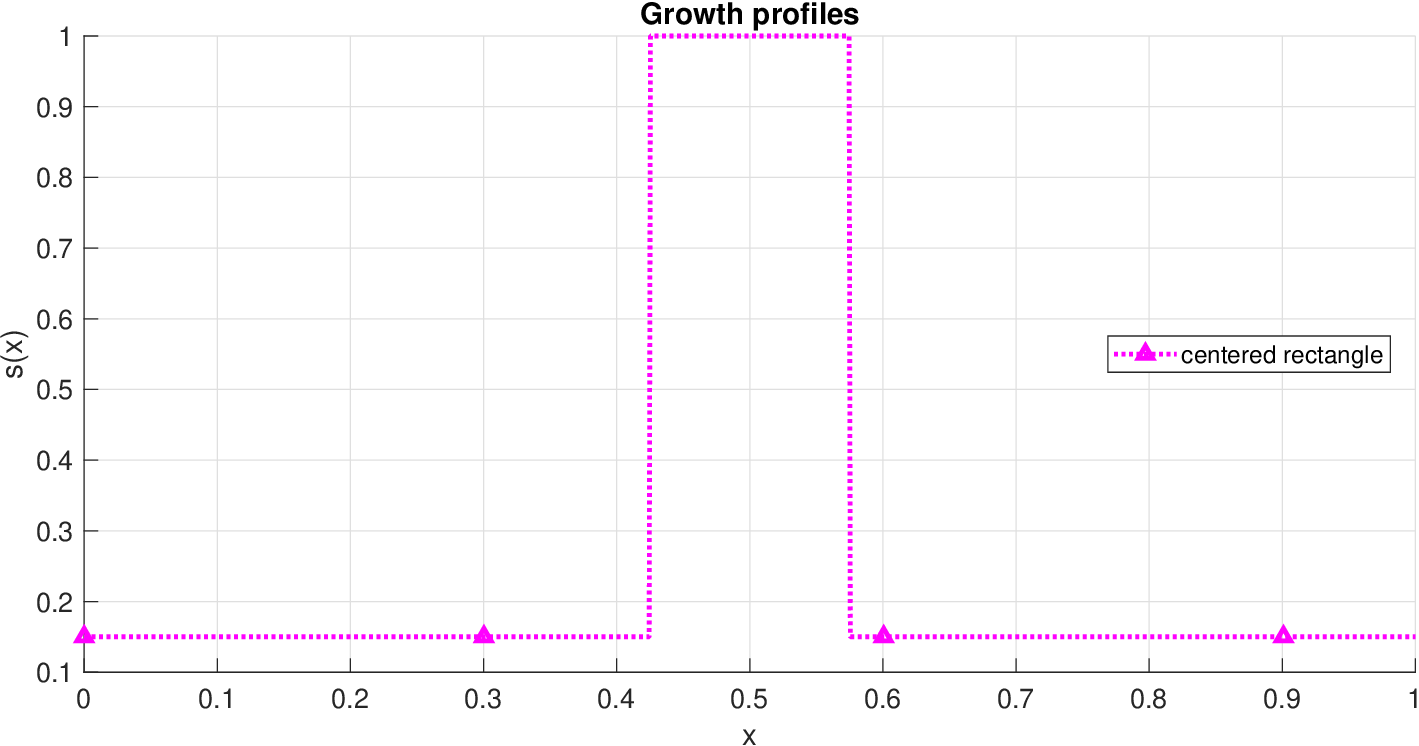} 
\includegraphics[width=0.45\textwidth]{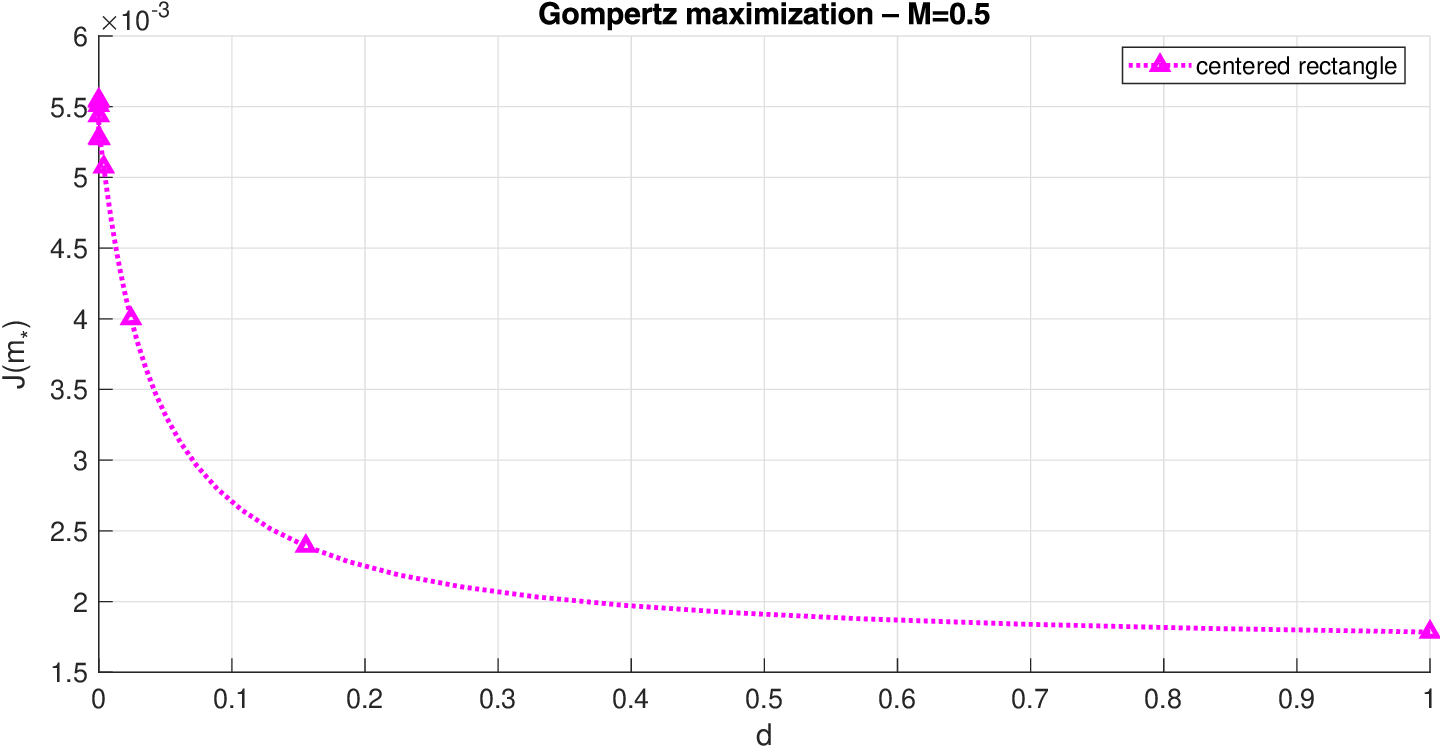}

\includegraphics[width=0.45\textwidth]{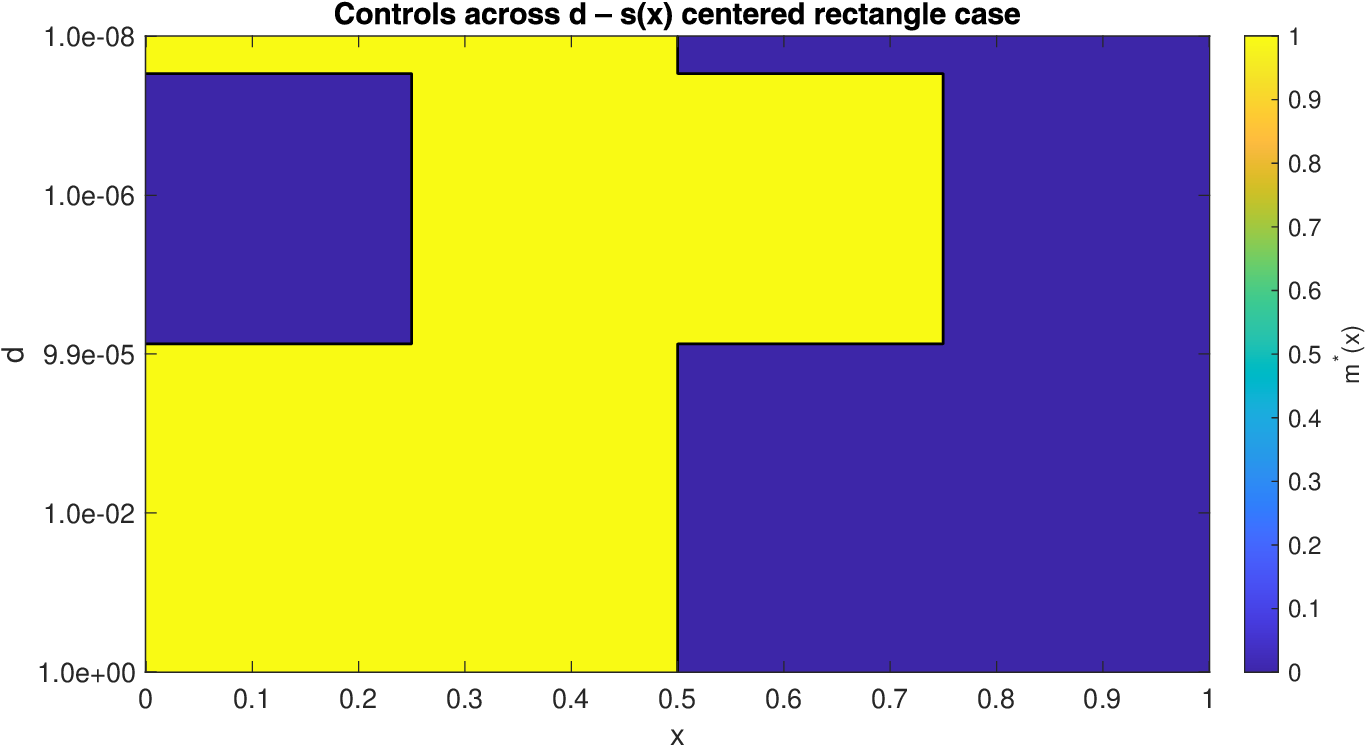} 
\includegraphics[width=0.45\textwidth]{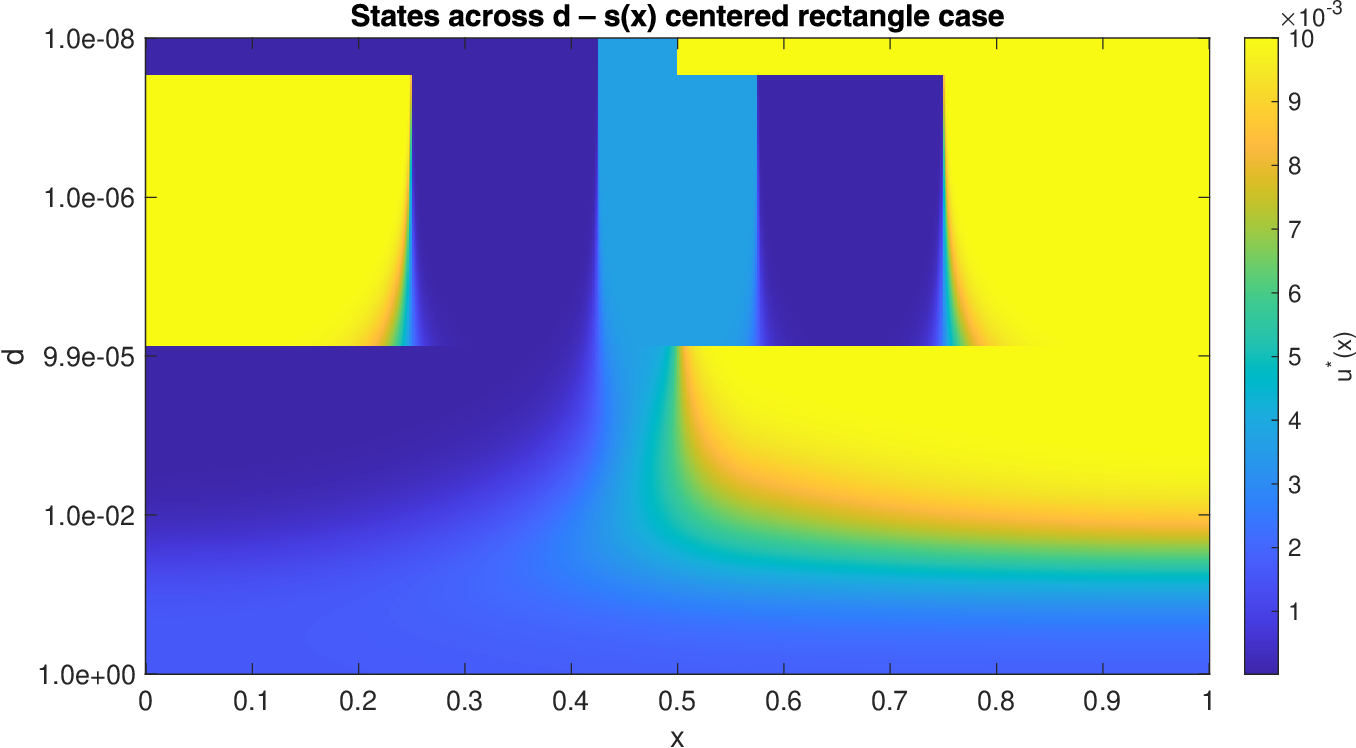}

\includegraphics[width=0.45\textwidth]{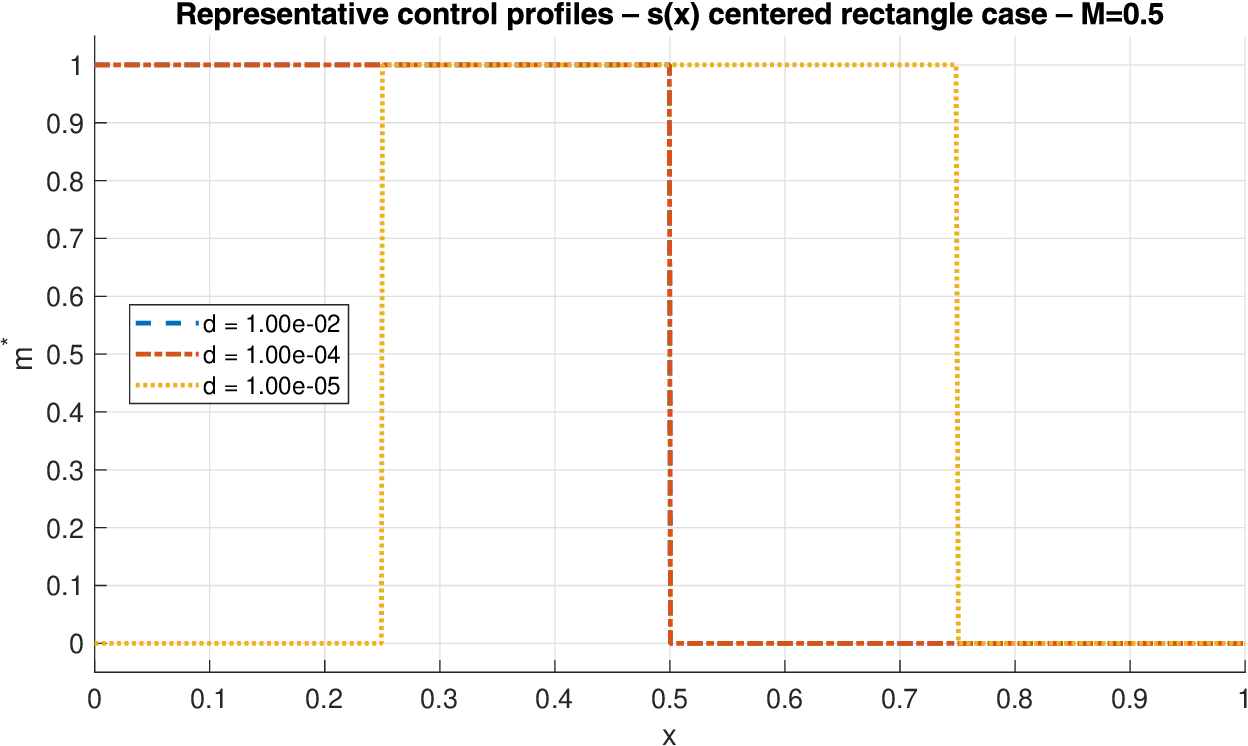} 
\includegraphics[width=0.45\textwidth]{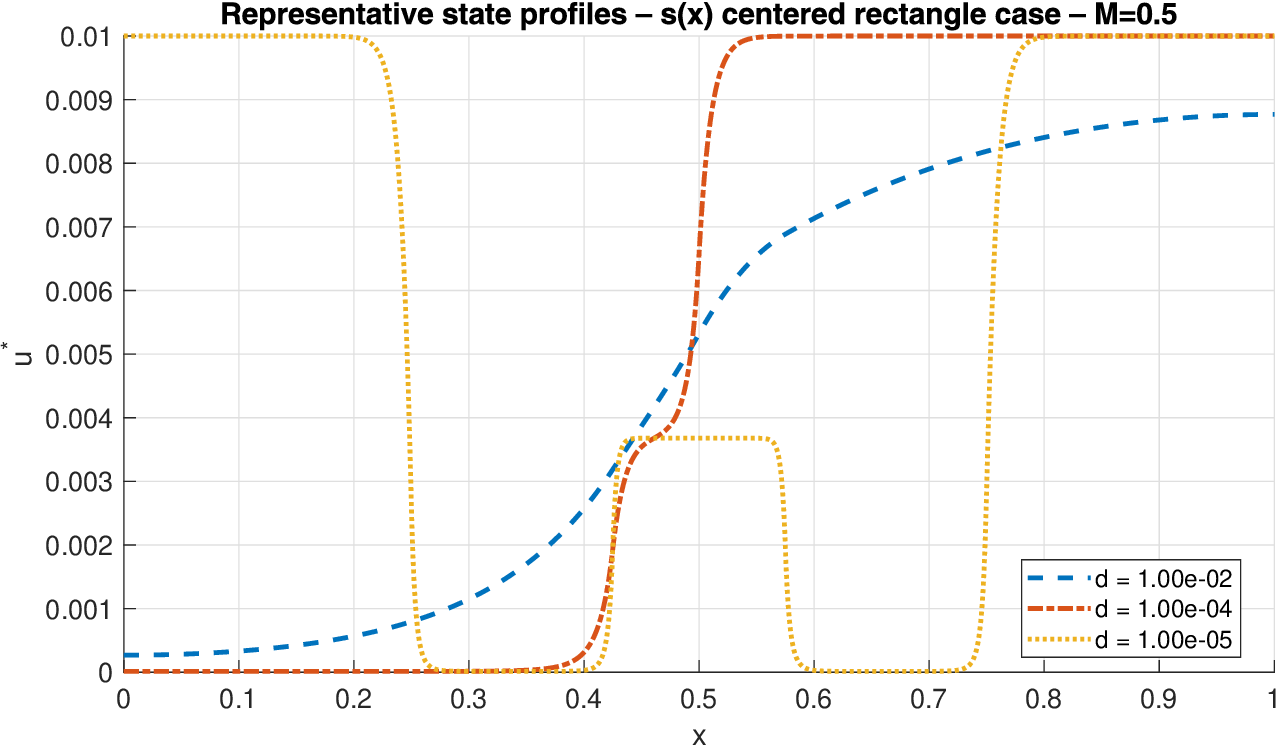}

\caption{Numerical results for the maximization problem in the centered-rectangle growth case. 
Top left: spatial growth profile \(s(x)\), consisting of a rectangular high-growth region centered at \(x=1/2\) and lower values near the boundary, normalized so that its spatial average is \(5/18\). 
Top right: computed value \(J(m^*)\) as a function of the diffusion coefficient \(d\). 
Middle left: heatmap of the computed optimal control \(m^*(x)\) as a function of \(x\) and \(d\). 
Middle right: heatmap of the corresponding state \(u^*(x)\). 
Bottom left: representative control profiles \(m^*(x)\) for selected values of \(d\). 
Bottom right: representative state profiles \(u^*(x)\). 
The parameters are fixed to \(K=0.01\), \(M=0.5\), and \(d\in[10^{-8},1]\). }
\label{fig:m_u_profile_max_1D_centered_rectangle}
\end{figure}

\begin{remark}
This behavior should be compared with the logistic case, where numerical and analytical results show a transition from concentrated optimizers for large diffusion to fragmented optimizers for small diffusion; see \cite[Figure~5, p.~32]{manaprjmpa}. In the present Gompertz simulations we do not observe the same fragmentation phenomenon. Rather, the computed maximizers seem to remain concentrated, while the location of the active region may change with the diffusion coefficient. This suggests that the small-diffusion regime for Gompertz growth is governed by a different mechanism, whose rigorous characterization remains open.
\end{remark}

We finally consider a two-dimensional version of the maximization problem on the square domain
\(\Omega=(0,1)^2\). The growth rate is taken to be constant, \(s(x,y)=5/18\), and the parameters are fixed to \(K=0.01\), \(M=0.2\), and \(d=10^{-4}\). This example illustrates that the concentration phenomena observed in the one-dimensional simulations persist in two space dimensions. Several initial configurations were tested in the numerical optimization, including uniform, strip-like, quarter-disk, and square-type initial controls. Among the computed candidates, the one displayed in Figure~\ref{fig:gompertz_2D_max} corresponds to the largest value of the objective functional.

\begin{figure}[H]
\centering

\includegraphics[width=0.45\textwidth]{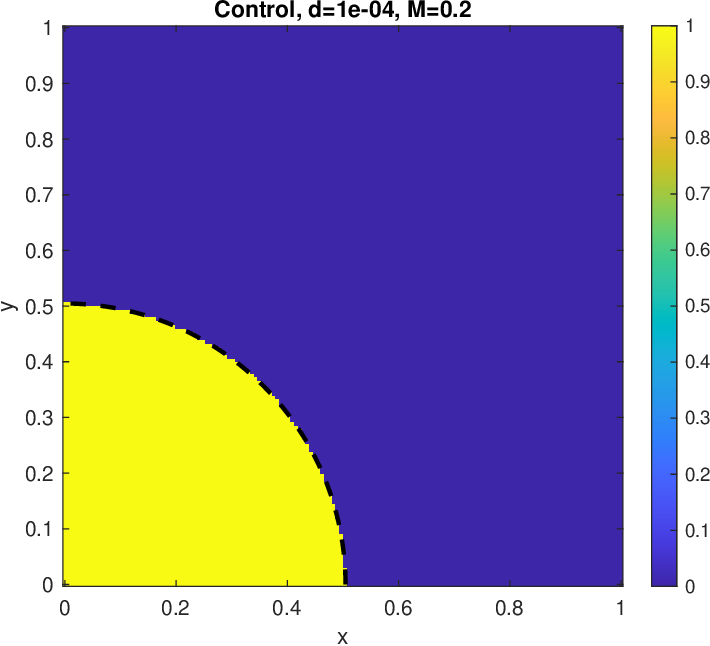}  
\hspace{1cm}
\includegraphics[width=0.45\textwidth]{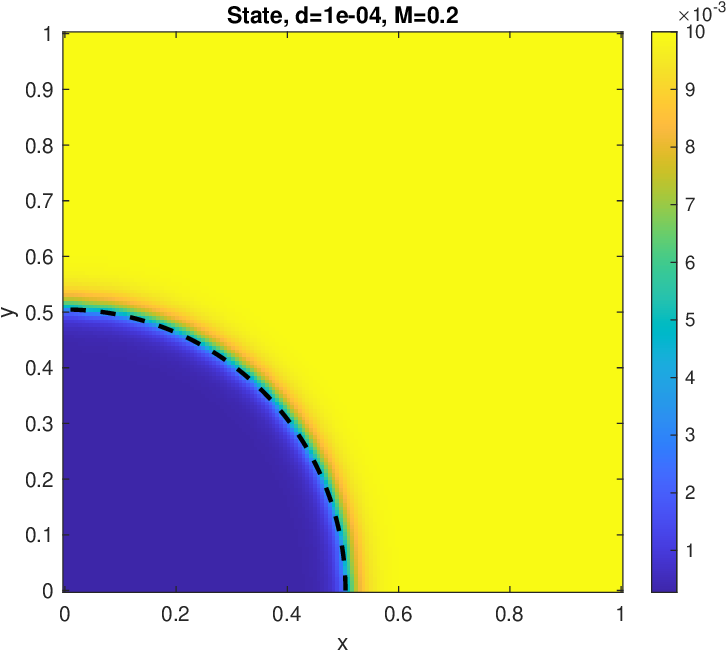} 

\caption{Two-dimensional numerical result for the Gompertz maximization problem. 
Left: optimal control \(m^*(x,y)\). Right: corresponding optimal state \(u^*(x,y)\). 
The parameters are \(s(x,y)=5/18\), \(K=0.01\), \(M=0.2\), and \(d=10^{-4}\). 
The dashed curve marks the boundary of the quarter-disk initial configuration used in one of the tested starts, and is included as a reference for the shape of the computed active set.}
\label{fig:gompertz_2D_max}
\end{figure}

This experiment is consistent with the bang-bang property of Theorem~\ref{thm:mbang}; the quarter-disk-like shape of the active set, however, is a numerical feature and is not covered by the one-dimensional large-diffusion characterization of Theorem~\ref{teo-concentration}. As shown in Figure~\ref{fig:gompertz_2D_max}, the computed control is bang-bang and occupies a quarter-disk-like region near the lower-left corner of the domain. The corresponding state is strongly depleted inside the controlled region and increases sharply across the interface, approaching values close to the carrying capacity \(K=0.01\) away from the active set. Thus, also in two space dimensions, the optimal strategy concentrates the treatment on a localized region; in this case, however, the active set is two-dimensional rather than an interval.

\section*{Acknowledgements}
IMB has been supported by INdAM - GNCS Project, CUP E53C25002010001. This research has been accomplished within RITA (Research ITalian network on Approximation) and the UMI Group TAA (Approximation Theory and Applications).
BP has been partially supported by
the Portuguese government through FCT Fundação para a Ciência e a Tecnologia under the project 2024.14494.PEX with DOI identifier 10.54499/2024.14494.PEX (project ASSO),  by the INdAM - Gnampa project ``Metodi variazionali e topologici tra Fisica, Geometria e Scienze Applicate'' (CUP E53C25002010001).
FG has been partially supported by  Next Generation EU-CUP-J55F21004240001,  DM 737-2021, risorse $2022-2023$ and  by the INdAM - Gnampa project  INdAM - GNAMPA Project,  CUP E53C25002010001. 

\end{document}